\documentclass[12pt]{article}

\usepackage{fullpage,amsfonts,amsmath,amsthm,amssymb,mathrsfs,graphicx,upgreek}
\usepackage[normalem]{ulem}
\usepackage{verbatim,url}
\usepackage{graphicx,hyperref,enumitem}
\usepackage[noadjust,sort]{cite}
\usepackage{tocloft}

\usepackage[margin=1in]{geometry}

\usepackage{tikz}
\usetikzlibrary{positioning}

\newcommand{\be}{\begin{equation}}
\newcommand{\ee}{\end{equation}}
\newcommand{\remove}[1]{}

\newtheorem{theorem}{Theorem}[section]
\newtheorem{corollary}[theorem]{Corollary}
\newtheorem{Conjecture}[theorem]{Conjecture}
\newtheorem{lemma}[theorem]{Lemma}
\newtheorem*{Definition}{Definition}
\newtheorem{proposition}[theorem]{Proposition}

\theoremstyle{definition}
\newtheorem{Remark}[theorem]{Remark}
\newtheorem{question}[theorem]{Question}

\def\Coupling{\textit{Coupling}}
\def\IndSample{\textit{IndSample}}

\def\loww{{G_\zeta}}
\def\upp{{G_0}}
\def\mlow#1{M^{(#1)}_\zeta}
\def\mup#1{M^{(#1)}_0}
\def\mreg#1{G^{(#1)}}

\def\Time{\mathcal{I}}

\def\G{{\mathcal G}}
\def\po{\operatorname{\mathbf Po}}

\def\Pr{\operatorname{\mathbb P}}
\def\diam{\operatorname{diam}}
\def\range#1{\operatorname{rng}(#1)}

\newcommand{\norm}[1]{{\left\|#1\right\|}}
\def\Bin{\operatorname{\bf Bin}}   % BDM
\def\bfd{{\dvec}}
\def\Reg{{G}}

\def\time{{\iota}}

\def\low{{G^{L}}}
\def\up{{G^U}}

\def\la{\lambda}

\def\eps{\varepsilon}
\def\ss{\smallskip}

\def\no{\noindent}

\def\abs#1{\lvert#1\rvert} 
 
\def\norm#1{\lVert#1\rVert}
\def\maxnorm#1{\norm{#1}_{\mathrm{max}}}

\def\dmax{d_{\mathrm{max}}}
\def\dmin{d_{\mathrm{min}}}

\def\dfrac#1#2{\lower0.15ex\hbox{\large$\textstyle\frac{#1}{#2}$}}
\def\({\bigl(}
\def\){\bigr)}
\def\st{\mathrel{:}}

\def\tr{\operatorname{tr}}
\def\nicebreak{\vskip 0pt plus 50pt\penalty-300\vskip 0pt plus -50pt }

\def\X{\boldsymbol{X}}
\def\Y{\boldsymbol{Y}}

\def\calB{\mathcal{B}}

\def\calB{\mathcal{B}}
\def\thetavec{\boldsymbol{\theta}}
\def\avec{\boldsymbol{a}}
\def\qvec{\boldsymbol{q}}
\def\dvec{\boldsymbol{d}}
\def\evec{\boldsymbol{e}}
\def\vvec{\boldsymbol{v}}
\def\tvec{\boldsymbol{t}}
\def\hvec{\boldsymbol{h}}
\def\svec{\boldsymbol{s}}
\def\rvec{\boldsymbol{r}}

\def\xvec{\boldsymbol{x}}
\def\yvec{\boldsymbol{y}}

\def\gvec{\boldsymbol{g}}

\def\betavec{\boldsymbol{\beta}}

\def\trans{^{\mathrm{T}}\!}

\def\E{\operatorname{\mathbb{E}}}
\def\V{\operatorname{\mathbb{V\!}}}

\def\Var{\operatorname{Var}}
\def\Cov{\operatorname{Cov}}

\def\Reals{{\mathbb{R}}}
\def\Complexes{{\mathbb{C}}}

\def\Naturals{{\mathbb{N}}}

\def\dmin{d_{\min}}
\def\mumin{\mu_{\min}}
\let\le=\leqslant
\let\leq=\leqslant
\let\ge=\geqslant
\let\geq=\geqslant

\numberwithin{equation}{section}

\date{}

\title{Sandwiching random regular graphs between binomial random graphs}
\author{
Pu Gao\thanks{Research  {\vrule height 2.5ex width 0ex} supported by ARC DE170100716, ARC DP160100835 and NSERC.} \\
University of Waterloo\\
\tt pu.gao@uwaterloo.ca \and 
Mikhail Isaev \thanks{Research supported by ARC DE200101045.}\ \footnotemark[3]\\
Monash University\\
\tt mikhail.isaev@monash.edu \and
Brendan D. McKay\thanks{Research supported by ARC DP170103687.}\\
Australian National University\\
\tt brendan.mckay@anu.edu.au}

\date{}

\begin{document}

\maketitle
\begin{abstract}
Kim and Vu made the following conjecture (\textit{Advances in Mathematics}, 2004): if $d\gg \log n$, then the random $d$-regular graph $\G(n,d)$ can asymptotically almost surely be ``sandwiched'' between $\G(n,p_1)$ and $\G(n,p_2)$ where $p_1$ and $p_2$ are both $(1+o(1))d/n$. They proved this conjecture for $\log n\ll d\le n^{1/3-o(1)}$, with a defect in the sandwiching: $\G(n,d)$ contains $\G(n,p_1)$ perfectly, but is not completely contained in $\G(n,p_2)$. Recently, the embedding $\G(n,p_1) \subseteq \G(n,d)$ was improved by Dudek,  Frieze,  Ruci\'{n}ski and \v{S}ileikis to $d=o(n)$.
In this paper, we prove Kim--Vu's sandwich conjecture, with perfect containment on both sides, for all $d\gg n/\sqrt{\log n}$. For $d=O(n/\sqrt{\log n})$, we prove a weaker version of the sandwich conjecture with $p_2$ approximately equal to $(d/n)\log n$, without any defect.
In addition to sandwiching regular graphs, our results  cover 
graphs whose degrees are asymptotically equal.  The proofs rely on estimates for the probability
that a random factor of a pseudorandom graph contains a given edge, which is of independent interest.

As applications, we obtain new results on the properties of random graphs with given near-regular degree sequences, including Hamiltonicity and universality in subgraph containment. We also determine several graph parameters in these random graphs, such as the chromatic number, small subgraph counts, the diameter, and the independence number. We are also able to characterise many phase transitions in edge percolation on these random graphs, such as the threshold for the appearance of a giant component.

\end{abstract}

\section{Introduction}

Random graph theory is one of the most important subjects in modern graph theory. Besides the rich theory in its own field of study, random graphs have many connections and applications in the general area of combinatorics. Many existence results in graph theory are proved by using and modifying random graphs. Today, random graphs are widely used in computer science, engineering, physics and other branches of sciences. 

There are many random graph models.  The most classical models $\G(n,p)$ and $\G(n,m)$ were introduced by Erd\H{o}s and R\'enyi~\cite{ER,ER2} more than half a century ago. The binomial model $\G(n,p)$ retains each 
edge of the complete graph $K_n$ independently with probability $p$. The uniform model $\G(n,m)$ is simply $\G(n,p)$ conditioned on having exactly $m$ edges. In other words, $\G(n,m)$ is the random graph on $n$ vertices and $m$ edges with the uniform distribution. These two models are the best studied and understood.
The independence between the occurrence of the edges makes $\G(n,p)$ a relatively easier model compared to many others, for analysing its properties and for analysing algorithms on $\G(n,p)$. Some algorithms depend on the degrees of vertices, and unavoidably the algorithms need to ``expose'' the degrees of the vertices as the algorithms proceed. For instance, the peeling algorithm~\cite{Jiang16,Gao14} for obtaining the $k$-core of a graph repeatedly deletes a vertex whose degree is below~$k$.

An important property of $\G(n,p)$ and $\G(n,m)$ is that, by conditioning on the degree sequence of $\G(n,p)$ or $\G(n,m)$ being ${\dvec}=(d_1,\ldots, d_n)$, the resulting random graph is exactly $\G(n,{\dvec})$, the uniformly random graph with given degree sequence ${\dvec}$. For the special case where ${\dvec}=(d,\ldots,d)$ for some constant $d$, that is 
 the random $d$-regular graph,  we simply write $\G(n,d)$.

The model $\G(n,{\dvec})$ is among the most important in the study of random graphs and 
large networks. It  is often  referred to as the Molloy--Reed  model~\cite{MolloyReed} in the network community.   
Unlike for $\G(n,p)$, probabilities of events in $\G(n,{\dvec})$ such as two vertices $u$ and $v$ being adjacent are highly non-trivial to compute. The most common methods of analysis of  $\G(n,{\dvec})$  
are the configuration model~\cite{Bollobas} for constant or slowly growing degrees,
the switching method~\cite{McKay1981} for degrees bounded by a small power of $n$, 
and the complex-analytic method \cite{IM, McKay2011} for very high degrees; see also the detailed survey by Wormald \cite{Wormald1999}.  

Nevertheless, many questions that deserve an affirmative answer remain open for $\G(n,{\dvec})$ because the methods 
listed above have severe restrictions. For instance, is $\G(n,{\dvec})$ Hamiltonian? What is the chromatic number of $\G(n,{\dvec})$? What is the connectivity of $\G(n,{\dvec})$? Using highly non-trivial switching arguments  and enumeration results for $d$-regular graphs, these particular questions were answered~\cite{Cooper,Krivelevich} for $\G(n,d)$. Using similar techniques it may be possible to work out the answers for the more general model $\G(n,{\dvec})$. However, it will be desirable to have simpler approaches.

This is the motivation of the sandwich conjecture, proposed by Kim and Vu in 2004. 
 They conjectured that for every $d\gg \log n$, the random $d$-regular graph can be sandwiched between two binomial random graphs  $\G(n,p_1)$ and $\G(n,p_2)$,
 the former with average degree slightly less than $d$, and the latter with
 average degree slightly greater.
The formal statement is as follows.
Recall that a coupling of random variables $Z_1,\ldots,Z_k$ is a random variable
$(\hat Z_1,\ldots,\hat Z_k)$ whose marginal distributions coincide with the
distributions of $Z_1,\ldots,Z_k$, respectively.  With slight abuse of notation,
we use $(Z_1,\ldots,Z_n)$ as a coupling of $Z_1,\ldots,Z_k$.

\begin{Conjecture}[Sandwich Conjecture~\cite{KimVu}]\label{con:sandwich}
For\/ $d\gg \log n$, there are $p_1=(1-o(1))d/n$ and $p_2=(1+o(1))d/n$ and a coupling $(\low, G, \up)$ such that $\low\sim \G(n,p_1)$, $\up\sim \G(n,p_2)$, $G\sim \G(n,d)$ and $\Pr(\low\subseteq G\subseteq \up)=1-o(1)$.
\end{Conjecture}

The condition $d\gg \log n$ in the conjecture is necessary. When $p=O(\log n/n)$, there exist vertices in $\G(n,p)$ whose degrees differ from $pn$ by a constant factor. Therefore, Conjecture~\ref{con:sandwich} cannot hold for this range of $d$. For $\log n\ll d\ll n^{1/3}/\log^2 n$,
Kim and Vu proved a weakened version of the sandwich conjecture where $G\subseteq\up$ is replaced by a bound on $\varDelta(\G\setminus \up)$ (see the precise statement in~\cite[Theorem 2]{KimVu}\footnote{Vu has confirmed that $\varDelta(\up\setminus G)$ in their theorem is a typo for $\varDelta(G\setminus \up$).}). Note that this weakened sandwich theorem already allows direct translation of many results from $\G(n,p)$ to $\G(n,d)$, including all increasing graph properties such as Hamiltonicity.

An immediate corollary of the sandwich conjecture, if it were true, is that one can couple two random regular graphs $G_{1}\sim \G(n,d_1)$ and $G_{2}\sim \G(n,d_2)$ such that asymptotically almost surely (a.a.s.)\ $G_{1}\subseteq G_{2}$, if $d_2$ is sufficiently greater than $d_1$. In fact we conjecture that such a coupling exists as long as $d_2\ge d_1$. 
 However,  the weakened versions of the sandwich conjecture, as proved in~\cite{KimVu} and~\cite{Dudek}, are not strong enough to imply the existence of such a coupling, even when $d_2$ is much greater than $d_1$. 
 
\begin{Conjecture}\label{con:regular}
Let $0\le d_1\le d_2\le n-1$ be integers, other than $(d_1,d_2)=(1,2)$
or $(d_1,d_2)=(n-3,n-2)$.  Assume $d_1n$ and $d_2n$ are both even. Then there exists a coupling $(G_{1}, G_{2})$ such that  $G_{1}\sim \G(n,d_1)$, $G_{2}\sim \G(n,d_2)$, and $\Pr(G_{1}\subseteq G_{2})=1-o(1)$.
\end{Conjecture}
\begin{Remark}
This conjecture or some variant of it has already been the
subject of speculation and discussion in the community, but we haven't found any written work about it.
The case when $d_1=1$ and $3\le d_2\le n-1$ is simple, since almost all $d_2$-regular graphs
have perfect matchings, which follows from them being at least $(d_2-1)$-connected~\cite{Cooper,Krivelevich}.  Generate a random $d_2$-regular graph $G_{2}$.  If $G_{2}$ has any perfect matchings, select one at random; otherwise select a random $1$-regular graph.  By symmetry, this gives a random $1$-regular graph which is a
subgraph of $G_{2}$ with probability $1-o(1)$.
\end{Remark}

The two binomial random graphs in Conjecture~\ref{con:sandwich} differ by $o(d/n)$ in edge density. This gap gives enough room to sandwich a random graph with more relaxed degree sequences.   We propose a stronger sandwich conjecture stated as  Conjecture~\ref{con:nonregular} below.

Given a vector ${\dvec}=(d_1,\ldots, d_n) \in \Reals^n$,
let $\range{\dvec}$ stand for the difference
between the maximum and minimum components of~$\dvec$. Denoting 
$\varDelta(\dvec) = \max_j d_j$, we can also write $\range{\dvec} = \varDelta(\dvec) + \varDelta(-\dvec)$.
 If  $\dvec(G)$ is the degree sequence of  a graph $G$, we will also use notations $\varDelta(G) = \varDelta(\dvec(G))$ and $\range{G} = \range{\dvec(G)}$.  

\begin{Definition}  
A  sequence $\dvec(n) \in \{0,\ldots, n-1\}^n$  is called near-regular as $n\to \infty$ if
\begin{align*}	
	\range{{\dvec(n)}}  = o(\varDelta(\dvec(n))) \quad  \text{and}\quad  \range{{\dvec(n)}}  = o(n- \varDelta(\dvec(n))).
\end{align*}
\end{Definition}

\begin{Conjecture}\label{con:nonregular}
Assume ${\dvec}= \dvec(n)$ is a near-regular degree sequence
that $\varDelta(\dvec)\gg \log n$. Then, there are $p_1=(1-o(1))\varDelta(\dvec)/n$ and $p_2=(1+o(1))\varDelta(\dvec)/n$ and a coupling $(\low, G, \up)$ such that $\low\sim \G(n,p_1)$, $\up\sim \G(n,p_2)$, $G\sim \G(n,{\dvec})$ and $\Pr(\low\subseteq G\subseteq \up)=1-o(1)$.
\end{Conjecture}

In this paper, we confirm Conjecture~\ref{con:nonregular} for all near-regular $\dvec$ where $\varDelta(\dvec)=\Theta(n)$, or $\varDelta(\dvec)\gg n/\sqrt{\log n}$ and $\range{\dvec}=O(\varDelta(\dvec)/\log n)$, which also confirms Conjecture~\ref{con:sandwich} for $d\gg n/\sqrt{\log n}$. For other near-regular degree sequences, we prove a weaker sandwich theorem, with perfect containment on both sides but with $p_2$ roughly $d\log n /n$.

\subsection{Discussion of the previous work}

 Recently, Dudek,  Frieze,  Ruci\'{n}ski and M. \v{S}ileikis~\cite{Dudek} improved one side of Kim and Vu's result,  $\low \subseteq G$, to cover all degrees $d$ such that $\log n \ll d \ll n$  and also extended it to  the  hypergraph setting. 
 In particular, this new embedding theorem allows them to translate  Hamiltonicity from binomial random hypergraphs to random regular hypergraphs.

Extending the results of \cite{Dudek} to $d=\Theta(n)$ requires new proof methods. As explained later, see Question~\ref{question} in Section \ref{sec:coupling}, a key step towards proving the sandwich conjecture by our approach is to estimate, to a desired accuracy, the edge probability in a random $\tvec$-factor of a graph $S$, where $\tvec$ is a degree sequence.
To embed $\G(n,p)$ inside $\G(n,d)$, for $d=o(n)$, it is sufficient to consider $S$ that is close to a complete graph, and $\tvec$ whose  maximum component is $o(n)$.
The edge probabilities can be estimated using a rather standard switching argument, which has already appeared in several enumeration works, e.g.~\cite{McKay1981}. However, if we wish to embed $\G(n,p)$ inside $\G(n,d)$ where $d=\Theta(n)$, we need to consider $S$ and $\tvec$ where $S$ is no longer a nearly complete graph, and components of $\tvec$ are all linear in $n$. The switching method fails in this case.

 The reader might suspect that the defect in Kim--Vu's coupling may be amended by choosing $p_2$ of slightly greater order than $d/n$ because $\varDelta(G \setminus \up)$ is quite small.
However, Kim--Vu's coupling argument fails to provide perfect upper containment unless
 $p_2$ is approximately~1.
This may look rather surprising and anti-intuitive.  
To explain this we  give a brief overview  of the coupling construction
proposed by Kim and Vu, which is essentially the same construction in Dudek,  Frieze,  Ruci\'{n}ski and M. \v{S}ileikis~\cite{Dudek}, and which we also partially adopt for our purposes, see Section~\ref{sec:coupling} for more details. 
The three graphs $\low \sim G(n,p_1)$, $G \sim  G(n,d)$ and $\up\sim G(n,p_2)$ are constructed in parallel.
Uniformly random edges from $K_n$ are added to all of the three graphs, where with a small probability an edge may be rejected in the construction of $\low$, and the rejection probability in $G$ is even smaller. This ensures the containment  $\low\subseteq G\subseteq \up$ and works well until near
the completion of the construction of $G$. However, the last few edges to be added to $G$ are highly correlated so
 most edges uniformly chosen from $K_n$ have to be rejected.  This forces $\up$ to be almost a complete graph (if we aim at a perfect upper containment). This paper gives the first result with perfect embedding of $\G(n,\bfd)$ inside a binomial random graph of  similar density.

\nicebreak
\subsection{Sandwich theorem}

Throughout the paper we assume that $\dvec$ is a realisable degree sequence, i.e.\ $\G(n,\dvec)$ is nonempty. This necessarily requires that  ${\dvec}$ has nonnegative integer coordinates and even sum.  All asymptotics in the paper refer to $n\to\infty$. For two sequences of real numbers $a_n$ and $b_n$, we say $a_n=o(b_n)$ if $b_n\neq 0$ eventually and $\lim_{n\to\infty} a_n/b_n=0$. We say $a_n=O(b_n)$ if there exists a constant $C>0$ such that $|a_n|\le C\,|b_n|$ for all $n$.
We write $a_n=\omega(b_n)$ or $a_n=\Omega(b_n)$  if $a_n>0$ always
and $b_n=o(a_n)$ or $b_n=O(a_n)$, respectively.
If both $a_n$ and $b_n$ are positive sequences, we will also write
$a_n\ll b_n$ if $a_n=o(b_n)$, and $a_n\gg b_n$  if $a_n=\omega(b_n)$. 
 Our contribution towards Conjecture~\ref{con:nonregular}  is given by the following theorem.

\begin{theorem}\label{thm:mainmain}
Assume ${\dvec} = \dvec(n) \in \Naturals^n$ is a near-regular degree sequence.  
 Then there is a coupling $(\low,\Reg,\up)$ such that $\low\sim \G(n,p_1)$, $\up\sim \G(n,p_2)$, $\Reg\sim \G(n,{\dvec})$ with
\begin{equation}
\Pr(\low\subseteq \Reg\subseteq \up)=1-o(1), \label{eq:probability}
\end{equation}
where $\dvec$, $p_1$ and $p_2$ satisfy the following conditions. 
\begin{itemize}
\item[(a)] If $\varDelta(\dvec)=O(\log n)$ then  $p_1=0$ and  $p_2 =n^{-1+\eps}$, for any  fixed $\eps\in (0,1)$.
\item[(b)] If $\log n\ll \varDelta(\dvec)  = o\left(n \right)$ then $p_1=(1-o(1))\dfrac{\varDelta(\dvec)}{n}$ and $p_2 \ge n^{-1+\eps}$ for any fixed $\eps\in (0,1)$, and  $p_2 \gg  \dfrac{\varDelta(\dvec)}{n} \frac{\log n}{\log \varDelta(\dvec)}\log \dfrac{n}{\varDelta(\dvec)}$. 
\item[(c)] If  $\varDelta(\dvec)=\Omega(n)$ then $p_1=(1-o(1))\dfrac{\varDelta(\dvec)}{n}$ and $p_2 =(1+o(1))\dfrac{\varDelta(\dvec)}{n}$.
\item[(d)] If, in addition, we assume   $\range{\dvec} = O\Bigl(\dfrac{\varDelta(\dvec)}{\log n}\Bigr)$, then (a) holds
for any $p_2\gg \dfrac{\log^3 n}{n\log\log n}$, (b) holds without the
condition $p_2\ge n^{-1+\eps}$, and  (c) holds for
$\varDelta(\dvec)\gg n/\sqrt{\log n}$.
\end{itemize}
\end{theorem}

As explained below Conjecture~\ref{con:sandwich}, 
 a tight sandwich for random regular graph with $d=O(\log n)$ does not exist.
 From the above theorem, we get that  $\G(n,d) \subseteq \G(n, p_2)$   a.a.s.  with $p_2 n \approx \log^3 n$
   for this range of $d$,  but we believe a tighter embedding should be possible. In fact, the values of $p_2$ in the above theorem can be improved in all cases by expressions in terms of $\range{\dvec}$, using  the more precise bounds of Theorem~\ref{thm:coupling} in Section~\ref{sec:coupling-results}.
 In addition, Theorem~\ref{thm:coupling}  provides sharper bounds on the probability in \eqref{eq:probability}. 
 This extra precision can be useful 
 in transferring properties of random graphs from $ \G(n, p)$ to $ \G(n, \dvec)$; see Section~\ref{s:chromatic}.

\begin{Remark}
We believe that $\range{\dvec}=o(n-\varDelta(\dvec))$ in the definition of near-regular sequences can be significantly relaxed for Theorem~\ref{thm:mainmain} to hold. However, it is not possible to remove this restriction completely.  This condition is only used 
in the proof of the sandwich theorem where all components of $\dvec$ are asymptotic to $n$. In this case $\varDelta(\dvec)/n\sim 1$ and thus we can set $p_2=1$. It is sufficient to prove that we can find coupling $\G(n,1-o(1))\subseteq \G(n,\dvec)$.
 Perhaps people are tempted to guess that the edge probability between any two vertices is $1-o(1)$ in $\G(n,\dvec)$ for such $\dvec$ and thus a coupling can be possible. This is not true.
 Let $\dvec$ be such that its complement follows a power law with exponent between 2 and 3. It is implied by~\cite[Lemma 3]{Gao18} that for such $\dvec$, there exist pairs of vertices for which the edge probability between them is $o(1)$. Hence, it is not possible to embed $\G(n,1-o(1))$ into $\G(n,\dvec)$. More examples of $\dvec$ which don't allow such an embedding can be found in~\cite{Gao16}.
\end{Remark}

   Theorem~\ref{thm:mainmain} directly implies a weaker version of Conjecture~\ref{con:regular}.

\begin{corollary}\label{cor2:main}
There is a coupling $(G_{d_1},G_{{d}_2})$ such that $G_{{d}_1}\sim \G(n, {d}_1)$, $G_{{d}_2}\sim \G(n, {d}_2)$ and 
\[
\Pr(G_{{d}_1}\subseteq G_{{d}_2})=1-o(1), \quad \mbox{if}
\]
\begin{itemize}\itemsep=0pt
\item $d_1=O(\log n)$ and $d_2\gg \log^3 n/\log\log n$; or
\item $\log n\ll d_1\ll n$ and $d_2\gg d_1\frac{\log n}{\log d_1}\log\frac{n}{d_1}$; or
\item $d_2-d_1=\Theta(n)$. 
\end{itemize}
\end{corollary}

We prove 
Theorem~\ref{thm:mainmain}  in  Section \ref{ss:mainmainproof}. It follows from several coupling results embedding a binomial random graph into $\G(n,\dvec)$, focussing on different ranges of $\varDelta(\dvec)$. 

Using our new sandwich theorem we deduce many new results for $\G(n,{\dvec})$.
Some of these results essentially rely on  the tight containment on both sides of the sandwich.
We prove several a.a.s.\ properties of $\G(n,\dvec)$ such as Hamiltonicity and universality in subgraph containment.
We determine several graph parameters of $\G(n,\dvec)$, such as the chromatic number, the small subgraph counts, the diameter, and the independence number.
We also characterise many phase transitions in edge percolation on $\G(n,\dvec)$, including the threshold for the appearance of a giant component.
These new results are presented in Section~\ref{sec:translation}.

\nicebreak
\section{From {embedding} to sandwiching}\label{sec:coupling}

Instead of constructing a sandwiched 3-component coupling $\G(n,p_1)\subseteq \G(n,\dvec)\subseteq \G(n,p_2)$ simultaneously, we will 
 embed $\G(n,p)$ into $\G(n,\dvec)$ where $\dvec$ is near-regular, and 
  to $\varDelta(\dvec)/n$. We will consider three cases in terms of the range of $\varDelta(\dvec)$: sparse, dense, and co-sparse, which correspond to, roughly speaking, sublinear $\varDelta(\dvec)$, linear $\varDelta(\dvec)$ and $n-\varDelta(\dvec)$, and sublinear $n-\varDelta(\dvec)$ respectively.  An {embedding} theorem (Theorem~\ref{thm:coupling}) which confirms that a coupling $\G(n,p)\subseteq \G(n,\dvec)$ a.a.s.\ exists is presented in Section~\ref{sec:coupling-results}.
 To prove Theorem~\ref{thm:mainmain} we will apply Theorem~\ref{thm:coupling} to embed $\G(n,p_1)$ into $\G(n,\dvec)$ and embed $\G(n,1-p_2)$ into $\G(n,(n-1) \boldsymbol{1}-\dvec)$. Then we construct a 3-component coupling with $\G(n,p_1)\subseteq \G(n,\dvec)\subseteq \G(n,p_2)$ by ``stitching'' the above two couplings together. The detailed proof of Theorem~\ref{thm:mainmain} is given in Section~\ref{sec:mainmainproof}.

To prove our {embedding} theorem, we will use a procedure called \Coupling$(\,)$ which constructs a joint distribution of $(\low,\Reg)$ where $\low\subseteq \Reg$ a.a.s.\ and their marginal distributions follow $\G(n,p)$ and $\G(n,\dvec)$ respectively. The procedure is given in Figure~\ref{fig:algorithms}.

\begin{figure*}[t!]

\hbox{\qquad
\vbox{
\no {\bf Procedure} \Coupling$(\dvec, \Time, \zeta)$:
\vspace{-0.3cm}

\begin{tabbing}
Let $\mlow{0}$, $\mreg{0}$ and $\mup{0}$ be the empty multigraphs on vertex set $[n]$.\\
For \= every \=  $1\le \time$ \=$\le \Time$:\\
\> Uniformly at random choose an edge $jk$ from $K_n$;\\
\> $\mup{\time}=\mup{\time-1}\cup \{jk\}$;\\
\> If $jk\in \mreg{\time-1}$ then\\
\>\> $\mreg{\time}=\mreg{\time-1}$;\\
   \>\> $\mlow{\time}=\mlow{\time-1}$ with probability $\zeta$, \\
\>\> $\mlow{\time}=\mlow{\time-1}\cup \{jk\}$ with probability $1-\zeta$; \\
   
\>If $jk\notin \mreg{\time-1}$, define
$\displaystyle\eta_{jk}^{(\time)}=1-\frac{\Pr(jk\in \G(n,{\dvec})\mid  \mreg{\time-1})}{\max_{jk\notin  \mreg{\time-1} }\Pr(jk\in \G(n,{\dvec})\mid  \mreg{\time-1})}$;
\\
\>\> If $\eta_{jk}^{(\time)}>\zeta$ then \textbf{Return} \IndSample$(\dvec,\mlow{\time-1},\mup{\time-1}, \time, \Time, \zeta)$; \\ 
\>\>     Otherwise,  generate  $a\in[0,1]$ uniformly randomly;\\
\>\>\>  If $a\in(\zeta,1]$ then $ \mreg{\time}= \mreg{\time-1} \cup\{jk\}$  and $\mlow{\time}=\mlow{\time-1}\cup \{jk\}$;\\
\>\> \> If $a\in [\eta_{jk}^{(\time)},\zeta]$ then $ \mreg{\time}= \mreg{\time-1} \cup\{jk\}$ and $\mlow{\time}=\mlow{\time-1}$;\\
\>\> \> If $a \in [0,\eta_{jk}^{(\time)})$ then $\mreg{\time}= \mreg{\time-1}$ and $\mlow{\time}=\mlow{\time-1}$;\\
For $\time\ge \Time+1$,  while $\mreg{\time-1}$ has fewer edges than  $\G(n,\dvec)$ repeat:\\
\>  Pick an edge $uv\notin \mreg{\time-1}$ with probability proportional to $\Pr(uv\in \G(n,d)\mid \mreg{\time-1})$,\\ 
\>  $\mreg{\time}=\mreg{\time-1}\cup \{uv\}$; \\
Assign $\Reg=\mreg{\time}$.\\
 \textbf{Return} $(\loww, \Reg, \upp)$, where $\loww\lhd \mlow{\Time}$ and $\upp\lhd \mup{\Time}$.
\end{tabbing}

\no {\bf Procedure}  \IndSample$(\dvec, M_\zeta, M_0, \time, \Time, \zeta)$:
\vspace{-0.3cm}

\begin{tabbing}
Let $\mlow{\time-1}=M_\zeta$ and $\mup{\time-1}=M_0$; and let $\Reg$ be sampled from $\G(n,{\dvec})$.\\
For \= every \=  $\time\le \tau$ \=$\le \Time$:\\
\> Uniformly at random choose an edge $jk$ from $K_n$;\\
\> $\mup{\tau}=\mup{\tau-1}\cup \{jk\}$;\\
\>$\mlow{\tau}=\mlow{\tau-1}$ with probability $\zeta$. \\
\> $\mlow{\tau}=\mlow{\tau-1}\cup \{jk\}$ with probability $1-\zeta$; \\
\textbf{Return} $(\loww, \Reg, \upp)$ where $\loww\lhd \mlow{\Time}$ and $\upp\lhd \mup{\Time}$.
\end{tabbing}
}}

\vskip-2ex
\caption{Procedures \Coupling$(\,)$ and \IndSample$(\,).$\label{fig:algorithms}}

\end{figure*}

\nicebreak
\subsection{The coupling procedure}

Procedure \Coupling$(\,)$ takes a graphical degree sequence $\dvec$, a positive integer $\Time$ and a positive real $\zeta<1$ as an input, and  outputs three random graphs $\loww$, $\Reg$, $\upp$, all on $[n]$, such that $\Reg\sim \G(n,{\dvec})$
and $\loww\subseteq \upp$. Roughly speaking, the procedure constructs $(\loww^{(t)},\Reg^{(t)},\upp^{(t)})$ by sequentially  adding edges to the three graphs, and $\loww^{(t)}\subseteq\Reg^{(t)}\subseteq\upp^{(t)}$ is maintained up to step $\Time$.
The outputs $\loww$ and $\upp$ of \Coupling$(\,)$ will be $\loww^{(\Time)}$ and $\upp^{(\Time)}$, ignoring some technicality. The output $\Reg$ will be a ``proper'' completion of $\Reg^{(\Time)}$ into a graph with degree sequence $\dvec$. For a careful choice of $\Time$ and~$\zeta$, procedure \Coupling$(\,)$ typically produces an outcome that $\loww\subseteq \Reg$ and $\Reg\setminus \upp$ is ``small''.
Moreover, if $\Time$ is chosen randomly according to a suitable distribution, which we  specify later in this section, then $\loww\sim \G(n,p_\zeta)$ and $\upp\sim \G(n, p_0)$, where $p_{\zeta}\approx p_0$ for small $\zeta>0$. (See the definition of $p_{\zeta}$ in~\eqref{p0}.)
Even though we  only need  the coupling $(\low,\Reg)$ with $\low = \loww$ for our purposes, it will be convenient to include  $\upp$ in our coupling construction in order to deduce certain properties of $\Reg$
 required for our proofs. 

In rare cases, \Coupling$(\,)$ calls another procedure \IndSample$(\,)$ (this happens when certain parameters become too large).  Procedure \IndSample$(\,)$ also generates three random graphs $\loww\sim \G(n,p_\zeta)$, $\Reg\sim \G(n,{\dvec})$ and $\upp\sim \G(n,p_0)$ but the relation $\loww\subseteq \Reg$ is not a.a.s.\ guaranteed. 
In fact, $\Reg$ will be independent of $(\loww,\upp)$. The main challenge will be to show that the probability for \Coupling$(\,)$ to call \IndSample$(\,)$ is rather small. 

If $M$ is a multigraph, we write $G \lhd M$ if $G$ is the simple graph obtained by suppressing  multiple edges in $M$ into single edges. With a slight abuse of notation, we write $jk\in \G(n,{\dvec})$ for the event that $jk$ is an edge in a graph randomly chosen from $\G(n,{\dvec})$. All graphs under consideration are defined on $[n]$ and thus we can treat graphs as subsets of $\binom{[n]}{2}$.
Thus $H\subseteq G$ is equivalent to $E(H)\subseteq E(G)$. If $X\subseteq K_n$, we write $\Pr(jk\in \G(n,{\dvec})\mid X)$ for the probability that $jk$ is an edge in $G$ where $G$ is randomly chosen from $\G(n,{\dvec})$ conditioned on $X\subseteq G$. 

The details of procedures \Coupling$(\,)$ and \IndSample$(\,)$ are
shown in Figure~\ref{fig:algorithms}.  Note that \Coupling$(\,)$ consists of
two loops indexed by a contiguous sequence of values of~$\time$.
When we refer to ``step $\time$'' or ``$\time$ iterations'', we refer to the
point in \Coupling$(\,)$ where $\time$ has that value, regardless of
which of the two loops we are~in.

Our next lemma verifies that $\loww$ and $\upp$ output
by \Coupling$(\dvec, \Time, \zeta)$
have the desired distributions if $\Time$ is an integer drawn from a 
Poisson random variable with a properly chosen mean.
(With a slight abuse of notation, we write $\Time\sim \po(\mu)$, but note
that the argument passed to \Coupling$(\,)$ is not a random variable but a single
integer drawn from the distribution $\po(\mu)$.)
Denote by 
$
N=\binom{n}{2} 
$
the number of edges in $K_n$.

\begin{lemma}\label{lem:upperdistr}
Let $\Time\sim \po(\mu)$ 
and $(\loww, \Reg, \upp)$ be the output of \Coupling$(\dvec, \Time, \zeta)$.
Then  $\upp\sim \G(n,p_0)$ and $\loww\sim \G(n, p_\zeta)$,  where
\be
	p_0  = 1- e^{-\mu/N} \ \text{ and } \ p_\zeta = 1-e^{-\mu(1-\zeta)/N}. \label{p0}
\ee
\end{lemma} 

\noindent \textit{Note}. $p_{\zeta}=p_0$ if $\zeta=0$.\ss

\begin{proof}
By the definition of \Coupling$(\,)$ and \IndSample$(\,)$, whether
\IndSample$(\,)$ is called or not, the construction for $G_{\zeta}$ and $G_0$ lasts exactly $\Time$ steps. In each step $1\le \time\le \Time$, an uniformly random edge $jk$ from $K_n$ is chosen. Then $jk$ is added to $M_0^{(\time)}$ always, and $jk$ is added to $M_{\zeta}^{(\time)}$ with probability $1-\zeta$.

Let $e_1,\ldots,e_N$ be an enumeration of the edges of $K_n$.
For $1\le z\le N$, 
 let $X_{z}$ denote the number of times that edge $e_z$ is chosen during these $\Time$ iterations. Clearly,
\begin{align*}
\Pr(X_z=0) &=\sum_{m=0}^{\infty} e^{-\mu}\frac{\mu^m}{m!} (1-1/N)^m=e^{-\mu+\mu(1-1/N)}=e^{-\mu/N}.
\end{align*}
Moreover, the probability generating function for the random vector $\X=(X_z)_{z\in [N]}$ is
\begin{align*}
\sum_{j_1,\ldots,j_{N}} &\Pr(X_1=j_1,\ldots,X_{N}=j_{N})
       x_1^{j_1}\cdots x_{N}^{j_{N}} 
     =\sum_{m=0}^\infty  e^{-\mu} \frac{\mu^m}{m!}\left(\frac{\sum_{1\le j\le N}x_j}{N}\right)^m
     \\&=\exp\left(-\mu+\mu\left(\frac{\sum_{1\le j\le N}x_j}{N}\right)\right)
     =\prod_{1\le j\le N} \exp\left(-\frac{\mu}{N}+\frac{\mu x_j}{N}\right).
\end{align*}
This implies that the conmponents of $\X$  are independent. Hence, each edge of $K_n$ is included in $G$ independently with probability $\Pr(X_z\ge 1)=1-e^{-\mu/N}$. This verifies that $\upp\sim \G(n,p_0)$.

Next we consider the distribution of $\loww$. By the construction of \Coupling$(\dvec,\Time,\zeta)$, for every $1\le \time\le \Time$, the chosen edge $e_z$ is added to $ \mlow{\time}$ with probability $1-\zeta$. Let $Y_z$ denote the  multiplicity of $e_z$ in $\mlow{\Time}$. 
Observe that the distribution of $\Y=(Y_z)_{z\in[N]}$ is similar to the distribution of $\X$ but with $\Time$ replaced by $\Time'\sim \Bin(\Time, 1-\zeta)$. It is also straightforward to verify that $\Time'\sim \po(\la')$ where $\la'=\mu(1-\zeta)$.
Thus, we conclude that $\loww\sim \G(n,p_\zeta)$.
\end{proof}

 If $G \sim \G(n,\dvec)$ and $m \le \frac12 \sum_j d_j = |E(G)|$, let  $\G(n,{\dvec},m)$
   denote the probability space of all subgraphs of $G$ containing exactly $m$ edges with the uniform distribution.
In the next lemma, we verify the marginal distribution of $\Reg^{(\time)}$ during the coupling procedure.
Define
$
m^{(\time)}$ to be the number of edges in  $\mreg{\time}$.

\begin{lemma}\label{lem:uniform}
Suppose \IndSample$(\,)$ was not called during the first $\time$ iterations of  \Coupling$(\,)$.
Then $\mreg{\time} \sim \G(n,\dvec,m^{(\time)})$.
\end{lemma}
\begin{proof}
With a slight abuse of notation, let $\mreg{\time} $ be the graph where edges are labelled with $[m^{(\time)}]$ in the order that they are added by \Coupling$(\,)$.  We will prove by induction that $\mreg{\time} $ has the same distribution as the graph obtained
by uniformly labelling edges in $\G(n,\dvec,m^{(\time)})$ with $[m^{(\time)}]$. This is obviously true for $\time=1$.

Without loss of generality, assume $\mreg{\time-1} $ has $m^{(\time)}-1$ edges and has the claimed distribution, and assume that $\mreg{\time}$ contains $m^{(\time)}$ edges. Let ${\mathcal L}(\mreg{\time-1})$ be the set of edge-labelled graphs with degree sequence ${\dvec}$ which contain $\mreg{\time-1} $ as an edge-labelled subgraph. For every $jk\notin \mreg{\time-1} $, let ${\mathcal L}(\mreg{\time-1},jk)$ be the set of edge-labelled $d$-regular graphs in ${\mathcal L}(\mreg{\time-1})$ which contains $jk$ as an edge labelled with $[m^{(\time)}]$.
Define ${\mathcal U}(\mreg{\time-1})$ and ${\mathcal U}(\mreg{\time-1},jk)$ similarly except that edges not in $\mreg{\time-1}$ are not labelled. Since every graph in  ${\mathcal U}(\mreg{\time-1},jk)$ corresponds to exactly $(M-m^{(\time)})!$ edge-labelled graphs in ${\mathcal L}(\mreg{\time-1},jk)$, and every graph in ${\mathcal U}(\mreg{\time-1})$ corresponds to exactly 
$(M-m^{(\time)}+1)!$ edge-labelled graphs in ${\mathcal U}(\mreg{\time-1})$, where $M=\frac12\sum_{j=1}^n d_j$, we have
\[
\frac{|{\mathcal U}(\mreg{\time-1},jk)|}{|{\mathcal U}(\mreg{\time-1})|}=(M-m^{(\time)}+1)\frac{|{\mathcal L}(\mreg{\time-1},jk)|}{|{\mathcal L}(\mreg{\time-1})|}.
\]
Since
 \[
\frac{|{\mathcal U}(\mreg{\time-1},jk)|}{|{\mathcal U}(\mreg{\time-1})|}= \Pr(jk\in \G(n,\dvec)\mid \mreg{\time-1}\),
 \] 
it follows that $|{\mathcal L}(\mreg{\time-1},jk)|/|{\mathcal L}(\mreg{\time-1})|$ is proportional to $\Pr(jk\in \G(n,\dvec)\mid \mreg{\time-1})$. Hence, the random graph $\mreg{\time}$ also has the claimed distribution. 

The above immediately implies the statement of the  lemma for the non-edge-labelled $\mreg{\time}$, since there are exactly $m^{(\time)}!$ ways to label   edges of $\mreg{\time}$ for any realisation of $\mreg{\time}$ with $m^{(\time)}$ edges. 
\end{proof}

Lemma~\ref{lem:uniform} immediately yields the following corollary.

\begin{corollary}\label{cor:uniform}
If $(\loww, \Reg, \upp)$ be the output of \Coupling$(\dvec, \Time, \zeta)$, then $\Reg \sim \G(n,{\dvec})$.
\end{corollary}

Thus, procedure \Coupling$(\dvec, \Time, \zeta)$ with
$\Time \sim \textbf{Po}(\mu)$ always produces a random triple of  graphs with suitable  marginal distributions.  Next, we need to  choose parameters $\mu$ and $\zeta$ 
in such a way that $p_\zeta$  approximate  the density of $\G(n,\dvec)$  reasonably  well and the probability 
of $\loww \not\subseteq \Reg$ is small. Note that $\loww \subseteq \Reg$  could only be violated when  \IndSample$(\,)$ is returned in which case $\loww$ and $\Reg$ are generated independently. Thus, 
\begin{align}
\Pr(\loww\not\subseteq \Reg) \notag
   &\le \Pr\(\text{\IndSample$(\,)$ was called during execution of \Coupling$(\,)$}\) \notag\\
&=\Pr\Bigl(\exists \time\leq \Time-1  \st \eta_{jk}^{(\time+1)}>\zeta \Bigl)\label{eta}
\\
&\leq 
\Pr \biggl(\exists \time\leq \Time-1  \st {}\notag {\quad}
\frac{\min_{jk\notin  \mreg{\time} } \Pr(jk\in \G(n,{\dvec})\mid  \mreg{\time})}
{\max_{jk\notin  \mreg{\time} }\Pr(jk\in \G(n,{\dvec})\mid  \mreg{\time})}<1-\zeta
\biggr).\kern-1em\notag
\end{align}
For each $0\le \time\le \Time-1$, define
 \[
 S^{(\time)}=K_n-\mreg{\time},
 \]
and let $\gvec^{(\time)}$ be  the degree sequence of  $\mreg{\time}$.  
Denoting  by $G(\tvec)$  the set of spanning subgraphs  of $G$ with degree sequence $\tvec$,  we get that
 	\begin{align}
 \Pr(jk\in \G(n,\dvec) \mid \mreg{\time}) &= 
 \frac{|\{G\in K_n(\dvec): \mreg{\time}\cup\{jk\}\subseteq G\}|}
 {|\{G\in K_n(\dvec): \mreg{\time}\subseteq G\}|} \notag\\
 &=\frac{|\{G \in S^{(\time)}(\dvec-\gvec^{(\time)}) :  jk \in G \}|}{|S^{(\time)}(\dvec-\gvec^{(\time)})|}. \label{eta:prob}
     \end{align}
 Thus, \eqref{eta} and \eqref{eta:prob} motivate the following question. 
 
 \nicebreak
 \begin{question}\label{question} Let $S_{\tvec}$ be a  uniform random 
 $\tvec$-factor (spanning subgraph with degree sequence $\tvec$) of a  graph $S$.
 Under which assumptions on $S$ and $\tvec$,  one can guarantee 
 \[
 	\frac{\Pr(z \in S_{\tvec})}{\Pr(z' \in S_{\tvec})} \approx 1
 \] 
 for any two edges $z,z'$ of $S$?   

\end{question}

Having an accurate estimate of the above probability ratio is crucial in our approach towards solving the sandwich conjecture, and tightening the density gap between the two binomial random graphs that sandwich $\G(n,\dvec)$. We are able to solve Question~\ref{question} for dense $S$  (for both sparse and dense $\tvec$), and also for sparse $S$ with $\tvec$ that is sparse relative to $S$ (i.e.\ $\varDelta(\tvec)=o(\varDelta(S))$). This is sufficient to prove Theorem~\ref{thm:coupling}, the {embedding} theorem. Addressing Question~\ref{question} for sparse $S$ with dense $\tvec$  relative to $S$ would allow us to resolve the sandwich conjecture completely.   

\nicebreak
\subsection{Proof techniques for the {embedding} theorem}
\label{sec:overview}

The proof of Theorem~\ref{thm:coupling} is given in Section~\ref{S:tuning}, with some technical components presented in Sections~\ref{s:enumeration} and~\ref{s:switchings}.
The proof of Theorem~\ref{thm:coupling} is given separately for different ranges of the density of $S$ and~$\tvec$. The parameters $\mu$ and $\zeta$ in procedure \Coupling$(\dvec, \Time, \zeta)$ with $\Time \sim \po(\mu)$ will be chosen differently in each case.

To prove the sparse case of Theorem~\ref{thm:coupling}, when $\varDelta(\dvec) = o(n)$, it is sufficient to answer Question~\ref{question} for $S$ that is very close to $K_n$ (which means that the complement of $S$ is sparse), and sparse~$\tvec$. The answer follows from an enumeration result of McKay~\cite{McKay1981}.  The proof of Theorem~\ref{thm:coupling}(a) is straightforward given~\cite{McKay1981} and is presented in Section~\ref{sec:sparse}. 

In the dense case of Theorem~\ref{thm:coupling}, when 
$\varDelta(\dvec)$ and $n- \varDelta(\dvec)$ are roughly linear, 
we want to answer Question~\ref{question} for dense $S$ and dense $\tvec$. We will estimate the edge probabilities by enumerating dense $\tvec$-factors of a dense graph,  using a complex-analytic approach which 
is  presented in detail in Section~\ref{s:enumeration}.  Here, we just give a quick overview.
Given $S$, the generating function for subgraphs of $S$ with given degrees is  $\prod_{jk\in S}(1+z_jz_k)$. 
Using Cauchy's integral formula,  we find that the number  $ N(S, \tvec)$ of $\tvec$-factors  of $S$ is  given by
\[
   N(S, \tvec)  = \frac{1}{(2\pi i)^n} \oint \cdots \oint \frac{\prod_{jk\in S}(1+z_jz_k)}
   { z_1^{t_1+1} \cdots z_n^{t_n+1}} \,dz_1 \cdots dz_n.
\]
We will derive an asymptotic expression of $N(S,\tvec)$ using 
 a multidimensional variant of the saddle-point method.
The integral is split into two parts. The first part corresponds to the neighbourhood of saddle points. Using the Laplace approximation, we need to estimate the  moment-generating function of a  polynomial with complex coefficients of an $n$-dimensional Gaussian random vector.  To do this,  we apply the general theory based on complex martingales  developed in~\cite{IM}. The second part consists of the integral over the other regions and has a negligible contribution.  

Estimating both parts of the integral is highly non-trivial and this analysis was previously done in the literature only for the case when $S$ is the complete graph $K_n$ or not far from it,  see  \cite{MW1990, McKay2011, BH2013, IM}. 
Extending these results to a general graph $S$ required significant improvements of known techniques.  
Our enumeration result (see Theorem~\ref{thm:enum})  gives an asymptotic value of 
$N(S,\tvec)$  for  $S$ such that every pair of vertices have $\Theta(\varDelta^2(S)/n)$ common neighbours
and under some technical conditions on $\tvec$.   
We also investigate the connection between the random graph $S_{\tvec}$ 
and the so-called \textit{$\beta$-model} which belongs to the exponential family of random graphs.  We show that  the probability  of containing/avoiding  a prescribed small set of edges is asymptotically the same for both models (see Section~\ref{s:enumeration} and Theorem~\ref{thm:prescribed}). 
Recall that for our coupling construction we only need probabilities to contain one edge.  
However, Theorems~\ref{thm:enum} and~\ref{thm:prescribed} are of independent interest and, in particular,  they extend previously known results even for the case $S=K_n$.

In the co-sparse case of Theorem~\ref{thm:coupling},  when 
 $n- \varDelta(\dvec) = o(n)$,  we need to address Question~\ref{question} for~$\tvec$ that is sparse relative to $S$ (for both dense and sparse $S$). The novel and technical analysis in this case is to estimate edge probabilities of a random $\tvec$-factor of $S$ when $S$ is sparse and pseudorandom, and $\tvec$ is sparse relative to $S$.
 We will use the switching technique under a set of pseudorandom properties for $S$. 
We give a quick introduction to the switching method here, and refer the readers to the detailed description and analysis in Section~\ref{s:switchings}.
Assume we want to estimate the probability that a random $\tvec$-factor of $S$ contains an edge $jk$ where $jk\in S$. Consider the set $U_1$ of $\tvec$-factors of $S$ which contain $jk$ and the set $U_2$ of $\tvec$-factors of $S$ which do not contain~$jk$. We will define a ``switching'' operation which switches an element $T$ in $U_1$ into another element $T'$ in $U_2$. The switching operation is defined as follows. Take an alternating walk in $S$ from vertex $j$ to $k$ with a pre-specified odd length $\ell$, such that the first edge is not in $T$ and the second edge is in $T$ and so on. This walk together with the edge $jk$ forms an alternating circuit. We require that the walk is chosen such that the corresponding circuit does not contain repeated edges in $S$. Now we swap all edges in the walk from $T$ to $S\setminus T$ and vice versa. This produces a $\tvec$-factor in $U_2$.
If we estimate the number of ways to perform a switching on a given element of~$U_1$,
and the number of ways to perform the inverse of a switching on a given element of~$U_2$, then
the ratio $|U_1|/|U_2|$ can be obtained from the ratio of these two numbers, which immediately produces the probability that a random $\tvec$-factor of $S$ contains~$jk$.

 The switching method has been extensively applied to enumerating $\tvec$-factors of a very dense graph $S$ (mostly the complete graph). If $S$ is an (almost) complete graph then the number of switchings do not depend (much) on the structure of $S$ and the analysis is much simpler. For sparser $S$, it is necessary to impose some pseudorandomness conditions for the switching method to have a chance of success, as otherwise there may not be any valid switchings.  
 For dense $S$, it is sufficient to choose small $\ell$. Switching a small number of edges helps with the control of errors arising from the switching analysis. As we are dealing with $S$ as sparse as having maximum degree of polylogarithmic order, we need to switch  up to $\log n$ edges simultaneously, since there might be no shorter alternating walks between two specified vertices.
 To our knowledge, this is the first time that the switching argument is applied to analyse a random subgraph of a graph as sparse as in our case. The sparsity requires innovative treatment for both the design of the switching, and its analysis.

\section{Translation from $\G(n,p)$ to $\G(n,\dvec)$}
\label{sec:translation}

Our sandwich theorem allows translation of many results from binomial random graphs to random graphs with specified near-regular degree sequences. Some of the translations can already be obtained from a one-sided sandwich, e.g.\ the monotone properties. Other translations require sandwiching on both sides. We give a few examples below.

\subsection{Translation of a.a.s.\ properties}

It is well known that for $p\gg \log n/n$, $\G(n,p)$ is a.a.s.\ Hamiltonian. This immediately implies the Hamiltonicity of random graphs of near-regular degrees. 
\begin{theorem}[Hamiltonicity]
Assume $\dvec$ is near-regular and $\varDelta(\dvec)\gg \log n$. Then a.a.s.\ $\G(n,{\dvec})$ is Hamiltonian.
\end{theorem}

The following universality property follows from~\cite[Theorem 1.1]{universality}. 

\begin{theorem}[Universality]
Let $k\ge 3$ be a fixed integer. Assume $\dvec$ is near-regular and $\varDelta(\dvec)\ge C n^{1-1/k}\log^{1/k} n$ for a sufficiently large constant~$C$.
Then a.a.s. $\G(n,{\dvec})$ is ${\mathbb H}(n,k)$-universal, where  ${\mathbb H}(n,k)$ denotes the set of graphs on $[n]$ with maximum degree at most~$k$. i.e.\ for $G\sim \G(n,{\dvec})$,
\[
\Pr\(\forall H\in {\mathbb H}(n,k), \exists H'\subseteq G\ s.t.\ H'\cong H\)=1-o(1).
\]
\end{theorem}

\nicebreak
\subsection{Translation of graph parameters}

\begin{theorem}[Chromatic number] \label{thm:chromatic} \footnote{The following statement was included in the SODA version of the paper~\cite{soda} as Part (b) of Theorem 3.3. The proof was wrong, and we will correct it in a future paper.

Assume $\dvec$ is near-regular.
If $\range{\dvec}\ll \varDelta(\dvec)/\log n$, and either $\log n\,\log^3\log n\ll \varDelta(\dvec)\ll n/\log n$ or $n/\sqrt{\log n}\ll \varDelta(\dvec)\ll n $, then $\chi(\G(n,{\dvec}))\sim \varDelta(\dvec)/2\log \varDelta(\dvec)$.}
\leavevmode
\makeatletter
\@nobreaktrue
\makeatother
%\begin{enumerate}
%\item[(a)] 
Assume that $\dvec$ is near-regular and  $\varDelta(\dvec), n{-}\varDelta(\dvec)=\Theta(n)$. Then a.a.s.\
\[
\chi(\G(n,{\dvec}))\sim
\frac{n}{2\log_b n}, \quad\mbox{where } b=\frac{1}{1-\frac{\varDelta(\dvec)}{n}}.
\]
%\item[(b)] Assume $\dvec$ is near-regular.
%If $\range{\dvec}\ll \varDelta(\dvec)/\log n$, and either $\log n\,\log^3\log n\ll \varDelta(\dvec)\ll n/\log n$ or $n/\sqrt{\log n}\ll \varDelta(\dvec)\ll n $, then $\chi(\G(n,{\dvec}))\sim \varDelta(\dvec)/2\log \varDelta(\dvec)$.
%\end{enumerate}
\end{theorem}

\begin{proof} The chromatic number of $\chi(\G(n,p))$ is determined by~\cite{Bollobas88} if $p,1-p=\Theta(1)$. Consequently the theorem follows by Theorem~\ref{thm:mainmain}(c). 
%To prove part (b) we need more precise bounds from Theorem~\ref{thm:coupling} so the proof is postponed 
%to  Section  \ref{s:chromatic}.
\end{proof}

Concentration of the number of small subgraphs in $\G(n,p)$ follows by~\cite{Rucinski88}, which immediately gives the following:

\begin{theorem}[Subgraph counts]
Let $H$ be an arbitrary graph of fixed order. Assume $\dvec$ is near-regular and $\varDelta(\dvec)=\Theta(n)$. Let $X_H$ denote the number of subgraphs of $\G(n,{\dvec})$ that are isomorphic to $H$. Then a.a.s.\ 
\[
X_H\sim \frac{n^{\abs{V(H)}}}{\abs{\operatorname{Aut}(H)}} \left(\frac{d}{n}\right)^{\abs{E(H)}},
\]%
where $\operatorname{Aut}(H)$ is the automorphism group of $H$.
\end{theorem}

Let $\diam(G)$ denote the diameter of $G$.
The diameter of $\G(n,p)$~\cite[Theorem 6, Corollaries 7 and 8]{Bollobas81} gives the diameter of $\G(n,{\dvec})$ as follows:

\begin{theorem}[Graph diameter]\label{thm:diameter}
 Suppose $\dvec$ is  \\ near- regular with  $\varDelta(\dvec)\gg \log^3 n$.
\begin{enumerate}
\item[(a)] If $\varDelta(\dvec)^2/n>(2+\eps)\log n$ for some fixed $\eps>0$  then a.a.s.\ $\ \diam(\G(n,{\dvec}))\le 2$.
\item[(b)] If $\varDelta(\dvec)<n^{2/3}$, let $D_0$ be the minimum integer $D$ such that  $\left((1-\eps)\varDelta(\dvec)\right)^{D}>2n\log n$ for  some fixed $\eps>0$ and all sufficiently large $n$. Then a.a.s.\ $\log_{\varDelta(\dvec)}(n-1) \le \diam(\G(n,\dvec)) \le D_0$.
\end{enumerate} 
 \end{theorem}
\begin{proof}
 Parts (a) follows directly from~\cite{Bollobas81} and Theorem~\ref{thm:mainmain}(b). For part (b), by Theorem~\ref{thm:mainmain}(b), a.a.s.\ $\G(n,(1-\eps/2)\varDelta(\dvec)/n)$ can be embedded into $\G(n,\dvec)$, where $\eps$ is a fixed constant such that $((1-\eps)\varDelta(\dvec))^{D_0}>2n\log n$ for all sufficiently large $n$. By~\cite{Bollobas81} and noting that graph diameter is a non-increasing function, a.a.s.\ $ \diam(\G(n,\dvec))\le  \diam(\G(n,(1-\eps/2)\varDelta(\dvec)/n))\le D_0$. We also have the trivial lower bound that $\diam(\G(n,\dvec))\ge \log_{\varDelta(\dvec)} (n-1)$. Our assertion follows.
\end{proof}

\begin{Remark}
A recent paper by Shimizu~\cite{Shimizu} (SODA'18) determined the diameter of $\G(n,d)$ for $d\sim \beta n^{\alpha}$ where $\beta$ and $\alpha$ are positive constants. Our Theorem~\ref{thm:diameter} recovers this  result except when $1/\alpha$ is an integer, in which case Theorem~\ref{thm:diameter}(b) yields a 2-point concentration. However, our result covers a much richer family of degree sequences. It holds for slightly non-regular degree sequences, and it does not restrict the degrees to be of form $\beta n^{\alpha}$. If $\varDelta(\dvec)=\exp(\Omega(\sqrt{\log n}))$, then part (b) typically yields a 1-point concentration, and only for very specific values of $\varDelta(\dvec)$ does it yield a 2-point concentration. 
\end{Remark}

Let $\alpha(G)$ denote the independence number of $G$, i.e.\ the order of the maximum independent set in $G$. The following theorem follows by Theorem~\ref{thm:mainmain} and~\cite{Bollobas76}.
\begin{theorem}[Independence number] \footnote{The following statement was included in the SODA version of the paper~\cite{soda} as Part (b) of Theorem 3.7. The proof was wrong (as it uses Theroem 3.3), and we will correct it in a future paper.

If $\range{\dvec}\ll \varDelta(\dvec)/\log n$, and either $\log n\,\log^3\log n\ll \varDelta(\dvec)\ll n/\log n$ or $n/\sqrt{\log n}\ll \varDelta(\dvec)\ll n $ then a.a.s.
 $\alpha(\G(n,\dvec))\sim 2n\log \varDelta(\dvec)/\varDelta(\dvec)$.
}
 Suppose ${\dvec}$ is a near-regular degree sequence. 
 %\begin{enumerate}
 %\item 
 If $\varDelta(\dvec),n-\varDelta(\dvec)=\Theta(n)$, then a.a.s.\
 \[
 \alpha(\G(n,\dvec))\sim 2\log_b n, \quad \mbox{where } b=\frac{1}{1-\frac{\varDelta(\dvec)}{n}}.
 \]
 %\item If $\range{\dvec}\ll \varDelta(\dvec)/\log n$, and either $\log n\,\log^3\log n\ll \varDelta(\dvec)\ll n/\log n$ or $n/\sqrt{\log n}\ll \varDelta(\dvec)\ll n $ then a.a.s.
% $\alpha(\G(n,\dvec))\sim 2n\log \varDelta(\dvec)/\varDelta(\dvec)$.
 %\end{enumerate}
 \end{theorem}
%\begin{proof}
% By Theorem~\ref{thm:mainmain}(b,c), there exists $p=(1-o(1))\varDelta(\dvec)/n$ such that a.a.s.\ $\G(n,p)$ can be embedded into $\G(n,\dvec)$.  Since $\alpha(G)$ is a non-increasing function of $G$, $\alpha(\G(n,\dvec))\le \alpha(\G(n,p))$, which is asymptotic to $2\log_{1/(1-p)} n$ if $p,1-p=\Theta(1)$ by~\cite{Bollobas76}, and is asymptotic to $2n \log pn/pn$ if $1\ll pn\ll n$ by~\cite{Frieze}.

%On the other hand, $\alpha(\G(n,\dvec))\ge n/\chi(\G(n,\dvec))$. By Theorem~\ref{thm:chromatic}, we immediately obtain the matching lower bound for $\alpha(\G(n,\dvec))$, which completes the proof of our assertion.
%\end{proof} 

\nicebreak
\subsection{Translation of phase transitions}
A graph property $\Gamma$ has threshold $f(n)$ in $\G(n,p)$ if
\[
\lim_{n\to \infty}\Pr(\G(n,p)\in\Gamma) =
\begin{cases}
0, & \mbox{if $p\ll f(n)$,} \\
1, &\mbox{if $p\gg f(n)$.}
\end{cases}
\]
We say $\Gamma$ has a sharp threshold $f(n)$ in $\G(n,p)$ if for every fixed $\eps>0$,
\[
\lim_{n\to \infty}\Pr(\G(n,p)\in\Gamma) =
\begin{cases}
0, & \mbox{if $p<(1-\eps) f(n)$,} \\
1, &\mbox{if $p>(1+\eps) f(n)$.}
\end{cases}
\]
The concept of (sharp) threshold extends naturally to other random graph models such as $\G(n,m)$, $\G(n,d)$ and $\G(n,\dvec)$ where $\dvec$ is near-regular.

Let $H$ be a fixed graph. Define
\begin{align*}
d(H)&= \frac{|E(H)|}{|V(H)|-1},\quad d^*(H)=\max_{H'\subseteq H: |V(H')|\ge 2} d(H').
\end{align*}
Graph $H$ is said to be strictly balanced if $d(H')<d(H)$ for every proper subgraph $H'$ of $H$ with at least 2 vertices.
The threshold of the emergence of an $H$-factor in $\G(n,p)$ for strictly balanced $H$ is determined in~\cite[Theorem 2.1]{Johansson08} to be $n^{-1/d(H)}(\log n)^{1/\abs{E(H)}}$. An upper bound for the threshold of the emergence of an $H$-factor in $\G(n,p)$ is obtained in~\cite[Theorem 2.2]{Johansson08} for general graph $H$. These results immediately yield the following.
\begin{theorem}[$H$-factors] Let ${\dvec}$ be a near-regular degree sequence.
For every $\eps>0$,
\begin{itemize}
\item if $H$ is strictly balanced,
\begin{align*}
\Pr&(\G(n,{\dvec})\ \mbox{has an $H$-factor}) {} \to
\begin{cases}
0, & \mbox{if $\varDelta(\dvec)<n^{1-1/d(H)-\eps}$}; \\[1ex]
1, & \vtop{\hbox{if $\varDelta(\dvec)>n^{1-1/d(H)+\eps}$}
      \hbox{\vrule height 2.5ex width 0pt\qquad and $n\equiv 0 \pmod {|V(H)|}$.}}
\end{cases}
\end{align*}
\item for any general graph $H$,
\[
\Pr(\G(n,{\dvec})\ \mbox{has an $H$-factor}) \to 1,
\]
if $\varDelta(\dvec)>n^{1-1/d^*(H)+\eps}$ and $n\equiv 0 \pmod {|V(H)|}$.
\end{itemize}
\end{theorem}

\begin{theorem}[Percolation on $\G(n,{\dvec})$]\label{thm:percolation}
Assume ${\dvec}$ is near-regular and $\varDelta(\dvec)=\Omega(n)$. Let $G\sim \G(n,{\dvec})$ and $G_p$ be the subgraph of $G$ obtained by independently keeping each edge with probability $p$. Let $Q$ be a monotone property and let $th(Q)$ denote a (sharp) threshold function of $Q$ in $G(n,p)$. Then $(n/\varDelta(\dvec))\cdot th(Q)$ is a (sharp) threshold function of $Q$ in $G_p$. 
\end{theorem}

We give one example of Theorem~\ref{thm:percolation}. A giant component in $\G(n,p)$ is a component of size linear in $n$. Determining the sharp threshold of the emergence of a giant component in $\G(n,p)$ is a remarkable benchmark result in random graph theory.
The emergence threshold of a giant component in other random graph models has also been extensively studied. For instance, the emergence threshold of a giant component in $G_p$ is known to be $1/(d-1)$ in the special case where $G\sim \G(n,d)$ where $d\ge 3$, following from a sequence of results~\cite{Nachmias,Krivelevich,Krivelevich13}. Theorem~\ref{thm:percolation} extends this result to near-regular degree sequences where $\varDelta(\dvec)=\Omega(n)$.

\begin{corollary}[Giant component]
Assume ${\dvec}$ is near-regular and $\varDelta(\dvec)=\Omega(n)$. The emergence of a giant component in $G_p$ has a sharp threshold $1/\varDelta(\dvec)$.
\end{corollary}

\nicebreak
\section{{The embedding theorem} and proof of Theorem~\ref{thm:mainmain}}\label{ss:mainmainproof}

\subsection{{Embedding $\G(n,p)$ into $\G(n,\dvec)$}}
\label{sec:coupling-results}

\begin{theorem}[{The embedding theorem}]\label{thm:coupling}
	Let $\dvec = \dvec(n) \in \Naturals^n$ be a degree sequence and 
	$\xi = \xi(n)>0$ be such that $\xi(n) = o(1)$.  Denote $\varDelta = \varDelta(\dvec)$.
	Then there exists a coupling  $(\low,G)$  with
	 	$\low \sim \G(n,p)$ (where $p$ is specified below) and  $G \sim \G(n,\dvec)$ for the following three cases.
	 	\begin{itemize}
	 		\item[(a)] \underline{Sparse case.} Assume
	 		\[
				\range{\dvec} \leq \xi \varDelta\qquad
				\text{and}
				\qquad
			  \xi n \geq \varDelta \gg \xi^{-3}\log n.
				\]
		Then there exists  $p = (1- O(\xi))\varDelta/n$ such that
		\[
	 		\Pr(\low \subseteq G) = 1 - e^{-\Omega (\xi^3 \varDelta)} \geq 1 - n^{-c}, 
	 	\] 
		for any constant $c>0$
	 	\item[(b)]  \underline{Dense case.} Assume
	 	$
			n \range{\dvec} \leq \xi \varDelta(n-\varDelta)$
			and
		   $n-\varDelta  \gg \xi \varDelta \gg n / \log n$.
					Then there exists  	$p = \left(1- O\left(\xi\right)\right)\varDelta/n$ such that
	 	\[
	 		\Pr(\low \subseteq G) = 1 - e^{-\Omega \left(\xi^3\varDelta\right)} = e^{-\omega(n/\log^3 n) }. 
	 	\] 
	 	\item[(c)] \underline{Co-sparse case.} Assume 
	 	\be
		\dfrac{\range{\dvec}}{n - \varDelta} =O(\sigma) 
				   \text{~and~}
				 \dfrac{n-\varDelta}{n} \log \dfrac{n}{n-\varDelta} =o(  \sigma \xi) \label{nminusd} 
			\ee
for some  positive  $\sigma=\sigma(n)$  such that 
\be
     \xi n  \gg n^{\sigma} \gg \dfrac{\log^3 n}{\log^2 \log n}. \label{xirange}
\ee
   Then there exists 
	 	$p =1 - O(\xi)$ such that
	 	\[
	 		\Pr(\low \subseteq G) = 
	 		1 - e^{-\Omega (\xi n^{1-\sigma} \log n)}
	 		\geq 1 - n^{-c} ,
                \]
			for any constant $c>0$.
	 	\end{itemize}
\end{theorem}

 The probability bounds of Theorem~\ref{thm:coupling} are almost tight in many cases. In particular, 
   they are tight up to an additional $\log^2 n$ in the exponent
   for 
 $(a)$ with $\xi \geq 1/\log n$, always for (b) and for (c) with $n^{\sigma} = \log^3 n$; see the proposition below.

\begin{proposition}  Assume $\dvec$ is $d$-regular, i.e.\ all components equal to $d$, and $(\dvec,\xi)$ satisfies one of the conditions in Theorem~\ref{thm:coupling}(a,b,c).  
 Let $(\low ,G)$ be any  coupling such  that   $ \low\sim \G(n,p)$,  $G\sim \G(n,d)$. 
 Then,
\[
\Pr(\low \subseteq G) \le 1-e^{-\Theta(\xi d)}.
\]
\end{proposition}
\begin{proof}
Note that for any coupling $(\low, G)$ where $\low\sim\G(n,p)$ and $G\in\G(n,d)$, 
\begin{align*}
1-\Pr(\low\subseteq G) &\ge \Pr_{\G(n,p)}(\text{$v_1$ has degree greater than $d$}){}=\Pr (\Bin(n-1,p) \ge d+1 ). 
\end{align*}
Using the assumptions of  Theorem~\ref{thm:coupling}(a,b,c) and omitting  uninteresting
technical details, we get that 
the probability on the right hand side above is  always at least 
$e^{-\Theta(\xi d)}$.
\end{proof}

\nicebreak
\subsection{Proof of Theorem~\ref{thm:mainmain}}
\label{sec:mainmainproof}
Now we prove that Theorem~\ref{thm:mainmain} follows from Theorem~\ref{thm:coupling}. Let $\dvec'=(n-1)\boldsymbol{1}-\dvec$. 

For part (a), noting that $K_n-G\sim \G(n,\dvec')$, and $K_n-\up\sim \G(n,1-p_2)$, it is sufficient to prove that we can a.a.s.\ embed $\G(n,1-p_2)$ inside $\G(n,\dvec')$. Fix an arbitrary $\eps>0$ and let $\xi=n^{-1+\eps}$. Let $\sigma=\eps/2$. 
It is straightforward to see that all conditions in Theorem~\ref{thm:coupling}(c) are satisfied by our choice of $\xi$ and $\sigma$.
Then a.a.s.\ $\G(n,1-C\xi)$ can be embedded into $\G(n,\dvec')$ for some $C>0$. Part~(a) now follows as this holds for any $\eps>0$.

Assume that additionally we have 
$\range{\dvec}/\varDelta=O(\log\log n/\log n)$. Then set 
\begin{align*}
\sigma&=\frac{3\log\log n-1.5\log\log\log n}{\log n},\quad 
\xi \gg \frac{\log^3 n}{n\log\log n}. 
\end{align*}
It is easy to check that all conditions for  Theorem~\ref{thm:coupling}(c) are satisfied. Hence, a.a.s.\ we can embed 
$\G(n,1-C\xi)$ into $\G(n,\dvec')$ for some constant $C>0$. Consequently, we can embed $\G(n,\dvec)$ into $\G(n, p_2)$ where $p_2\gg \log^3 n/n\log\log n$. This proves the first claim in part (d). 

For part (b), we will show that a.a.s.\ we can embed $\G(n,p_1)$ into $\G(n,\dvec)$, and embed $\G(n,1-p_2)$ into $\G(n,\dvec')$. Then, let $\pi$ be the first coupling, which embeds $\G(n,p_1)$ into $\G(n,\dvec)$, and $\pi'$ be the second coupling that 
embeds $\G(n,\dvec)$ into $\G(n,p_2)$. We can now stitch $\pi$ and $\pi'$ together to construct a coupling $(\low,\Reg,\up)$, where $\low\sim \G(n,p_1)$, $\Reg\sim\G(n,\dvec)$ and $\up\sim \G(n,p_2)$. First uniformly generate $G\in\G(n,\dvec)$. Then, conditional on $G$, generate $\low$ under $\pi$ and generate $\up$ under $\pi'$.
This yields $(\low,\Reg,\up)$ with the desired marginal distributions.
Moreover, a.a.s.\ $\low\subseteq \Reg\subseteq\up$. 

To embed $\G(n,p_1)$ into $\G(n,\dvec)$ we will apply Theorem~\ref{thm:coupling}(a).
By the assumption on $\varDelta(\dvec)$, there exists $\xi=o(1)$ satisfying both conditions in Theorem~\ref{thm:coupling}(a). Hence, there exists $p_1=(1-o(1))\varDelta(\dvec)/n$ such that a.a.s.\ $\G(n,p_1)$ can be embedded into $\G(n,\dvec)$.
To embed $\G(n,1-p_2)$ into $\G(n,\dvec')$ we will apply Theorem~\ref{thm:coupling}(c). 
Fix $\eps>0$ and assume that $\xi\ge n^{-1+\eps}$ and $\xi \gg \varDelta(\dvec)/n\log(n/\varDelta(\dvec))$.
Set $\sigma=\eps/2$. The near-regularity of $\dvec$ implies~\eqref{nminusd}. Moreover, condition~\eqref{xirange} is satisfied as $\xi n\ge n^{\eps}\gg n^{\sigma}$. By Theorem~\ref{thm:coupling}(c), a.a.s., $\G(n,1-C\xi)$ can be embedded into $\G(n,\dvec')$. Part (b) follows now as $\eps>0$ can be chosen arbitrarily.

If in addition we have\\ $\range{\dvec}/\varDelta(\dvec)=O(\log\log n/\log n)$, then set 
\[
\sigma=\max\left\{\frac{3\log \log n-1.5\log\log\log n}{\log n},\,\frac{\log \varDelta(\dvec)}{\log n}\right\}
\]
 and assume $\xi\gg n^{\sigma-1}$ and $\xi\gg \frac{\varDelta(\dvec)}{\sigma n} \log \frac{n}{\varDelta(\dvec)}$. Then all conditions in Theorem~\ref{thm:coupling}(c) are satisfied for $\dvec'$. Thus, we can embed $\G(1-p_2)$ into $\G(n,\dvec')$ with 
 \[
 p_2\gg \max\left\{\frac{\log^3 n}{\log\log n},\, \frac{\varDelta(\dvec)}{n}\frac{\log n}{\log \varDelta(\dvec)}\log \frac{n}{\varDelta(\dvec)} \right\}.
 \]
  This proves the second claim in part (d) by noting that the second term in the maximum function is always of an order that is at least of that of the first term.

For (c), as in (b), it is sufficient to embed $\G(n,p_1)$ into $\G(n,\dvec)$ and  embed $\G(n,1-p_2)$ into $\G(n,\dvec')$. First consider the case that $n-\varDelta(\dvec)=\Omega(n)$.
By the near-regularity of $\dvec$ there exists $\xi=o(1)$ which satisfies all the conditions in Theorem~\ref{thm:coupling}(b) for both $\dvec$ and $\dvec'$.
Hence, there exists $p_1=(1-o(1))\varDelta(\dvec)/n$ such that we can a.a.s.\ embed $\G(n,p_1)$ into $\G(n,\dvec)$.  Also, we can a.a.s.\ embed $\G(n, p')$ into $\G(n,\dvec')$ for $p'=(1-o(1))\varDelta(\dvec')/n$. Taking $p_2=1-p'=(1+o(1))\varDelta(\dvec)/n$ completes the proof for this range of $\varDelta(\dvec)$.
 Next, consider $\varDelta(\dvec)$ such that $n-\varDelta(\dvec)=o(n)$. We simply set $p_2=1$ in this case, and thus, it is sufficient to embed $\G(n,p_1)$ into $\G(n,\dvec)$. By the near-regularity of $\dvec$, both $\range{\dvec}/(n-\varDelta(\dvec))$ and $\frac{n-\varDelta(\dvec)}{n}\log \frac{n}{n-\varDelta(\dvec)}$ are $o(1)$ for $\dvec$ in this range. Hence, there exists $\sigma,\xi=o(1)$ which satisfy all conditions in Theorem~\ref{thm:coupling}(c). Hence, a.a.s.\ we can embed $\G(n,1-o(1))$ into $\G(n,\dvec)$. This completes the proof for part (c).

Finally we prove the last claim in part (d).
Assume that $n/\sqrt{\log n}\ll \varDelta(\dvec)\le n/2$ and additionally that $\range{\dvec}=O(\varDelta(\dvec)/\log n)$. We have already shown that $\G(n, (1-o(1))\varDelta(\dvec)/n)$ can be embedded in $\G(n,\dvec)$ by (b,c). Next we will prove that $\G(n, p')$ can be embedded in $\G(n, \dvec')$ for some $p'=1-(1+o(1))\varDelta(\dvec)/n$ which then completes the proof for part (d).
Let $\xi=1/\sqrt{\log n}$. Then both conditions in Theorem~\ref{thm:coupling}(b) are satisfied. Hence, we have the embedding for $p'=(1-O(\xi))(n-\varDelta(\dvec)+\range{\dvec})/n=1-(1+o(1))\varDelta(\dvec)/n+O(\xi+\range{\dvec}/n)=1-(1+o(1))\varDelta(\dvec)/n$, where the error $O(\xi+\range{\dvec}/n)$ in the last equation is absorbed because of the condition on the range of $\varDelta(\dvec)$.
 \qed

\nicebreak
\remove{%%%%%%%
%%%%%%%%%
\subsection{Proof of Theorem \ref{thm:chromatic}(b)}\label{s:chromatic}
In this section we use the bounds of Theorem~\ref{thm:coupling} to complete the proof of  Theorem \ref{thm:chromatic}.

First consider the case that $\range{\dvec}\ll \varDelta(\dvec)/\log n$ and  $\log n\,\log^3\log n\ll \varDelta(\dvec)\ll n/\log n$.
Then there exists $\xi=o(1/\log\varDelta(\dvec))$ such that both conditions in Theorem~\ref{thm:coupling}(a) are satisfied. Hence, there exists $p=(1+O(\xi))\varDelta(\dvec)/n$ such that a.a.s.\ $\G(n,p)$ can be embedded into $\G(n,\dvec)$. Thus, a.a.s.
 \begin{align*}
 \chi(\G(n,p)) &\le \chi(\G(n,\dvec))\le \chi(\G(n,p))+\varDelta(\G(n,\dvec)-\G(n,p))+1.
 \end{align*}
By the choice of $\xi$ and~\cite{Luczak}, it follows that $\varDelta(\G(n,\dvec)-\G(n,p))=
o\(\frac{\varDelta(\dvec)}{\log\varDelta(\dvec)}\)$, and thus
\[
  \chi(\G(n,\dvec))\sim  \chi(\G(n,p))\sim \frac{\varDelta(\dvec)}{2\log\varDelta(\dvec)}.
\]

Next, assume $\range{\dvec}\ll \varDelta(\dvec)/\log n$ and $n/\sqrt{\log n}\ll \varDelta(\dvec)\ll n$.  Take $\xi=1/\sqrt{\log n}$. Then both conditions in Theorem~\ref{thm:coupling}(b) are satisfied for both $\dvec$ and $\dvec'$. Hence, there exist
\begin{align*}
p_1&=(1-O(\xi))\frac{\varDelta(\dvec)}{n}=(1-o(1))\frac{\varDelta(\dvec)}{n}\\
p_2&=1-(1-O(\xi))\varDelta(\dvec')/n=\frac{\varDelta(\dvec)}{n}+O(\xi)=(1+o(1))\frac{\varDelta(\dvec)}{n}
\end{align*}
such that $\G(n,\dvec)$ can be sandwiched between $\G(n,p_1)$ and $\G(n,p_2)$. Our theorem follows as the chromatic number for both $\G(n,p_1)$ and $\G(n,p_2)$ are asymptotic to $\frac{\varDelta(\dvec)}{2\log\varDelta(\dvec)}$ by~\cite{Luczak}.
}
%%%%%%%%

\nicebreak
\section{Marginal distributions in  the coupling procedure}
\label{sec:lemmas}

\subsection{ Proof of Lemma~\ref{lem:upperdistr}.} 
By the definition of \Coupling$(\,)$ and \IndSample$(\,)$, whether
\IndSample$(\,)$ is called or not, the construction for $G_{\zeta}$ and $G_0$ lasts exactly $\Time$ steps. In each step $1\le \time\le \Time$, an uniformly random edge $jk$ from $K_n$ is chosen. Then $jk$ is added to $M_0^{(\time)}$ always, and $jk$ is added to $M_{\zeta}^{(\time)}$ with probability $1-\zeta$.

Let $e_1,\ldots,e_N$ be an enumeration of the edges of $K_n$.
For $1\le z\le N$, 
 let $X_{z}$ denote the number of times that edge $e_z$ is chosen during these $\Time$ iterations. Clearly,
\[
\Pr(X_z=0)=\sum_{m=0}^{\infty} e^{-\mu}\frac{\mu^m}{m!} (1-1/N)^m=e^{-\mu+\mu(1-1/N)}=e^{-\mu/N}.
\]
Moreover, the probability generating function for $\X=(X_z)_{z\in [N]}$ is
\begin{align*}
\sum_{j_1,\ldots,j_{N}} &\Pr(X_1=j_1,\ldots,X_{N}=j_{N})
       x_1^{j_1}\cdots x_{N}^{j_{N}} \\[-2ex]
       &=\sum_{j_1,\ldots,j_{N}}\sum_{m=0}^\infty e^{-\mu} \frac{\mu^m}{m!} \frac{[x_1^{j_1}\cdots x_N^{j_N}](x_1+\cdots+x_N)^m}{N^m} x_1^{j_1}\cdots x_{N}^{j_{N}}\\
     &=\sum_{m=0}^\infty  e^{-\mu} \frac{\mu^m}{m!}\left(\frac{\sum_{1\le j\le N}x_j}{N}\right)^m=\exp\left(-\mu+\mu\left(\frac{\sum_{1\le j\le N}x_j}{N}\right)\right)\\
     &=\prod_{1\le j\le N} \exp\left(-\frac{\mu}{N}+\frac{\mu x_j}{N}\right).
\end{align*}
This implies that $(X_z)_{z\in [N]}$ are independent random variables. Hence, each edge of $K_n$ is included in $G$ independently with probability $\Pr(X_z\ge 1)=1-e^{-\mu/N}$. This verifies that $\upp\sim \G(n,p_0)$.

Next we consider the distribution of $\loww$. By the construction of \Coupling$(\dvec,\Time,\zeta)$, for every $1\le \time\le \Time$, the chosen edge $e_z$ is added to $ \mlow{\time}$ with probability $1-\zeta$. Let $Y_z$ denote the  multiplicity of $e_z$ in $\mlow{\Time}$. 
Observe that the distribution of $\Y=(Y_z)_{z\in[N]}$ is similar to the distribution of $\X$ but with $\Time$ replaced by $\Time'\sim \Bin(\Time, 1-\zeta)$. It is also straightforward to verify that $\Time'\sim \po(\la')$ where $\la'=\mu(1-\zeta)$.
Thus, we conclude that $\loww\sim \G(n,p_\zeta)$.

\nicebreak
\subsection{Proof of Lemma~\ref{lem:uniform}} 
With a slight abuse of notation, let $\mreg{\time} $ be the graph where edges are labelled with $[m^{(\time)}]$ in the order that they are added by \Coupling$(\,)$.  We will prove by induction that $\mreg{\time} $ has the same distribution as the graph obtained
by uniformly labelling edges in $\G(n,\dvec,m^{(\time)})$ with $[m^{(\time)}]$. This is obviously true for $\time=1$.

Without loss of generality, assume $\mreg{\time-1} $ has $m^{(\time)}-1$ edges and has the claimed distribution, and assume that $\mreg{\time}$ contains $m^{(\time)}$ edges. Let ${\mathcal L}(\mreg{\time-1})$ be the set of edge-labelled graphs with degree sequence ${\dvec}$ which contain $\mreg{\time-1} $ as an edge-labelled subgraph. For every $jk\notin \mreg{\time-1} $, let ${\mathcal L}(\mreg{\time-1},jk)$ be the set of edge-labelled $d$-regular graphs in ${\mathcal L}(\mreg{\time-1})$ which contains $jk$ as an edge labelled with $[m^{(\time)}]$.
Define ${\mathcal U}(\mreg{\time-1})$ and ${\mathcal U}(\mreg{\time-1},jk)$ similarly except that edges not in $\mreg{\time-1}$ are not labelled. Since every graph in  ${\mathcal U}(\mreg{\time-1},jk)$ corresponds to exactly $(M-m^{(\time)})!$ edge-labelled graphs in ${\mathcal L}(\mreg{\time-1},jk)$, and every graph in ${\mathcal U}(\mreg{\time-1})$ corresponds to exactly 
$(M-m^{(\time)}+1)!$ edge-labelled graphs in ${\mathcal U}(\mreg{\time-1})$, where $M=\frac12\sum_{j=1}^n d_j$, we have
\[
\frac{|{\mathcal U}(\mreg{\time-1},jk)|}{|{\mathcal U}(\mreg{\time-1})|}=(M-m^{(\time)}+1)\frac{|{\mathcal L}(\mreg{\time-1},jk)|}{|{\mathcal L}(\mreg{\time-1})|}.
\]
Since
 \[
\frac{|{\mathcal U}(\mreg{\time-1},jk)|}{|{\mathcal U}(\mreg{\time-1})|}= \Pr(jk\in \G(n,\dvec)\mid \mreg{\time-1}\),
 \] 
it follows that $|{\mathcal L}(\mreg{\time-1},jk)|/|{\mathcal L}(\mreg{\time-1})|$ is proportional to $\Pr(jk\in \G(n,\dvec)\mid \mreg{\time-1})$. Hence, the random graph $\mreg{\time}$ also has the claimed distribution. 

The above immediately implies the statement of the  lemma for the non-edge-labelled $\mreg{\time}$, since there are exactly $m^{(\time)}!$ ways to label   edges of $\mreg{\time}$ for any realisation of $\mreg{\time}$ with $m^{(\time)}$ edges. % \qed\ss

%%%%%%%
%%%%%%%%

%%%%%%%%%%%%%%%%%%%%%%%%%%%%%%%%%
\nicebreak
\section{Proof of Theorem \ref{thm:coupling}}\label{S:tuning}
We continue using all notations introduced in Section \ref{sec:coupling}. In this paper, all graphs are defined on the vertex set $[n]$. When we do algebraic operations on graphs, we always operate on the edge sets of the graphs. In particular, for graphs $G$ and $H$, $G-H$ denotes $E(G)\setminus E(H)$,   $G+H$ denotes $E(G)\cup E(H)$, and $G\cap H$ denotes $E(G)\cap E(H)$.

As explained  before  (in particular, see  \eqref{eta} and \eqref{eta:prob}) it is important that all edges of $S^{(\time)} = K_n - \mreg{\time}$  are approximately equally likely to appear in the uniform random subgraph of $S^{(\time)}$ with degree sequence $\dvec - \gvec^{(\time)}$, where $\gvec^{(\time)}$ denotes the degree sequence of $\mreg{\time}$.  In this section we show how to choose $\mu$ and $\zeta$ such that 
the coupling procedure produces a desirable outcome.   

 We will need the following  bounds.
 \begin{lemma}\label{l:technical}
 Let $Y \sim \Bin(K,p)$  for some postive integer $K$ and  $p \in[0,1]$. 
 	\begin{itemize}
 	\item [(a)] For any $\eps \geq 0$, we have 
 	$
 		\Pr(|Y - pK| \geq \eps pK) \leq 2 e^{-\frac{\eps^2}{2+\eps} pK }.
 	$
 	\item[(b)] If $p = m/K$ for some integer $m\in (0,K)$, then 
 	$\Pr(Y=m) \geq \frac 13 \left(p(1-p)K \right)^{-1/2}.$
 	\item[(c)] Let $\Time \sim \po(\mu)$ for some $\mu>0$. Then, 
 	for any $\eps \geq 0$, 
$ \Pr(\Time \geq  \mu (1+\eps) ) \leq e^{-\frac{\eps^2}{2+\eps} \mu }.$
 	\end{itemize}
 
 \end{lemma}
 \begin{proof}
    Bound (a) follows combining the upper and lower  Chernoff bounds in multiplicative form. 
    For (b), we just use the inequalities $ \sqrt{2\pi k} \left(\frac{k}{e}\right)^k \leq k! \leq  \sqrt{2\pi k} \left(\frac{k}{e}\right)^k e^{\frac{1}{12}}$ to estimate the factorials in  the expression 
    $\Pr(Y=m) = \frac{K!}{ K^K} \cdot \frac{m^m}{m!} \cdot\frac{(K-m)^{K-m}} {(K-m)!} $.  The bound (c)  comes from approximating  $\po(\mu)$ with $\Bin(K,\mu/K)$ as $K\to \infty$ and using  the upper 
 Chernoff bound. 
 \end{proof}

Let's recall that
 \[
 N  = \binom{n}{2}, \qquad
M = \dfrac{1}{2} \sum_{j=1}^n d_j.
\]
 Lemma \ref{l:technical} is sufficient to extract some information about the density  of $S^{(\time)}$
 and the sequence $\dvec - \gvec^{(\time)}$, as described in the following lemma. Recall the definition of $p_0$ and $p_{\zeta}$ from~\eqref{p0} and that $m^{(\time)}$ denotes the number of edges in $\mreg{\time}$.  Define
 \[
  p^{(\time)}  =  (M-m^{(\time)})/M .
 \]
 
\begin{lemma}\label{l:degrees}
	Let $\xi \in (0,\frac13)$ be such $\varDelta=\varDelta(\dvec) \gg \xi^{-3}\log n$.
	Take $\Time \sim \po(\mu)$, where  $\mu$ is such that
	\[
		p_0 = 1- e^{-\mu/N} \leq (1 -\xi)M/N. 
	\]
	 Suppose \IndSample$(\,)$ was not called during the first $\time$ steps of  \Coupling$(\dvec, \Time, \zeta)$.
	 Then,
	 \begin{itemize}
	 \item[(a)]
	       $
	 			 p^{(\time)}   \geq \xi/2
	 $
	   with probability $1-e^{-\Omega(\xi^2 M)}$; 
	 \item[(b)] 
	$
	 			 \|\dvec - \gvec^{(\time)} -  p^{(\time)} \dvec \|_\infty 
		 \leq  \xi p^{(\time)}\varDelta
	 $ with probability $1-e^{-\Omega(\xi^3 \varDelta)}$.
	 \end{itemize}
\end{lemma}
\begin{proof}
By the assumption that \IndSample$(\,)$ was not called during the first $\time$ steps, and using Lemma \ref{lem:upperdistr}, we have that  $\mreg{\time} \subseteq G_0 \sim \G(n,p_0)$. 
Therefore, $m^{(\time)} = |E(\mreg{\time})|  \leq |E(G_0)|  \sim \Bin(N,p_0)$. 
Applying Lemma \ref{l:technical}(a), we  find that
\[
	\Pr \bigl( p^{(\time)} \leq  \xi/2\bigr) = \Pr \bigl(M-m^{(\time)} \leq  \xi M /2 \bigr)  = e^{-\Omega(\xi^2M  )}.
\]
Since  $e^{-\Omega(\xi^2M  )} = e^{-\Omega(\xi^3 \varDelta  )}$ we  can proceed  conditioned on the event that $ p^{(\time)} \geq  \xi/2$. 
Take $G \sim G(n,\dvec)$ and let $\hvec=(h_1,\ldots, h_n)$ denote the degree sequence of the random graph 
$G_{p^{(\time)}}$ obtained by independently keeping every edge from $G$ with probability $p^{(\time)}$. 
 	  	 By Lemma \ref{lem:uniform}, the sequence  $\dvec - \gvec^{\time}$   has exactly the same distribution as 
 	  	 $\hvec$  conditioned on the event  $|E(G_{p^{(\time)}})| = M-m^{(\time)}$,  therefore
 	 \[
 	 	 \Pr \bigl(\|\dvec - \gvec^{(\time)} -  p^{(\time)} \dvec\|_\infty \geq \xi p^{(\time)}\varDelta  \bigr)
 	 	\leq \frac{\Pr\left( \|\hvec- p^{(\time)} \dvec\|_\infty \geq      \xi p^{(\time)}\varDelta \right) }
 	 	{\Pr(|E(G_{p^{(\time)}})| = M-m^{(\time)}) }
 	 \]
 Observing  $h_j \sim \Bin(d_j,p^{(\time)})$ and using Lemma \ref{l:technical}(a), we find that 
 \begin{align*}
 	 	\Pr\bigl(\|\hvec- p^{(\time)} \dvec\|_\infty \geq   \xi p^{(\time)} \varDelta  \bigr) 
 	 	&\leq 
 	 	2  \sum_{j=1}^n \exp\left(-\frac{ \xi \varDelta }{2 d_j + \xi\varDelta}  \xi p^{(\time)}  \varDelta \right) 
 	 	\\ &= ne^{- \Omega(\xi^2  p^{(\time)}  \varDelta)} = e^{-\Omega(\xi^3 \varDelta)}.
 \end{align*}
Applying Lemma \ref{l:technical}(b) to bound $\Pr(|E(G_{p^{(\time)}})| = M-m^{(\time)})$, we complete the proof.
\end{proof}

 Unfortunately, there is not much structural information available about the graphs~$S^{(\time)}$. 
In fact, by virtue of Lemma \ref{lem:uniform},  such questions are similar in some sense to investigating the model $\G(n,\dvec)$ that is the problem we started with.   
Nevertheless,  
 it turns out that the following trivial observation will be sufficient for our purposes:
\[
	K_n - G_0^{(\time)} \subseteq S^{(\time)} \subseteq K_n - G_\zeta^{(\time)}. 
\]
where 
\[
 G_0^{(\time)} \lhd  \mup{\time}\quad \mbox{and}\quad  G_\zeta^{(\time)} \lhd  \mlow{\time}.
 \]

\begin{lemma}\label{l:edge_diff}
Let $m_0^{(\time)}=|E(G_0^{(\time)})|$ and   $m_\zeta^{(\time)} = |E(G_\zeta^{(\time)})|$. Then
\[
	G_0^{(\time)}\sim \G(n,m_0^{(\time)}), \qquad \text{and} \qquad G_\zeta^{(\time)}\sim \G(n,m_\zeta^{(\time)}).
\]
Suppose \IndSample$(\,)$ was not called during the first $\time$ steps of \Coupling$(\dvec, \Time, \zeta)$.
Assume  also that $\time \zeta \leq N/3$.  Then we have 
\[ 
	 \(1- \dfrac{3\time\zeta}{N}\)  (N-m^{(\time)}) \leq 
	   N-m_0^{(\time)}    \leq  N- m_\zeta^{(\time)} \leq \(1+\dfrac{3\time\zeta}{N}\)  (N-m^{(\time)})
\]
with probability at least 
$1 - e^{-\Omega \left( N-m^{(\time)}\right)}$.
	\end{lemma}
	\begin{proof}
	The  distributions of $G_0^{(\time)}$ and $G_\zeta^{(\time)}$
	follow directly from the definition. For the second part, it is sufficient to  bound  $\Pr \bigl( m_0^{(\time)}-m_\zeta^{(\time)}  \geq  \dfrac{3\time\zeta}{N} (N- m^{(\time)} )  \bigr)$ because
	\[
		 N- m_0^{(\time)} \leq N- m^{(\time)}   \leq  N- m_\zeta^{(\time)}.
	\]
	Recall that in \Coupling$(\,)$ a uniformly random edge $jk\in [N]$ is chosen. We call this a \textit{test}. This test will contribute 1 towards $m_0^{(\time)}-m_{\zeta}^{(\time)}$ only if (a) $jk$ is rejected by $M_{\zeta}^{(\time)}$, which happens with probability $\zeta$; and (b) if $jk\in K_n-G_{\zeta}^{(\time)}$. Otherwise the test contributes 0 to the difference.
Denote by $A$  the set of steps in procedure \Coupling$(\,)$  where case (a) occurs; so $|A| \sim \Bin(\time,\zeta)$.  Define $Y$ to be the number of elements of  $A$ where case (b) occurs, i.e.\ $jk$ is taken from  $K_n - G_\zeta^{(\time)}$. Then
\[
	Y \sim \Bin\bigl(|A|,  1-m_\zeta^{(\time)}/N\bigr) \sim  \Bin \bigl(\time, (1 -m_\zeta^{(\time)} /N)\zeta \bigr).
\] 
As we have shown above, $Y$ is an upper bound on the number of edges in
      $G_0^{(\time)}-G_\zeta^{(\time)}$.  
 Since $\time\zeta \leq N/3$, observe that  
 $m_0^{(\time)}-m_\zeta^{(\time)}  \leq
\dfrac{3\time\zeta}{2N}  (N- m_\zeta^{(\time)} ) $  implies  
$N- m^{(\time)} \geq  \dfrac12 (N- m_\zeta^{(\time)} )$.
Applying Lemma \ref{l:technical}(a), we obtain that 
 \begin{align*}
 	\Pr \left( m_0^{(\time)}-m_\zeta^{(\time)}  \geq  \dfrac{3\time\zeta}{N} (N- m^{(\time)} )  \right)
 	&\leq 
 	\Pr \left(Y  \geq  \dfrac{3\time\zeta}{2N}  (N- m_\zeta^{(\time)} )\right) = e^{-\Omega ( N-m_\zeta^{(\time)})}
 	= e^{-\Omega ( N-m^{(\time)})}.
 \end{align*}
 The statement follows as explained above.
	\end{proof}
	
 Lemma \ref{l:edge_diff} implies  that  $N-m_0^{(\time)}, N-m_\zeta^{(\time)}
  = (1+o(1)) |E(S^{(\time)})|$ 
 with high probability provided $\time \zeta  \ll N$ and $|E(S^{(\time)})| \gg 1$. This enables us to derive all the necessary structural properties about $S^{(\time)}$  from the well-studied  model $\G(n,m)$.

%%%%%%%%%%%%%%%%%%%%%%%%%%%%%%%%%
%%%%%%%%%%%%%%%%%%%%%%%%%%%%%%%%%
\nicebreak
\subsection{Sparse case}
\label{sec:sparse}

Recall that  $S_{\tvec}$ denotes  a  uniform random 
 $\tvec$-factor (spanning subgraph with degree sequence $\tvec$) of a  graph $S$. 
 In the sparse case of Theorem~\ref{thm:coupling}, we need to prove that every edge in $S$ appears in $S_{\tvec}$ approximately with the same probability, for very dense $S$ and sparse~$\tvec$. 
 The next technical lemma will be sufficient for this purpose.
 \begin{lemma}\label{l:sparse}
 	Let  $S$ be a graph on  $n$ vertices and   $\tvec = (t_1,\ldots,t_n)$ be a degree sequence such that
 the set of $\tvec$-factors of $S$ is not empty and 
 	\[
 		\varDelta (\tvec) (\varDelta(\tvec) +\varDelta(K_n-S))  \ll tn, \qquad \text{where }\  t= \frac{t_1+\cdots +t_n}{n}>0.
 	\]
 	Then, for all $jk \in S$, we have
 	\[
 		\Pr(jk\in S_{\tvec}) = \left(1+ O\left(\frac{\varDelta (\tvec) (\varDelta(\tvec) +\varDelta(K_n-S))}{tn}\right)\right)\frac{t_j t_k}{t n}.
 	\] 
 \end{lemma}
\begin{proof}
We follow the notation in~\cite{McKay1981} and apply the bounds of \cite[Corollary 2.4]{McKay1981} and  \cite[Lemma 2.8]{McKay1981} with  $\gvec = \tvec$, $H =\emptyset$, 
and $L = K_n -S +\{jk\}$ to estimate the ratio $N(\gvec, L, jk)/N(\gvec, L,\emptyset)$,
where $N(\gvec,L,A)$ denotes the number of graphs $H$ with degree sequence $\gvec$ such that $H\cap L=A$.
Observing that  
\[
	\Pr(jk\in S_{\tvec})
	=  \frac{N(\gvec, L, jk)}{N(\gvec, L, jk)+N(\gvec, L,\emptyset)}
	\]
we complete the proof.
\end{proof}

\subsubsection{Specifying $\zeta$ and $\mu$}

Take $\Time \sim \po(\mu)$, where $\mu$ is the unique solution of 
	\[
		(1 -\xi)M/N = p_0 = 1- e^{-\mu/N}.
	\]
Let $\zeta=C\xi$ where $C>0$ is a sufficiently large constant (which depends only on the implicit constant in the $O(\,)$ bound of Lemma \ref{l:sparse}).

\subsubsection{Proof of Theorem \ref{thm:coupling}(a)}

Suppose \IndSample$(\,)$ was not called during first $\time$ steps of  \Coupling$(\dvec, \Time, \zeta)$. 
Let's bound the probability that it is called at the next iteration. Take $\tvec = \dvec - \gvec^{(\time)}$. 
Note that the set of $\tvec$-factors of $S$ is not empty by definition of our coupling procedure
(there is a $\dvec$-factor at the beginning since $\dvec$ is graphical,
then at each step we only select an edge if it lies in a $\tvec$-factor). From Lemma \ref{l:degrees}(a), we have that 
\[
\Pr\bigl(p^{(\time)} \leq \xi /2 \bigr) = e^{-\Omega(\xi^2 M)} = e^{-\Omega (\xi^3 \varDelta)}
\]
and since $|\range{\tvec} - p^{(\time)}\range{\dvec}| \le \norm{\tvec-p^{(\time)}\dvec}_{\infty} $, by Lemma~\ref{l:degrees}(b),
\be
	\Pr \bigl( \range{\tvec}  \geq  p^{(\time)} \range{\dvec} +  p^{(\time)}\xi \varDelta \bigr) 
	\leq  
	\Pr\bigl( \|\tvec - p^{(\time)}\dvec\|_\infty \geq  p^{(\time)} \xi \varDelta \bigr)
	= e^{-\Omega (\xi^3 \varDelta)}.\label{range-t-upperbound}
\ee
By the assumptions of Theorem~\ref{thm:coupling}(a), 
 we have $M/n \geq  (\varDelta - \range{\dvec})/2 \geq (1-\xi)\varDelta/2 \geq  \varDelta/3$. Then, 
\[
	 p^{(\time)}\range{\dvec} +  p^{(\time)}\xi \varDelta \leq2 \xi  p^{(\time)}   \varDelta \leq   6 \xi
	 p^{(\time)}   M/n \ll t,
\]
where  $t = \frac{t_1 + \cdots + t_n}{n}  = 2(M-m^{(\time)})/n =  p^{(\time)}   M/n $. Combining with~\eqref{range-t-upperbound} we have verified that with probability $1-e^{-\Omega(\xi^3\varDelta)}$, $\tvec$ is near-regular and thus $\varDelta(\tvec)=O(t)$. 
Applying Lemma \ref{l:sparse} and observing 
 \[
 	\varDelta(\tvec) (\varDelta(\tvec) + \varDelta (K_n -S^{(\time)})) =\varDelta(\tvec) (\varDelta(\tvec) + \varDelta (\mreg{\time}))  = 
 	O(t \varDelta) \ll tn,
 \]
  we get  that,
with probability $1-e^{-\Omega (\xi^3 \varDelta)}$,
\begin{align*}
  \frac{\Pr(jk\in \G(n,\dvec) \mid \mreg{\time})}
  { \Pr(j'k'\in \G(n,\dvec) \mid \mreg{\time})} &= 
    \biggl(1+O\biggl( \frac{\varDelta(\tvec) +  \varDelta( \mreg{\time})}{n} \biggr)\biggr) \frac{t_j t_k}{ t_{j'} t_{k'}}
    =1+ O\left(\xi+ \frac{\range{\tvec}}{t} \right) \\ 
    &=   1+ O\biggl(\xi + \frac{n\,\range{\dvec}}{M} + \frac{n\xi \varDelta}{M}\biggr)
 = 1+ O(\xi) > 1 -\zeta
\end{align*}
  for any $jk, j'k' \notin\mreg{\time}$, where the last inequality follows by choosing sufficiently large $C$ in the definition of $\zeta$.
Applying the union bound for all $jk, j'k' $  and using $\xi^3 \varDelta \gg \log n$,
we get that the probability that \IndSample$(\,)$ is called at step $\time+1$ 
is  $e^{-\Omega (\xi^3 \varDelta)}$.

	Since $\xi=o(1)$, we have that $M\leq n\varDelta/2 = o(N)$ so  $\mu = O(M)$. 
Bounding  $\Time$ by Lemma \ref{l:technical}(c) and using \eqref{eta}, we conclude that 
  procedure \Coupling$(\dvec, \Time, \zeta)$ produces a ``bad'' output $(G_\zeta,G,G_0)$ with probability 
 \[
 	\Pr(G_\zeta \not\subseteq G)  = O(M)  e^{-\Omega (\xi^3 \varDelta)} = e^{-\Omega (\xi^3 \varDelta)}.
 \]
To  complete the proof we take  $(\low,G) = (G_\zeta,G)$ and $p=p_\zeta$. Recall that $G \sim \G(n,\dvec)$ by Corollary \ref{cor:uniform},  and $G_\zeta \sim \G(n,p_\zeta)$, by Lemma \ref{lem:upperdistr}, where
 \[
 p_\zeta =   1 - e^{-\mu(1-\zeta)/N} = \bigl(1-\zeta + O(\mu/N)\bigr)p_0
 	 = (1+ O(\xi) )\varDelta/n.
 \]

%%%%%%%%%%%%%%%%%%%%%%%%%%%%%%%%%
%%%%%%%%%%%%%%%%%%%%%%%%%%%%%%%%%
\nicebreak
\subsection{Dense case}
\label{sec:dense}

For Theorem~\ref{thm:coupling}(b), we aim to prove that every edge in $S$ appears approximately equally likely in $S_{\tvec}$ for dense $S$ and dense $\tvec$. 
The analog of Lemma \ref{l:sparse} in this case requires  a weak pseudo-random property about the numbers of common neighbours in $S$, as below.

\begin{lemma}\label{l:dense}
Let $S$ be a graph 
with degree sequence $\svec= (s_1,\ldots,s_n)$.  
Let $\tvec= (t_1,\ldots, t_n)$ be a degree sequence satisfying the following assumptions.  
 \begin{itemize}
 	\item[(A1)]  $\lambda(1-\lambda) \varDelta(S) \gg   \|\tvec - \lambda \svec\|_\infty  + n/\log n$, where  $\lambda = \dfrac{t_1+\cdots +t_n}{s_1 +\cdots +s_n}$.  
 	
 		\item[(A2)]    The number of common neighbours of any two vertices in  $S$
 		lies in $\left[\dfrac{\gamma \varDelta^2(S)}{n}, \dfrac{\varDelta^2(S)}{\gamma n}\right]$ for some fixed $\gamma>0$.
 \end{itemize}
 Then, for any $jk \in S$ and any $\eps>0$, we have
   \[
       \Pr(jk\in S_{\tvec})  = \left(1 + 
       O\left(n^{-1/2 + \eps}+ \dfrac{\|\tvec - \lambda \svec\|_\infty}{\lambda \varDelta(S)} \right)
       \right) \lambda,
       \]
   where the constant implicit in $O(\,)$ depends on $\gamma$ and $\eps$ only.
\end{lemma}
The proof of Lemma \ref{l:dense} is given in  Section \ref{s:prescribed}. It relies on the complex-analytical approach to enumeration of graphs with given degrees.  
The next lemma will assist us in verifying assumption (A2) in Lemma~\ref{l:dense}.
\begin{lemma}\label{l:neighbours}
	Let $H \sim \G(n,m)$ for some integer $m \gg n^{3/2}(\log n)^{1/2}$. Then,
	with probability $1 - e^{-\Omega(m^2/n^3)}$,  assumption (A2) of Lemma \ref{l:dense}
	is satisfied with $\gamma= 8$.  
\end{lemma}
\begin{proof} 
Let $\tilde{H}\sim \G(n,p)$  where $p=m/N$. 
Observe that the degrees of  $\tilde{H}$ are distributed according to $\Bin(n-1,p)$.  
Also, the number of common neighbours of any 
two vertices in $\tilde{H}$ is distributed according $\Bin(n-2, p^2)$.
Observing that $n p^2 \gg \log n$ and combining Lemma \ref{l:technical}(a)
 and the union bound, we get that, with probability $e^{-\Omega(p^2 n)}$,
 \[
 	\frac{\varDelta(\tilde{H})} {pn} \in [\tfrac12,2] \qquad \text{and} \qquad  
 	 \frac{|\{\ell \st j\ell \in \tilde{H} \text{ and } k\ell \in \tilde{H}\}|}{p^2 n} \in [\tfrac12,2]
 \]
 for all pairs of vertices $j$ and $k$.  This implies that 
 \[
 	 \dfrac{\varDelta^2(\tilde{H})}{8n}\leq \{\ell \st j\ell \in \tilde{H} \text{ and } k\ell \in \tilde{H}\} \leq  8 \dfrac{\varDelta^2(\tilde{H})}{n}.
 \]
 Note that $H$ has the same distribution as $\tilde{H}$
  conditioned on  the event that $\tilde{H}$ has exactly $m$ edges. 
   From   Lemma \ref{l:technical}(b),  we know that  
 $\Pr ( |E(\tilde{H})| = m) = \Omega(m^{-1/2})$. 
 Then observing  that   $e^{-\Omega(p^2 n)}/ \Pr ( |E(\tilde{H})| = m)  = e^{-\Omega(m^2/n^3)}$ completes the proof.
\end{proof}

\subsubsection{Specifying $\zeta$ and $\mu$}
Take $\Time \sim \po(\mu)$, where $\mu$ is the unique solution of 
	\be
		\left(1- \xi \right)M/N = p_0 = 1- e^{-\mu/N}. \label{def-p0}
	\ee
Let $\zeta = C\xi $ for some sufficiently large constant $C>0$ (which depends only on the implicit constant in $O(\,)$ of Lemma \ref{l:dense} with 
$\gamma = 9$ and $\eps =1/4$).

\subsubsection{Proof of Theorem \ref{thm:coupling}(b)}
  
First,  by the assumptions, observe that
\[
	\xi^3 \varDelta \geq  \frac{\xi^3 \varDelta^3 }{  n^2 }  \gg \frac{n}{ (\log n)^3}.
\]
Thus, it is sufficient to prove the assertion with probability $1-e^{-\Omega(\xi^3\varDelta)}$.

Suppose \IndSample$(\,)$ was not called during the first $\time$ steps of  \Coupling$(\dvec, \Time, \zeta)$. 
Again, we will bound the probability that it is called at the next iteration.  To do this, we are going to use Lemma \ref{l:dense} for $S=S^{(\time)}$ and  $\tvec = \dvec - \gvec^{(\time)}$.
As argued before, the set of $\tvec$-factors of $S$ is not empty by definition of our coupling procedure.
By the theorem assumptions,  for all~$j$,
\begin{equation}
\varDelta \geq d_j  \geq  \varDelta - \xi\, \dfrac{\varDelta(n-\varDelta)}{n}. \label{d-assumption}
\end{equation}
Using Lemma \ref{l:degrees}, we get that, with probability $1 - e^{-\Omega (\xi^3 \varDelta)}$,
\be 
 p^{(\time)} \geq \xi/2, \qquad	  t_j = ( 1+O(\xi))p^{(\time)}\varDelta. \label{p-assumption}
\ee
Let $\svec$ denote the degree sequence of $S^{(\time)}$
and $\lambda = \dfrac{t_1 + \cdots + t_n}{ s_1 + \cdots + s_n}$.
Then, by~\eqref{d-assumption} and~\eqref{p-assumption},
\[
 s_j = (1+ O(\xi)) (n-\varDelta + p^{(\time)}\varDelta) 
 \qquad
 \text{and} \qquad 
 t_j  - \lambda s_j = O(\xi) p^{(\time)}\varDelta.
\]
From the theorem assumptions, we have 
\[
 n-\varDelta \gg \xi \varDelta + \xi(n-\varDelta) = \xi  n  \geq \xi \varDelta \gg n/\log n.
 \]  
 Combining the bounds above we get that
\begin{align*}
	\lambda(1-\lambda)\varDelta(S^{(\time)})& =  
	 \frac{p^{(\time)}M (N-M)}{(N-M +p^{(\time)}M)^2}  \varDelta(S^{(\time)})\\
	 &\geq  \frac{2p^{(\time)}M (N-M)}{ n (N-M +p^{(\time)}M)}
	 = (1+o(1))   \frac{2p^{(\time)}  \varDelta (n-\varDelta)}{p^{(\time)}  \varDelta  +  n -\varDelta} \\ &\gg  
	  \norm{\tvec-\lambda \svec}_{\infty} +  n / \log n. 
\end{align*}
Thus, the assumption (A1) of Lemma \ref{l:dense} is verified.
Next, observe that $1 - p_0 \geq \xi$, so $\mu \leq N \log \tfrac{1}{\xi}$.
From Lemma  
\ref{l:technical} (c) we get that
\begin{equation}\label{dense:eq_time}
	\Pr \left( \Time >2N \log \dfrac{1}{\xi} \right)  = e^{-\Omega(N \log \frac{1}{\xi})} .
\end{equation}
Then, using $\time\le \Time$ and $\zeta=O(\xi)$ from its definition, we get that $\time \zeta = O(N \xi \log \dfrac{1}{\xi})\ll N$.  
Combining Lemma \ref{l:edge_diff}, Lemma \ref{l:neighbours} and using
the monotonicity of the number of common neighbours,  we find that, with probability at least
\[
	1 - e^{-\Omega(\xi^3 \varDelta)} - e^{-\Omega(N \log \tfrac{1}{\xi})}- e^{-\Omega( (\xi M + N -M)^2/n^3)}     = 
	1 - e^{-\Omega(\xi^3 \varDelta)},
\]
assumption (A2) of Lemma \ref{l:dense}  holds for  $S^{(\time)}$ with 
$\gamma =9$. Observing $\xi \gg (\log n)^{-1}$ from the theorem assumptions and applying Lemma \ref{l:dense} with $\eps =1/4$,  we get that, with probability  $1 - e^{-\Omega(\xi^3 \varDelta)}$
\[
	  \frac{\Pr(jk\in \G(n,\dvec) \mid \mreg{\time})}
  { \Pr(j'k'\in \G(n,\dvec) \mid \mreg{\time})} = 
  1 + O\left(n^{-1/4} + \frac{ \norm{\tvec-\lambda \svec}_{\infty}}{\lambda \varDelta(S^{(\time)})}\right)
  = 1+ O(\xi)> 1- \zeta
\]
for any $jk, j'k' \notin \mreg{\time}$, where the last inequality holds by choosing sufficiently large $C$ in the definition of $\zeta$.
Applying the union bound for all such $jk, j'k' $  
we get that the probability that \IndSample$(\,)$ is called at step $\time+1$ 
is  $e^{-\Omega (\xi^3 \varDelta)}$.

Using \eqref{eta} and \eqref{dense:eq_time}, we conclude that 
 procedure \Coupling$(\dvec, \Time, \zeta)$ produces a ``bad'' output $(G_\zeta,G,G_0)$ with probability 
 \[
 	\Pr(G_\zeta \not\subseteq G)  = O(N \log \dfrac{1}{\xi})  e^{-\Omega (\xi^3 \varDelta)} = e^{-\Omega (\xi^3 \varDelta)}.
 \]
To  complete the proof we take  $(\low,G) = (G_\zeta,G)$ and $p=p_\zeta$, recall that $G \sim \G(n,\dvec)$ by Corollary \ref{cor:uniform},  and $G_\zeta \sim \G(n,p_\zeta)$ by Lemma \ref{lem:upperdistr}, where
 \begin{align*}
 p_\zeta &=   1 - e^{-\mu(1-\zeta)/N} = 
 1 - e^{-\mu/N} + e^{-\mu/N} \bigl(1- e^{\mu\zeta/N}\bigr)\\ &=
 p_0 + O\bigl(\xi(1-p_0)\log(1-p_0)\bigr)
 	 = (1- O(\xi) )p_0 = (1 - O(\xi)) \varDelta/n,
 \end{align*}
 where the last equation follows by~\eqref{def-p0} and the theorem assumption that $\dvec$ is near-regular.
%%%%%%%%%%%%%%%%%%%%%%%%%%

\nicebreak
\subsection{Co-sparse case}\label{S:co-sparse}

For the co-sparse case of Theorem~\ref{thm:coupling}, we need to estimate the edge probability in $S_{\tvec}$ where $\tvec$ is sparse relative to $S$. Here, $S$ can be as dense as the complete graph, or as sparse as having average degree polynomial in $\log n$. We will estimate the edge probabilities assuming that $S$ satisfies some pseudorandom properties. 
For $\xvec,\yvec \in \Reals^n$, denote
\begin{equation}\label{Z_def}
	\langle  \xvec,\yvec\rangle_S = \sum_{(jk)\st jk \in S} x_j y_k
\end{equation}
Let $J$ denote the $n\times n$ matrix with all entries equal to $1$.  Given a graph $G$, let 
\[
\mbox{$A(G)$ denote its adjacency matrix.}
\]
Below is an analog of Lemmas~\ref{l:sparse} and~\ref{l:dense}, for the case where $\tvec$ is sparse relative to $S$.

\begin{lemma}\label{l:co-sparse}
Let $S$ be a graph on $n$ vertices and $\tvec$ be a degree sequence  
that  the set of $\tvec$-factors of $S$ is not empty and the following assumptions hold. There exist
 a supergraph $S'\supseteq S$,  and some  $\alpha = \alpha(n) \in (0,1)$, $\beta=\beta(n)=O(1)$ and $\gamma=\gamma(n)=O(1)$ such that
\begin{itemize}
		 \item[(A1)] 
		   \[\frac{\varDelta(\tvec)}{ \alpha \varDelta(S)} + 
			\frac{1}{ \alpha^2 \varDelta(\tvec) \varDelta(S)} =o(1)
\]
		 \item[(A2)]  
		   	\[	
			  6\range{S'} +\left\|A(S') - p' J\right\|_2   \leq n^{-\alpha} \varDelta(S'), 
			  \qquad \text{where $p' = |E(S')|/N$.}
	\]
	
	\item[(A3)] 
	\[
	\left(\frac{\varDelta(S')}{\varDelta(S)-\range{S}}\right)^{1/\alpha} \left(\frac{\varDelta(\tvec)}{\varDelta(\tvec)-\range{\tvec}}\right)^{1/\alpha}\le \beta.
	\]
	
	\item[(A4)]  
	\[
		\left|\log \frac{\langle  \xvec,\yvec\rangle_S  }{  \|\xvec\|_1\|\yvec\|_1 \varDelta(S)/n}\right| \le \gamma
	\]
	 for all $\xvec,\yvec \in [0,1]^n$  with $\|\xvec\|_1, \|\yvec\|_1 \geq  
	n \min\{(16 \beta)^{-6}, \beta^{-14}/2\}$. 
\end{itemize}
Then, for any $jk \in S$, 
	\[
		\Pr(jk \in S_{\tvec}) = \exp\left(O\left(\gamma+\dfrac{\varDelta(\tvec)}{\alpha \varDelta(S)}
		+  \dfrac{1}{\alpha^2 \varDelta(\tvec) \varDelta(S)} + \dfrac{\range{S}}{\alpha(\varDelta(S)-\range{S})} +
		\dfrac{\range{\tvec}}{\alpha (\varDelta(\tvec)-\range{\tvec})}\right)\right) \frac{\varDelta(\tvec)}{\varDelta(S)}.
	\]
\end{lemma}
The  proof of Lemma \ref{l:co-sparse} is given in Section \ref{s:switchings}, and uses a switching argument.

\subsubsection{Specifying $\mu$ and $\zeta$}
Take $\Time \sim \po(\mu)$, where $\mu$
is the unique solution of 
\be
	1 - \xi = p_0 = 1 - e^{-\mu/N}. \label{pnot}
\ee
 Set 
 \be
 \zeta = C\dfrac{n-\varDelta}{\xi n},\quad \mbox{for sufficiently large $C>0$.} \label{zeta}
 \ee
 Note that $\zeta = o(1)$ since, by assumptions, 
  $\sigma\leq 1$ and  $\sigma \xi n \gg (n-\varDelta) \log \dfrac{n}{n-\varDelta} \geq n- \varDelta$.

\subsubsection{Proof of Theorem \ref{thm:coupling}(c)}
 
 Suppose \IndSample$(\,)$ was not called during the first $\time$ steps of  \Coupling$(\dvec, \Time, \zeta)$. 
First we bound the probability that \IndSample$(\,)$  is called in the next iteration,
using the following two claims. 
Let \[
\tvec=n\boldsymbol{1}-\dvec.
\] 
\begin{lemma}\label{claim:conditions}
Let $S=S^{(\time)}=K_n-\mreg{\time}$, $S'=K_n-G_{\zeta}^{(\time)}$ and $S''=K_n-G_0^{(\time)}$.  Then, under the assumptions of Theorem~\ref{thm:coupling}(c),  there exist $\alpha,\beta,\gamma=O(1)$ such that the following hold with   probability $1 - e^{-\Omega \left(\xi n^{1-\sigma} \log n\right)}$. 
\begin{itemize}
\item[(a)] Conditions (A1), (A2), (A3) and (A4) of Lemma~\ref{l:co-sparse} hold;
\item[(b1)] $\mu \zeta = o(\sigma N)$;
\item[(b2)] $|E(S)|\ge \xi N/4$;  
\item[(b3)] $\Time \le 2N \log \tfrac1\xi $, $\Time\zeta/N=o(\sigma)$; 
\item[(c1)] $|E(S')|-|E(S'')|=(1+o(\sigma)) |E(S')|$;
\item[(c2)] $\range{S'},\range{S''}=o(\sigma)\varDelta(S')$;
\item[(c3)] $\range{S}=o(\sigma)\varDelta(S)$. 
\end{itemize}
\end{lemma}
To complete the proof of the theorem, we will use Lemma~\ref{claim:conditions}(a) and (b1)--(b3). Parts (c1)--(c3) are used in the proof for part (a). The proof for the Lemma is postponed to Section~\ref{sec:conditions}.

It is easy to see that $\alpha<1$ by the theorem assumption.
By Lemma~\ref{claim:conditions}(a),
\begin{align*}
	  \frac{\Pr(jk\in \G(n,\dvec) \mid \mreg{\time})}
  { \Pr(j'k'\in \G(n,\dvec) \mid \mreg{\time})} &= 
  \frac{\Pr (jk \notin S_{\tvec})} {\Pr (j'k' \notin S_{\tvec})}  
   = 
    \frac{1 - O(1) \varDelta(\tvec) /\varDelta(S)}
    {1
    - O(1)  \varDelta(\tvec) /\varDelta(S)} \\&=
         1 + O(1)\varDelta(\tvec) /\varDelta(S).
\end{align*}
for any $jk, j'k' \notin \mreg{\time}$.  Note that the set of $\tvec$-factors of $S$ is not empty by definition of our coupling procedure (at each step we choose edges proportional to the probability of containing a given edge, which is not zero).
By Lemma~\ref{claim:conditions}(b2), we have $\varDelta(S)=\Omega(\xi n)$. We also have
$\varDelta(\tvec)=n-(\varDelta(\dvec)-\range{\dvec})=O(n-\varDelta)$ since $\range{\dvec}=O(n-\varDelta)$ by the theorem assumption that $\range{\dvec}=O(\alpha(n-\varDelta))$.
Thus, the ratio of the probabilities above is
\[ 
    1 + O\left(\dfrac{n-\varDelta}{\xi n}\right) >  1- \zeta,
\]
by choosing sufficiently large $C$ in our definition of $\zeta$.

Applying the union bound for all such $jk, j'k' $  
we get that the probability that \IndSample$(\,)$ is called at step $\time+1$ 
is $Ne^{-\Omega(\xi n^{1-\sigma} \log n)}=e^{-\Omega(\xi n^{1-\sigma} \log n)}$. 

	Using \eqref{eta} and Lemma~\ref{claim:conditions}(b3), we conclude that 
  procedure \Coupling$(\dvec, \Time, \zeta)$ produces a ``bad'' output $(G_\zeta,G,G_0)$ with probability 
 \[
 	\Pr(G_\zeta \not\subseteq G)  = O(N \log \dfrac{1}{\xi})  e^{-\Omega(\xi n^{1-\sigma} \log n)} = e^{-\Omega(\xi n^{1-\sigma}\log n)},
 \]
 where the last equation holds by the theorem assumption~\eqref{xirange}.
To  complete the proof we take  $(\low,G) = (G_\zeta,G)$ and $p=p_\zeta$, and recall that $G \sim \G(n,\dvec)$ by Corollary \ref{cor:uniform},  and $G_\zeta \sim \G(n,p_\zeta)$, by Lemma \ref{lem:upperdistr}, where
 \begin{align*}
 p_\zeta =   1 - e^{-\mu(1-\zeta)/N} &= 
 1 - e^{-\mu/N} + e^{-\mu/N} \bigl(1- e^{\mu\zeta/N}\bigr)=p_0+(1-p_0) (1- e^{\mu\zeta/N}), 
 \end{align*}
 since $p_0=1-\xi$. As $\mu\zeta/N=O(1)$ by Lemma~\ref{claim:conditions}(b1) we have
 \[
 p_{\zeta}=1-O(\xi),
 \]
 as desired.

%%%%%%%
%%%%%%

\subsubsection{Proof of Lemma~\ref{claim:conditions}}
\label{sec:conditions}

We say that an event happens \textit{with sufficiently high probability} (w.s.h.p.) if the probability that it occurs is at least $1-e^{-\Omega(\xi n^{1-\sigma} \log n)}$.

\smallskip\noindent {\bf Proof of (b1)--(b3)}

We first prove (b2). 
Observe
\be
	\range{\dvec} =O(\sigma(n-\varDelta))=O(n-\varDelta). \label{drange}
\ee
By the theorem assumptions, we have $\sigma<1$.
Hence 
\be
\dfrac{n-\varDelta}{n } \log \dfrac{n }{n-\varDelta} = o(\sigma \xi)=o(1). \label{1}
\ee
It follows then that $n-\varDelta=o(n)$. 
Consequently by~\eqref{1} and~\eqref{drange}, we have
\be
\xi \gg \dfrac{n-\varDelta}{n } =O\left( 
	\dfrac{N-M}{N}\right).\label{2}
\ee
By definition~\eqref{pnot} we have
$p_0=1-\xi$. Thus, by~\eqref{2}, we obtain
$p_0\le (1-2\xi/3)M/N$. 
Using Lemma \ref{l:degrees}(a), we obtain that, with probability 
$1-e^{-\Omega(\xi^2 M)}$, 
\[
	|E(S^{(\time )})| = N-m^{(\time)} \geq M- m^{(\time)} \geq  \xi M/3 \geq \xi N/4.
\]
Note that $1-e^{-\Omega(\xi^2 M)}= 1 - e^{-\Omega \left(\xi n^{1-\sigma} \log n\right)}$. This follows from assumption~\eqref{xirange} and~\eqref{drange}.
Now we have verified (b2).

For (b1), observe that $\mu  = N \log \tfrac 1\xi$. Using~\eqref{zeta} and $\xi \gg \dfrac{n-\varDelta}{\sigma n } \log  \dfrac{n}{n-\varDelta} $, we get that
\[ 
 \mu \zeta = O\Bigl(N  \dfrac{n-\varDelta}{\xi n} \log \dfrac1\xi\Bigr)  = o(\sigma N).
 \]
For (b3), note that
by Lemma \ref{l:technical}(c), 
\[
	\Pr\(\Time >2N \log \tfrac1\xi \) = e^{-\Omega(N \log \tfrac1\xi )}\le e^{-\Omega \left(\xi n^{1-\sigma} \log n\right)}.
\]
The other equation then follows by the definition of $\mu$.
 
\smallskip \noindent {\bf Proof of (c1)--(c3)}

Part (c1) follows immediately from Lemma~\ref{l:edge_diff} and (b3). 
For (c2), by Lemma~\ref{l:Z_xy}, w.s.h.p.
\[
\range{S'}, \range{S''} = O(\eps) \varDelta(S').
\]
Now (c2) follows by the theorem assumption that $n^{\sigma}\gg \log^3 n/\log^2(\log n)$, which implies that
\be
\sigma\ge \frac{\log\log n}{\log n} \gg \sqrt{\log n/n^{\sigma}}=\eps.~\label{c2}
\ee
Now (c3) follows from (c1) and (c2).

\smallskip \noindent {\bf Proof of (a)}
 
Before proceeding to part (a), we prove the following lemma.
\begin{lemma}\label{l:Z_xy}
	Let $H \sim \G(n,m)$ where  $m \gg n\log n$.
	 Let $h_1,\ldots,h_n$  be the degrees of the random graph $H$ and 	  let $A(H)$ denote  its  adjacency matrix.    
	Assume 
 	$\eps = \eps(n)>0$ is such that
 	$ \sqrt{\dfrac{n\log n}{m}}\ll \eps \ll 1$.
 	Then, with probability $1 -  e^{-\Omega( \eps^2 m/n)}$,  
 	 \[
 	 		\max_j |h_j  - 2m/n| \leq \eps m/n, \qquad \|A(H) - (m/N) J\|_2 \leq \eps m/n,
 	 \]
 	 and, 
  uniformly	for all $\xvec,\yvec \in [0,1]^n$  with $\|\xvec\|_1, \|\yvec\|_1 = \Omega(n)$,
 	 \[
 	 	\bigl| \langle \xvec, \yvec \rangle_H- (m/N)\|\xvec\|_1\|\yvec\|_1 \bigr| \leq \eps m.
 	 \]
		\end{lemma}
\begin{proof}
Similarly to Lemma \ref{l:neighbours},  it is sufficient to prove the bounds above for  the random graph $\tilde{H}\sim \G(n,p)$  where $p=m/N$. 
since the probability of the  event  $|E(\tilde{H})|=m$ is 
   substantially larger than  $e^{-\Omega(\eps^2 m/n)}$. 
    Observe that the degrees of  $\tilde{H}$ are distributed according to $\Bin(n-1,p)$.
Applying Lemma \ref{l:technical}(a) and using the union bound, we show get the concentration bound for degrees.  The two other bounds  for $\tilde{H}$ hold with even better probability estimates 
and are given in  Lemma \ref{A:spec} and Lemma \ref{A:z_con}.
\end{proof}

Now we are ready to prove part (a). Take $\alpha=\sigma/4$. 
By definition of $\tvec$, 
\be
\dfrac{\varDelta(\tvec)}{\varDelta(S)} \le \dfrac{ n-\varDelta+\range{\dvec}}{|E(S)|/n} = O\left(\dfrac{n-\varDelta}{\xi n}\right)=o(\alpha), \label{4}
\ee
where the second last equality follows by~\eqref{drange} and (b2), and the last equality follows by~\eqref{1}.
Observe also that $\alpha^{-1} = O(\log n) = O \((n-\varDelta)^2 \log \dfrac{n}{n-\varDelta}\)$. Then
\[
	\dfrac{1}{\varDelta(\tvec)\varDelta(S)}  = O \left( \dfrac{1}{(n- \varDelta) \xi n}\right)
	 = O\left(  \alpha \dfrac{n- \varDelta }{ \xi n}   \log \dfrac{n}{n-\varDelta}\right) = o(\alpha^2),
\]
which proves  (A1). 

Next we are going to verify that 
\[
 \range{S'} +\left\|A(S') - p' J\right\|_2   \leq n^{-\alpha} \varDelta(S').
\]
We will apply Lemma~\ref{l:Z_xy} to $S'$ with $\eps=\sqrt{\log n/n^{\sigma}}$. By our choice of $\eps$ and $\alpha$ and the theorem assumption~\eqref{xirange}, we have
\be
\eps=o(n^{-\sigma/4})=o(n^{-\alpha}).\label{eps-alpha}
\ee
By (b2) and the fact that $S\subseteq S'$, we may assume that 
\[
    |E(S')|\ge \xi N/4.
\]
 By assumption~\ref{xirange}, we have
$\eps\gg \sqrt{n\log n/|E(S')|}$, and hence the assumption on $\eps$ in Lemma~\ref{l:Z_xy} is satisfied.

By Lemma~\ref{l:Z_xy}, with probability at least $1-e^{-\Omega(\eps^2 |E(S')|/n)}=1-e^{-\Omega(\xi n^{1-\sigma} \log n)}$,
\be
\range{S'}=O(\eps |E(S')|/n), \quad \mbox{and}\quad \left\|A(S') - p' J\right\|_2  =O(\eps |E(S')|/n).\label{events}
\ee
Finally, using~\eqref{eps-alpha} and $|E(S')|\le n\varDelta(S') $, we have
\[
\eps |E(S')|/n = O(\eps \varDelta(S')) =o(n^{-\alpha} \varDelta(S')).
\]
This, together with~\eqref{events} and~\eqref{4}, implies that (A2) holds for $S'$ w.s.h.p.

For (A3), we want to bound $\range{S}$ and $\range{S'}$ relative to $\varDelta(S')$. We also need to bound the ratio $\varDelta(S')/\varDelta(S)$.
Using the fact that $S''\subseteq S\subseteq S'$ we have
\[
\range{S}\le \varDelta(S'-S)+\range{S'}\le \varDelta(S'-S'')+\range{S'}.
\]
By (c1) and (c2),
\begin{align*}
\varDelta(S'-S'') &\le \range{S'}+\range{S''}+\frac{2(|E(S')-|E(S'')|)}{n}= o(\sigma)\varDelta(S')\\
\varDelta(S)&= (1+o(\sigma))\varDelta(S')\\
\range{S'}&= o(\sigma) \varDelta(S').
\end{align*}
Hence
\[
\left(\frac{\varDelta(S')}{\varDelta(S)-\range{S}}\right)^{1/\alpha} =\left(1+o(\sigma)\right)^{1/\alpha}=1+o(1).
\]
Next, we bound $\range{\tvec}$ relative to $\varDelta(\tvec)$. By definition $\range{\tvec}=\range{\dvec}$. 
Thus, observing that $\varDelta(\tvec)-\range{\tvec}=n-\varDelta$, we obtain
\[
 \left(\frac{\varDelta(\tvec)}{\varDelta(\tvec)-\range{\tvec}}\right)^{1/\alpha} =\left(1+\frac{\range{\dvec}}{n-\varDelta} \right)^{1/\alpha}=O(1),
 \]
 where the last equation follows by~\eqref{drange} and the choice of $\alpha=\sigma/4$. Thus, we have verified (A3) with some $\beta=O(1)$.

For (A4), we will bound $ \langle \xvec, \yvec \rangle_S'$,   
$ \langle \xvec, \yvec \rangle_{S''}$  and translate the bound to $S$.
By Lemmas~\ref{l:Z_xy} and~\ref{l:edge_diff} and using (b3), w.s.h.p.
\begin{align*}
		\frac{ \langle \xvec, \yvec \rangle_{S''}}
		{\|\xvec\|_1 \|\yvec\|_1},\frac{ \langle \xvec, \yvec \rangle_{S'}}
		{\|\xvec\|_1 \|\yvec\|_1} = 
		(1+ O(\eps)+o(\sigma)) (n-m^{(\time)})/N=(1+o(\sigma)) \varDelta(S)/n, 
			\end{align*}
	uniformly for all $\xvec,\yvec \in [0,1]^n$  with $\|\xvec\|_1, \|\yvec\|_1 =\Omega(n)$.
	The term $O(\eps)$ in this equation is absorbed by $o(\sigma)$ by~\eqref{c2}, and the equation then holds by (c3). Now (A4) follows by monotonicity of $ \langle \xvec, \yvec \rangle_S$ with respect to $S$. \qed
	
%%%%

\nicebreak
  \section{Complex-analytic approach}\label{s:enumeration}
  
  In this section we establish an asymptotic formula for the number of factors (subgraphs with given degree sequence) of a graph in the dense case.   Then, as a corollary, we  prove Lemma~\ref{l:dense}. 
  
Let $S$ be a simple graph.
We start from the observation that  $\prod_{jk\in S}(1+z_jz_k)$
is the generating function for subgraphs of $S$ with powers of  $z_1, \ldots, z_n$ corresponding to
degrees. In particular, the number $ N(S, \tvec)$ of $\tvec$-factors  of $S$ is  given by
\[
 N(S, \tvec) = [z_1^{t_1}\cdots z_n^{t_n}] \prod_{jk\in S}(1+z_jz_k),
\]
where $[\,\cdot\,]$ denotes coefficient extraction. Using Cauchy's integral formula, it follows that 
\[
   N(S, \tvec)  = \frac{1}{(2\pi i)^n} \oint \cdots \oint \frac{\prod_{jk\in s}(1+z_jz_k)}
   { z_1^{t_1+1} \cdots z_n^{t_n+1}} \,dz_1 \cdots dz_n.
\]
Let
\[
	 U_n(\rho)=\{\thetavec = (\theta_1,\ldots, \theta_n)\in \Reals^n :  \norm{\thetavec}_{\infty} \le \rho\}
\]
Substituting $z_j = e^{\beta_j + i \theta_j}$, we get that
\begin{equation}\label{integral}
\begin{aligned}
 N(S, \tvec)  &= \frac{1}{(2\pi)^n} \int_{-\pi}^{\pi}\cdots \int_{-\pi}^{\pi} \frac{\prod_{jk\in S}(1+e^{\beta_j+\beta_k+i(\theta_j+\theta_k)})}{e^{\sum_{j=1}^n t_j(\beta_j+i\theta_j)}} d \theta_1\cdots d \theta_n\\
&=\frac{\prod_{jk\in S}(1+e^{\beta_j+\beta_k})}{(2\pi)^n e^{\sum_{j=1}^n t_j\beta_j}} \int_{-\pi}^{\pi}\cdots \int_{-\pi}^{\pi} 
\frac{\prod_{jk\in S}\frac{1+e^{\beta_j+\beta_k+i(\theta_j+\theta_k)}}{1+e^{\beta_j+\beta_k}}}{ e^{\sum_{j=1}^n i t_j\theta_j} }\,d \theta_1\cdots d \theta_n\\
&=\frac{\prod_{jk\in S}(1+e^{\beta_j+\beta_k})}{(2\pi)^n e^{\sum_{j=1}^n t_j\beta_j}} \int_{U_n(\pi)} F_{S, \tvec} (\thetavec)\, d \thetavec,
\end{aligned}
\end{equation}
where  
\[
   F_{S, \tvec} (\thetavec) = \frac{ \prod_{jk \in S} \bigl(1+\lambda_{jk}(e^{i(\theta_j+\theta_k)}-1)\bigr)}{e^{\sum_{j=1}^n i t_j \theta_j}}
\]
and 
\begin{equation}\label{def_lambda}
	\lambda_{jk} =   \lambda_{jk}(\betavec)=
		\frac{e^{\beta_j  + \beta_k} }{1+e^{\beta_j  + \beta_k}},  \text{ for } jk \in S.  
\end{equation}
The choice of parameters $\betavec = (\beta_1,\ldots\beta_n)$  will be specified later.

The values $(\lambda_{jk})$ defined  in \eqref{def_lambda}  have an interesting property:  if we consider a random subgraph  $S_{(\lambda_{jk})}$  of $S$ with independent adjacencies where, for each $jk \in S$, the probability that  vertices $j$ and $k$ are connected equals $\lambda_{jk}$, then  the probability of each outcome  depends only on its degree sequence $\tvec=(t_1,\ldots, t_n)$. In other words, the conditional distribution of $S_{(\lambda_{jk})}$ with respect to given $\tvec$  is uniform.  The random  model  of $S_{(\lambda_{jk})}$
 is referred as the $\beta$-model and it is a special case of the exponential family of random graphs,  see \cite{CDS2011, IM} for more details.  A further connection between $S_{(\lambda_{jk})}$ and $S_{\tvec}$ is established in Section \ref{s:prescribed}.

The exact value of the integral \eqref{integral} can be found very rarely. Instead, we will  approximate it. 
The complex-analytical approach consists of the following steps:  
\begin{itemize}
\item[(i)]  estimate the contribution of critical regions around concentration points, where  the integrand achieves its maximum value, 
\item [(ii)] show that other regions give a negligible contribution.  
\end{itemize}

The maximum absolute value of $|F_{S, \tvec} (\thetavec)| $   is $1$. It is  achieved at points 
$(0,\ldots,0)$ and $(\pm \pi, \ldots, \pm \pi)$.
If $S$ does not contain a bipartite component then  $|F_{S,\tvec}(\thetavec)|$ is  strictly less than $1$ at any other point of $U_n(\pi)$ because 
there will  be at least one pair  $jk\in S$ such that $e^{i(\theta_j+\theta_k)} \neq  1$.  
Since $\tvec$ is a degree sequence, we have that $t_1+\cdots + t_n$ is even.
Then the contributions of neighbourhoods of $(0,\ldots,0)$ and  $(\pm \pi, \ldots, \pm \pi)$ to the integral  \eqref{integral} are 
identical because $F_{S,\tvec}(\thetavec)$ is $2\pi$-periodic with respect to each component of $\thetavec$ and 
\begin{equation}\label{symmetry} 
 F_{S, \tvec} (\theta_1 + \pi, \ldots, \theta_n+\pi)
  =  e^{i(t_1+\cdots+t_n)\pi} F_{S, \tvec} (\theta_1, \ldots, \theta_n)  = F_{S, \tvec} (\theta_1, \ldots, \theta_n). 
\end{equation} 
Thus, we can focus on estimates around the origin and then multiply by 2.

By Taylor's theorem, for  $a \in [0,1]$ and $x \in [-\pi/4, \pi/4]$,  we have
\begin{align*}
1+a(e^{i  x} -1)=  \exp\Bigl( i ax &- \dfrac12 a(1-a) x^2  
- \dfrac16 i  a(1-a)(1-2a) x^3 \\ &+\dfrac{1}{24}  a(1-a) (1-6a+6a^2) x^4 +O(x^5)\Bigr).
\end{align*}
Using this to expand the multipliers of   $F_{S,\tvec}(\thetavec)$, 
we find that
\begin{equation}\label{F-approx}
\begin{aligned}
		F_{S, \tvec}(\thetavec) =  \exp
		\Bigl(
		&- i\sum_{j=1}^n \theta_j t_j + i  \sum_{jk \in S} \lambda_{jk} (\theta_j + \theta_k) \\
		&-  \thetavec\trans Q \thetavec + u(\thetavec) + i v(\thetavec) + O\bigl(\|\thetavec\|_\infty^5 |E(S)|\bigr)
		  \Bigr).
\end{aligned}
\end{equation}
where the $n \times n$ symmetric matrix  $Q$  is defined by
\begin{equation}\label{def_A}
   \thetavec\trans Q \thetavec = \dfrac{1}{2} \sum_{jk \in S} \lambda_{jk}(1-\lambda_{jk})(\theta_{j}+\theta_k)^2
\end{equation}
and the multivariable polynomials $u$ and $v$ are defined by
\begin{equation}\label{def_polynomials}
\begin{aligned}
           u(\thetavec) &=  \dfrac{1}{24}\sum_{jk \in S} \lambda_{jk}(1-\lambda_{jk}) (1-6\lambda_{jk} + 6\lambda_{jk}^2 ) (\theta_{j}+\theta_k)^4,\\
		v(\thetavec) &=  \dfrac{1}{6}\sum_{jk \in S} \lambda_{jk}(1-\lambda_{jk}) (1-2\lambda_{jk}) (\theta_{j}+\theta_k)^3.
\end{aligned}
\end{equation}
Observe that 
$
\thetavec\trans Q \thetavec  \geq 0,
$   
so $Q$ is a positive semidefinite matrix.  Moreover, it is positive definite if  $S$ does not contain a bipartite component.

The optimal choice for $\betavec$ is such that the linear part in \eqref{F-approx} disappears, which corresponds to the case when  
our contours in the complex  plane pass through the saddle point.   Thus, we get the following system of equations:  
\begin{equation}\label{system}
  t_j =  \sum_{k:  jk \in S} \lambda_{jk} = \sum_{k:  jk \in S} \frac{e^{\beta_j+\beta_k}}{1+ e^{\beta_j + \beta_k}}  \qquad
  \text{ for all $1\leq j\leq  n$.}
\end{equation}
For the case $S=K_n$, the existence  and the uniqueness of  the solution was studied in 
\cite{BH2013, CDS2011,Rinaldo2013}: the necessary and sufficient condition is that $\tvec$ lies in the interior of the polytope defined by the Erd\H{o}s-Gallai inequalities. 
When $S$ is the complete graph, it is also known that system \eqref{system} is equivalent to
(i) maximisation of the likelihood  with respect to the parameters of the $\beta$-model given  observations of the degrees (ii) finding  the random model with independent adjacencies and given expected degrees that maximises the entropy.  
Unfortunately, analogs of these results are not available for general $S$ even though the methods used in the literature will certainly carry over.   Since such results are not needed for our purposes here,  we leave these questions for   a subsequent paper.

Denote  
\begin{equation}\label{def:lambda}
   \lambda  = \frac{\sum_{ jk \in S} \lambda_{jk} }{|E(S)| } 
 \end{equation}
If system \eqref{system} holds then we have
 $\lambda = \dfrac{t_1+\cdots +t_n}{2 |E(S)|}$, which is  the relative density of a $\tvec$-factor in $S$.
We are ready to state our main result of this section.

\begin{theorem}\label{thm:enum}
 Let $\eps, \gamma$ and $c$ be fixed positive constants. 
   Suppose a  graph $S$ on $n$ vertices and degree sequence $\tvec$ 
   satisfy the following assumptions: 
   \begin{itemize}
      \item[(A1)] for  any two vertices $j$ and $k$,  we have
   	 \[ 
               \dfrac{\gamma \varDelta^2(S)}{n} \leq    
               \left|\{\ell \st j\ell \in S \text{ and } k\ell \in S\}\right|  \leq \dfrac{\varDelta^2(S)}{\gamma n};
   	   \]  
   	\item[(A2)] there exists a solution $\betavec$ of system \eqref{system} such that $\range{\betavec} \leq c$;

   	\item[(A3)]  $\lambda(1-\lambda) \varDelta(S) \gg \dfrac{n}{\log n}$.
      	   \end{itemize}
 Let $\X$ be a random variable with the normal density 
   $\pi^{-n/2} |Q|^{1/2} e^{-\xvec\trans  Q \xvec}$. Then, 
   \[
      N(S,\tvec) =   \frac{ 2    \prod_{jk\in S}  (1 + e^{\beta_j+\beta_k})
                  } 
       {
         (2 \pi)^{n/2} |Q|^{1/2} \,   \prod_{j=1}^n e^{t_j \beta_j} 
       }
       \exp\left( \E u(\X) - \dfrac12 \E v(\X)^2 + O(n^{-1/2 + \eps})\right),
   \]
   where the constant implicit in $O(\,)$ depends on $\gamma, \eps$ and $c$ only.
\end{theorem}

There is a vast literature on asymptotic enumeration of 
dense subgraphs with given degrees in the case when $S$ is the complete  graph or not far from it,
see, for example,  \cite{MW1990, McKay2011, BH2013, IM} and references therein.  An important advantage of Theorem \ref{thm:enum} with respect to the previous results is that it allows $S$ to be essentially different from $K_n$ and it holds  for a very wide range of degrees. Theorem \ref{thm:enum} follows immediately from equations \eqref{integral}, \eqref{symmetry},  Lemma \ref{L:inside} and Corollary 
\ref{Cor:outside}. 

%%%%%%%%%%%%%%%%%%%%%%%%%%%%%%%%%%%%%%%%%%%%%%%%%%%%%%%%%%%%%%%%%%%

\nicebreak
\subsection{The integral in the critical regions}\label{S:inside}

For given $S$ and $\tvec$, denote 
\[
 \varLambda = \lambda(1-\lambda)\qquad  \text{and} \qquad \varDelta = \varDelta(S).
 \]  In the following, we always assume that 
$\varLambda \varDelta \gg n/\log n$ which is the assumption (A3) of Theorem \ref{thm:enum}. 
Let $\eps$ be a fixed positive constant  required to be sufficiently small in several places of the argument. In particular, we need that 
 \[
    \eta =  \dfrac{ n^{\eps}}{ (\varLambda \varDelta)^{1/2}} = o(1).
 \] 
 Given $x \in \Reals$, define 
 \[
    |x|_{2\pi} = \min\{|y| \,:\, y \equiv x  \mod  2\pi\}. 
 \]
 It is easily seen that $|\cdot|_{2\pi}$ is a seminorm on $\Reals$ that induces a norm on $\Reals/(2\pi)$, the real numbers modulo $2\pi$.
Our critical regions are 
\[
 \calB_0 = U_n(\eta) 
  \qquad \text{and} \qquad
  \calB_{\pi}=
	\{\thetavec \in \Reals^n \,:\, |\theta_j -\pi|_{2\pi} \leq \eta \text { for all  } j\},
\] 
As explained above (see \eqref{symmetry}), the contributions of these two regions to the integral in \eqref{integral}  are identical so we can focus on~$\calB_0$. From  \eqref{F-approx}, we have
\begin{equation}\label{start_non-bi}
	 \int_{\calB_0} F_{S, \tvec}(\thetavec) d\thetavec =  
	  \int_{U_n(\eta)} e^{-\thetavec\trans  Q  \thetavec + u(\thetavec)-iv(\thetavec) + h(\thetavec)} d \thetavec,
\end{equation}
where $h(\thetavec) = O(n^{-1/2 + 6\eps})$ uniformly for  $\thetavec \in \calB_0$.
A general theory on the estimation of such integrals  was developed in \cite{IM}, based on the second-order approximation of complex martingales. We will apply  the tools from  \cite{IM} here and, for the reader's convenience, also quote them  in the appendix, see Section \ref{S:mother}.

We will need the following bounds. 
\begin{lemma}\label{matrix_non-bi}
    If $\range{\betavec} \leq c$ for some fixed $c>0$,  then 
	 \begin{itemize}
	  \item[(a)]   uniformly over all $jk\in S$,
	    $ \lambda_{jk} =  \Theta(\lambda)$ and $1 - \lambda_{jk} = \Theta(1-\lambda)$, where 
$\lambda$ is defined in \eqref{def:lambda}.
	 \end{itemize}
	  		  	Furthermore, suppose  $\varDelta = \Omega(n^{1/2})$ and 
	  		  	assumption  (A1) of Theorem \ref{thm:enum} holds. 
	  		  	Then $Q$ is positive definite and the following hold.
	  \begin{itemize}
	  		\item[(b)]  If    $Q^{-1} = (\sigma_{jk})$, then 
	  		   		  	$
	  		\sigma_{jk} = 
	  		\begin{cases}
	  			 \Theta\left(\dfrac{1 }{\varLambda\varDelta}\right), & \text{if } j=k;\\[1ex]
	  			 O\left(\dfrac{1 }{\varLambda\varDelta^2 }\right), & \text{if } jk \in S;\\[1ex]
	  			  O\left(\dfrac{1 }{\varLambda\varDelta n}\right), & \text{otherwise}.
	  		\end{cases}
	  	$
     		\item[(c)]    There exists a real matrix $T$ such that
     $T\trans  Q T=I$ and
     \[
      \|T\|_1, \|T\|_{\infty} = O\((\varLambda \varDelta)^{-1/2}\),  \qquad
       \|T^{-1}\|_1, \|T^{-1}\|_{\infty} =  O\((\varLambda \varDelta)^{1/2}\).
      \]
     \end{itemize}
\end{lemma}

\begin{proof}
Observe that 
$1 \leq \frac{1+e^y}{1+ e^x}  \leq e^{y-x}$ for any real $x \leq y$.  Since all $\beta_j+\beta_k$  and 
$\beta_{j'}+\beta_{k'}$ are at most $2c$ apart, this implies that
$\dfrac{\lambda_{jk}}{\lambda_{j'k'}} = \Theta(1)$ and $\dfrac{1-\lambda_{jk}}{1-\lambda_{j'k'}} = \Theta(1)$ for all $jk, j'k' \in S$. Recalling  the definition \eqref{def:lambda}, we have proved~(a).

Note that assumption  (A1) of Theorem \ref{thm:enum}  implies that $S$ is a connected non-bipartite graph. Thus, $Q$ is positive definite. Parts (b) and (c)  follow from Lemma \ref{A:weighted} (see Section~\ref{sec:appendix}) applied to the scaled matrix $Q/\varLambda$.
The condition $\range{\betavec}\le c$ for some constant $c>0$
in Lemma~\ref{matrix_non-bi} will be verified in Section~\ref{s:prescribed}.
In the following, we will assume that that condition holds.
\end{proof}

We are ready to establish asymptotic estimates for the critical region $\calB_0$. 
Note that  in the next lemma we allow
the components of $\tvec$ to be non-integers.
\begin{lemma}\label{L:inside}
	 Suppose a graph $S$ and a real vector $\tvec \in \Reals^n$  
	 satisfy assumptions (A1)-(A3) of Theorem \ref{thm:enum}.
	 	  Then, for any sufficiently small  fixed $\eps>0$, we have 
	   \[
	   \int_{U_n(\eta)} F_{S, \tvec}(\thetavec) d\thetavec
	     = \frac{
	     \pi^{n/2}
                  } 
       {
         |Q|^{1/2}
       }
       \exp\Bigl( \E u(\X) - \dfrac12 \E v^2(\X) +  O(n^{-1/2 + 13\eps})\Bigr),
	   \]  
	    where  $\X$ is  a random vector in $\Reals^n$ with the normal density 
   $\pi^{-n/2} |Q|^{1/2} e^{-\xvec\trans  Q \xvec}$.  Furthermore,
     \[
     	 \E u(\X) = O\Bigl(\dfrac{n}{\varLambda\varDelta}\Bigr),  
	 \qquad  \E v^2(\X)   = O\Bigl(\dfrac{n}{\varLambda\varDelta}\Bigr).
     \]
      and
     \[
     			  \int_{U_n(\Theta(\eta))} |F_{S, \tvec}(\thetavec)| d\thetavec
	     =   \frac{
	     \pi^{n/2}
                  } 
       {
         |Q|^{1/2}
       }e^{O\left(\tfrac{n}{\varLambda\varDelta}\right)}.
     \]
\end{lemma}
\begin{proof}
  The proof is based on \cite[Theorem 4.4]{IM} which is quoted as  Theorem \ref{thm:integral} for the reader's convenience. Let 
\[
   \varOmega = U_n(\eta), \qquad  f(\thetavec) = u(\thetavec) -i v(\thetavec), \qquad g(\thetavec) = u(\thetavec).
\]
From Lemma \ref{matrix_non-bi}, we know that $Q$ is positive definite.
 Let $T$ be the matrix from  Lemma \ref{matrix_non-bi}(c). Define
 \[
 	 \rho_1 =   \|T\|_\infty^{-1}\eta, \qquad \rho_2 = \|T^{-1}\|_\infty \eta.
 \]
Then we have   $
       U_n(   \rho_1 ) \subseteq T^{-1}( \varOmega) \subseteq U_n( \rho_2)$
and
     \[ 
      \rho_2   \geq \rho_1 = \Omega\left( \dfrac{n^{\eps} }{(\varLambda \varDelta)^{1/2}} \, 
         (\varLambda\varDelta)^{1/2}\right)  =          \Omega( n^{\eps}).
      \] 
      Thus,     $\rho_1$  and $\rho_2$  satisfy assumption (a) of Theorem \ref{thm:integral}.  Similarly, observe
    $\rho_2 = O(n^{\eps})$.
    
Next, we estimate the partial derivatives of $f(\xvec)$.  Recalling the definitions of $u$ and $v$ from \eqref{def_polynomials} and using 
Lemma \ref{matrix_non-bi}(a),
we get that, provided $\|\thetavec\|_\infty \leq 1$
\begin{align*}
	\dfrac{\partial f}{\partial \theta_j}(\thetavec) &=      \dfrac{1}{6}\sum_{k:  jk \in S} \lambda_{jk}(1-\lambda_{jk}) (1-6\lambda_{jk} + 6\lambda_{jk}^2 ) (\theta_{j}+\theta_k)^3  \\
		 &   {~~}- \dfrac{i}{2}\sum_{k:   jk \in S} \lambda_{jk}(1-\lambda_{jk}) (1-2\lambda_{jk}) (\theta_{j}+\theta_k)^2 
		 = O\left(\varLambda \varDelta  \|\thetavec\|_\infty^2\right).
\end{align*}
  and, if $ jk \in S$,
  \begin{align*}
  	\dfrac{\partial^2 f}{\partial \theta_j \partial\theta_k} (\thetavec) &=   \dfrac{1}{2} \lambda_{jk}(1-\lambda_{jk}) (1-6\lambda_{jk} + 6\lambda_{jk}^2 ) (\theta_{j}+\theta_k)^2  \\
		 &  {~~}-  i \lambda_{jk}(1-\lambda_{jk}) (1-2\lambda_{jk}) (\theta_{j}+\theta_k) 
		 = O(\varLambda \|\thetavec\|_\infty ).
  \end{align*}  
  Again using  Lemma \ref{matrix_non-bi}(c), we find that
  assumption (b) of Theorem   \ref{thm:integral} holds with $\phi_1 = n^{-1/6 + 4\eps}$. Exactly the same calculation shows (c)(ii) with $\phi_2 =n^{-1/6 + 4\eps}$. Assumption (d) also holds because $u$ and $v$ are polynomials. Applying Theorem   \ref{thm:integral} to the integral of \eqref{start_non-bi}, we obtain that
 \begin{equation}\label{int_inter}
	   \int_{U_n(\eta)} F_{S, \tvec}(\thetavec) d\thetavec
	     = (1+K)\frac{\pi^{n/2}
                  } 
       {
         |Q|^{1/2}
       }
       \exp\left( \E f(\X) + \dfrac12 \E (f(\X) - \E f(\X))^2) \right),
	   \end{equation}    
	   where 
	$
	   	K =   O( n^{-1/2+ 12\eps}) e^{\frac12 \Var v(\X) }.
	 $
	 Similarly, using \eqref{F-approx} and Theorem   \ref{thm:integral},  we get
	 \begin{align*}
	 		 \int_{U_n(\Theta(\eta))}|F_{S, \tvec}(\thetavec) |d\thetavec
	 		 &=  \bigl(1+ O(n^{-1/2 + 6\eps})\bigr)
	  \int_{U_n(\Theta(\eta) } e^{-\thetavec\trans  Q  \thetavec + u(\thetavec) } d \thetavec\\
	 		 &= \bigl(1+ O( n^{-1/2 + 12\eps})\bigr) 
	 		 \frac{\pi^{n/2}
                  } 
       {
         |Q|^{1/2}
       }
       \exp\left( \E u(\X)  - \dfrac12 \Var u(\X)) \right).
	 \end{align*}
	   
  Next, we need to estimate some moments of $u(\X)$ and $v(\X)$.  Let  $\varSigma = (\sigma_{jk, \ell m})$ denote the
   covariance matrix of the variables $X_j + X_k$ for  $jk\in S$:
   \begin{equation}\label{sigma_def}
   	\sigma_{jk,\ell m} = \Cov (X_j+X_k, X_\ell+X_m).
   \end{equation}
   Since $\X$ is a gaussian vector with density $\pi^{-n/2} |Q|^{1/2} e^{-\xvec\trans Q\xvec}$, 
    the values of $\Cov(X_j,X_k)$ equal the corresponding entries of $(2Q)^{-1}$.  Using the bounds of  Lemma \ref{matrix_non-bi}(b), we find that 
    \begin{equation}
    \label{sigma_bounds}
    	\sigma_{jk,\ell m} =
    	 \begin{cases} 
    	     O\left( \dfrac{1}{\varLambda\varDelta}\right), & \text{if }  \{j,k\}\cap\{\ell,m\} \neq \emptyset;
    	       \\[1ex]
    	        O\left( \dfrac{1}{\varLambda\varDelta^2}\right), & \text{if }  
    	        \{j,k\}\cap\{\ell,m\} = \emptyset  \text{ and } \{j\ell,jm,k\ell, km\}\cap S \neq \emptyset; 
    	       \\[1ex]
    	      O\left( \dfrac{1}{n\varLambda\varDelta}\right), & \text{otherwise.} 
    	 \end{cases}
    \end{equation}
    The expectation of a polynomial of odd degree is  zero (due to the symmetry of the distribution) so 
    $  \Cov(u(\X),v(\X)) = \E v (\X)    = 0$.      
  The following are special cases of  Isserlis' theorem (see  \cite{Isserlis}), which is also known as Wick's formula in quantum field theory:
\begin{equation*}
    \begin{aligned}
     \E(X_j + X_k)^4 &= 3 \sigma_{jk,jk}^2,  \qquad  \E(X_j + X_k)^6 = 15 \sigma_{jk,jk}^3\\
     \E(X_j+X_k)^3(X_\ell+X_m)^3&= 9 \sigma_{jk, jk} \, \sigma_{\ell m,\ell m} \,\sigma_{jk,\ell m} + 6 \sigma_{jk, \ell m}^3,\\
     \E(X_j+X_k)^4(X_\ell+X_m)^4 &= 9  \sigma_{jk,jk}^2  \sigma_{\ell m, \ell m}^2  
     + 72\, \sigma_{jk,jk}\,  \sigma_{\ell m, \ell m}\, \sigma_{jk,lm}^2 + 24 \sigma_{jk,lm}^4.
     \end{aligned} 
  \end{equation*}
 Recalling \eqref{def_polynomials} and using \eqref{sigma_bounds}, we obtain  that
 \begin{equation}\label{E_u}
 	\E u (\X) = \dfrac{1}{8} \sum_{jk \in S} \lambda_{jk} (1-\lambda_{jk})
 	 (1- 6\lambda_{jk}+6\lambda_{jk}^2)
 	 \sigma_{jk,jk}^2 = O\left(\dfrac{n}{\varLambda\varDelta}\right).
 \end{equation}
 Similarly as above, we derive that
 
 Observe that $|U_{jk}| \leq 2 \varDelta$.  Then, similarly to above, we get that
 \begin{equation}\label{Var_v}
 \begin{aligned}
        \Var  v (\X)= \E v^2(\X)  &=
        O\biggl(
        \varLambda^2\sum_{jk\in S}\sum_{\ell m\in S} (|\sigma_{jk,jk}\sigma_{\ell m,\ell m}\sigma_{jk,\ell m}|+|\sigma^3_{jk,\ell m}|)
        \biggr)\\
        &= O\biggl(
            \varLambda^2 \sum_{jk \in S}\sum_{\ell m \in S}
\dfrac{|\sigma_{jk,\ell m}|}{(\varLambda \varDelta)^2}
        \biggr)
        \\&=
         O\left( \dfrac{1}{\varDelta^2}\right) \left( \dfrac{n\varDelta^2 }{\varLambda \varDelta} +
         \dfrac{n\varDelta^3 }{\varLambda \varDelta^2} + \dfrac{n^2\varDelta^2 }{ n\varLambda \varDelta} \right)
         =
        O\left(\dfrac{n}{\varLambda\varDelta}\right)
 \end{aligned}
 \end{equation}
 and
 \begin{align*}
        \Var  u (\X) &= \E u^2(\X) -(\E u(\X))^2\\
        &= O\biggl(
        \varLambda^2\sum_{jk\in S}\sum_{\ell m\in S} (|\sigma_{jk,jk}\sigma_{\ell m,\ell m}\sigma_{jk,\ell m}^2|+\sigma_{jk,\ell m}^4)
        \biggr) \\
        &= O\biggl(
            \varLambda^2 \sum_{jk \in S}\sum_{\ell m \in S}
\dfrac{|\sigma_{jk,\ell m}|^2}{(\varLambda \varDelta)^2}
        \biggr)
        \\ &=O\Bigl( \dfrac{1}{\varDelta^2}\Bigr)
         \left(\dfrac{n \varDelta^2}{(\varLambda\varDelta)^2} +
          \dfrac{n\varDelta^3 }{(\varLambda\varDelta^2)^2} +
         \dfrac{n^2 \varDelta^2}{(n \varLambda \varDelta)^2}\right) = o\Bigl(\dfrac{\log^2 n}{n} \Bigr),
 \end{align*}
noting that the leading term containing $\sigma_{jk,jk}^2\sigma_{\ell m,\ell m}^2$ appears in both $\E u^2(\X)$ and $(\E u(\X))^2$ and gets cancelled from the subtraction.
   Substituting these bounds  into \eqref{int_inter} and bounding $e^{\frac12 \Var v(\X))} = e^{o(\log n)} = n^{o(1)}$, we complete the proof.
\end{proof}

%%%%%%

\nicebreak
\subsection{Estimates outside of the critical regions}\label{S:outside}

In this section, we show that    the contribution to the integral \eqref{integral} of the remaining region $\calB = U_n(\pi) -  \calB_0- \calB_\pi$ is negligible, where  the critical regions  $ \calB_0$  and $ \calB_{\pi}$ are defined in Section \ref{S:inside}.
Observe  that 
\[
	|F_{S,\tvec}(\thetavec)| =	 \prod_{j k \in S}\, \bigl| 1+\lambda_{jk}(e^{i(\theta_j+\theta_k)}-1)\bigr| 
\]
 depends on $S$ and  $(\lambda_{jk})$  only but does not depend on $\tvec$.
To bound the factors of $|F_{S,\tvec}(\thetavec)|$, we use the following  inequality, whose uninteresting proof we omit.
\begin{lemma}\label{boring}
 For $x\in \Reals$ and $a\in[0,1]$, we have
  $|1 + a(e^{ix}-1)| 
                  \le e^{-\frac15 a(1-a)\abs{x}_{2\pi}^2 }$.
  \end{lemma}
  Throughout this section, including the lemma statements, we always assume that the assumptions of Theorem \ref{thm:enum} hold. Recall that
  \[
  	\calB_0 = U_n(\eta), \qquad \eta = \dfrac{n^{\eps}}{(\varLambda \varDelta)^{1/2}}.
  \]
Lemma \ref{matrix_non-bi}(a,b) implies that all the eigenvalues of $Q$ 
are   $\Theta\left(\varLambda \varDelta\right)$ 
(by bounding the 1-norms of $Q$ and $Q^{-1}$). From Lemma \ref{L:inside}, we find that 
\begin{equation}\label{J0est} 
   J_0 := \int_{\calB_0} F_{S, \tvec}(\thetavec) d\thetavec
       = \pi^{n/2} |Q|^{-1/2} e^{O\left(\tfrac{n}{\varLambda \varDelta}\right)}
	     \ge \exp\(-\dfrac12n\log n+O(\log n)\).  
\end{equation}

As a first step, we eliminate the case when many components of  $\thetavec \in U_n(\pi)$ lie sufficiently  far from $0$ and $\pm \pi$.  Define
\[\calB'=\bigl\{ \thetavec\in U_n(\pi) :
  \text{~more than $\tfrac12 n^{1-\eps}$ components $\theta_j$
   satisfy }\allowbreak
    \text{$\eta/2\le\abs{\theta_j}_{2\pi}\le \pi-\eta/2$} \bigr\}.\]

The following lemma depends on a technical lemma (Lemma \ref{A:neighbours}) which we present in Section~\ref{sec:appendix}.

\begin{lemma}\label{l:firstcut} We have
    \[
    \int_{\calB'} |F_{S,\dvec}(\thetavec)| \,d\thetavec
     = e^{-\Omega(n^{1+\eps})} J_0.
    \]
\end{lemma}
\begin{proof}
 Without loss of generality, at least $\frac14 n^{1-\eps}$
 components $\theta_j$ lie in $[\eta/2,\pi-\eta/2]$.
 Denote $U = \{ j  : \theta_j\in [\eta/2,\pi-\eta/2] \}$.
  Let's estimate the number $N_T(U)$ of triangles $\{j,k,\ell\}$ (i.e. $jk,j\ell,k\ell \in S$) 
such that $\{j,k,\ell\}\cap U \neq \emptyset$. 
  Using Lemma \ref{A:neighbours}(a), 
  we find that the degree of any vertex of $U$ is at least $\gamma \varDelta$. 
  For any $jk \in S$ and $\{j,k\} \cap U \neq \emptyset $
   there are at least $\dfrac{\gamma \varDelta^2}{n}$ common neighbours each of which gives rise to a
    triangle contributing to $N_T(U)$. Since every triangle is counted at most $3$ times, we get that 
    \[	
    	N_T(U) \geq   \dfrac{ \gamma \varDelta  |U|}{2} \cdot  \dfrac{ \gamma \varDelta^2}{3n} =   \dfrac{ \gamma^2 \varDelta^3 |U|}{6n}.
    \]
  For each such triangle $\{j,k,\ell\}$ that $j\in U$, observe that 
  \[
	|\theta_j + \theta_k|_{2\pi} + |\theta_k + \theta_\ell|_{2\pi} + |\theta_\ell + \theta_j|_{2\pi}
	\geq |\theta_j +\theta_k - \theta_k-  \theta_\ell + \theta_\ell + \theta_j|_{2\pi}  \geq \eta.
  \]
  Therefore, we can mark one edge $j'k'$ from this triangle such that 
 $|\theta_{j'}+\theta_{k'}|_{2\pi}\ge \eta/3$. Repeating this argument 
  for  all such triangles and observing that any edge is present in at most $\dfrac{\varDelta^2}{\gamma n}$ triangles, we show that 
  at least   $\gamma^3 \varDelta  |U|/6$ edges were marked.   Using Lemma \ref{matrix_non-bi}(a) and  Lemma \ref{boring}, we get that
    \[
    		 |F_{S,\tvec}(\thetavec)|   \leq e^{- \Omega\left( \varLambda \varDelta  |U|   \eta^2 \right)} = 
    		e^{-\Omega(n^{1+\eps})}.
  \]
   Multiplying by the volume of $\calB'$, which is less than $(2\pi)^n$,
 and comparing with \eqref{J0est}, completes the proof.
\end{proof}

If Lemma~\ref{l:firstcut} doesn't apply, we have at least $n-\frac12 n^{1-\eps}$
components of $\thetavec$ lying in neighbourhoods of $0$ and $\pm\pi$.
Next we will use a similar argument to show that most of these components
lie in one of those two intervals (on a circle). 
Define
\begin{align*}
 \calB''=\bigl\{ \thetavec\in U_n(\pi) \setminus \calB' : 
  &\text{ $|\theta_j| \leq \eta/2$ holds for  more than $n^{2\eps}$ components $\theta_j$ and }   \\
   &~~\text{ $|\theta_j-\pi|_{2\pi} \leq \eta/2$ holds for  more than $n^{2\eps}$ components $\theta_j$} \bigr\}.
   \end{align*}

\begin{lemma}\label{l:secondcut}
We have
    \[
    \int_{\calB''} |F_{S,\dvec}(\thetavec)| \,d\thetavec
      = e^{-\Omega(n^{1+\eps})} J_0.
    \]
\end{lemma}
\begin{proof}
  Let $U_1=\{ j : |\theta_j| \leq \eta/2\}$ and
  $U_2=\{ j  : |\theta_j-\pi|_{2\pi} \leq \eta/2\}$. 
  Since $\thetavec\notin \calB'$, we have $|U_1|+|U_2| \geq n - \frac12 n^{1-\eps}$.
  For $j \in U_1$, $k \in U_2$ and any $\ell$ such that $j\ell, k\ell \in S$, we have
  \[
  	|\theta_j + \theta_\ell|_{2\pi} +   	|\theta_k + \theta_\ell|_{2\pi}
  	\geq |\theta_j + \theta_\ell - \theta_k - \theta_\ell|_{2\pi} \geq \pi -\eta.
  \]
  Thus, we can mark  some $j'k' \in \{j\ell, k\ell\}$ that  $|\theta_j+ \theta_k|_{2\pi} = \Omega(1)$.
   By the assumptions, the number of choices for $(j,k,\ell)$ is at least $|U_1|\, |U_2| \dfrac{\gamma \varDelta^2}{n}$. 
   Dividing by $2 \varDelta$ to compensate for over-counting, we get that at least  $|U_1|\, |U_2|   \dfrac{\gamma \varDelta}{2n}$
   edges were marked. Using Lemma \ref{matrix_non-bi}(a) and  Lemma \ref{boring}, we find that
   \[
   	|F_{S,\tvec}(\thetavec)|  = e^{- \Omega(|U_1| |U_2| \varLambda  \varDelta/n )}
   	 = e^{-\Omega(n^{1+2\eps}/\log n)} = e^{-\Omega(n^{1+\eps})}.
   \]
  The proof now follows the same line as in the previous lemma.
\end{proof}

Since adding $\pi$ to each component is a symmetry, see \eqref{symmetry}, we can
now assume that at least $n-n^{1-\eps}$ components of
$\thetavec$ lie in $[-\eta/2,\eta/2]$.  If $\thetavec \notin \calB_0$ then we should have some components 
$|\theta_j| > \eta$.   Let $\calB(m)$ denote the region of $\thetavec \in \calB \setminus (\calB' \cup \calB'')$ 
such  that exactly $m$ components of $\thetavec$ lie outside of  $[-\eta,\eta]$, where 
  $1\leq m \leq n^{1-\eps}$. Let 
\[
	J(m) =  \int_{\calB(m)} |F_{S,\tvec}(\thetavec)| \,d\thetavec.
\]	
For notational simplicity, we first prove a bound for the integral over the region $\calB^*(m) \subset \calB(m)$, where the set  of $m$ components of $\thetavec$ lying   outside of  $[-\eta,\eta]$ is exactly $\{\theta_1,\ldots,\theta_m\}$.  Our bound will be actually independent of  this choice of  $m$  components  so then we just need to multiply it by $\binom{n}{m} \leq n^m$.

  Note that  
  \begin{equation}\label{m:bound}
       m \leq n^{1-\eps} = o\left( \dfrac{n}{\log^2 n}\right) = o\left(\dfrac{\varDelta^2}{n}\right) = o(\varDelta). 
       \end{equation}
Take any $j\leq m$. Using Lemma \ref{A:neighbours}(a), we find that 
 at least $\gamma \varDelta - n^{1-\eps}= \Theta(\varDelta)$ vertices $k$ such that $jk\in S$ and 
 $|\theta_k|_{2\pi} \leq \eta/2$.  For such $k$, we have 
  $|\theta_j+ \theta_k|_{2\pi}\geq \eta/2 $.  Similarly as before,  by   Lemma~\ref{boring},  for $\thetavec \in \calB^*(m) $,
\[
	\prod_{j=1}^m \prod_{k=m+1}^n \bigl|1+\lambda_{jk}(e^{i(\theta_j+\theta_k)}-1)\bigr|
	  =   e^{-\Omega(m \varLambda \varDelta  \eta^2)} =e^{-\Omega(m n^{2\eps})}.
\] 
 Thus, we can bound 
 \[
 	 \int_{\calB^*(m)} |F_{S,\tvec}(\thetavec)|  d\thetavec
 	  \leq   \int_{U_{m}(\pi)} e^{-\Omega(mn^{2\eps})}
 	   \left(  \int_{U_{n-m}(\eta)} |F_{S',\tvec'}(\thetavec^1)| d \thetavec^{1}  \right)   d \thetavec^{2} , 
 \]	
 where $\thetavec^{1}\in \Reals^{n-m}$, $\thetavec^{2} \in \Reals^m$ and $S'$ is obtained from $S$ by deletion of the first $m$ vertices.  Recall that $|F_{S',\tvec'}(\thetavec^1)|$ does not depend on $\tvec'$, but we   define it anyway by
  \[
 	t_j' = \sum_{j : jk\in S'} \lambda_{jk} \text{ for all $j$.}
 \]
 Using \eqref{m:bound}, we get that $S'$ and $\tvec'$ satisfy all the assumptions of Lemma \ref{L:inside}. Thus,
\[
	\int_{U_{n-m}(\eta)} |F_{S',\tvec'}(\thetavec^1)| d \thetavec^{1} = \frac{\pi^{(n-m)/2}}{ |Q'|^{1/2} } 
	e^{O\left(\frac{n}{\varLambda\varDelta}\right)},
\]	
where $Q'$ is the matrix of \eqref{def_A} for the graph $S'$ and $(\lambda_{jk})_{jk \in S'}$.
Applying Lemma \ref{A:weighted}(d)  (see Section~\ref{sec:appendix})
$m$ times for the scaled matrix $Q/\varLambda$, we find that 
\[
	|Q|/  |Q'|  =  \(\varLambda \varDelta)^m e^{O(m)}
\]
 Allowing $n^m$ for the choice of the set of $m$  big components and  using  \eqref{J0est},  \eqref{m:bound}, we obtain that
 \begin{align*}
 	J(m) & \leq n^m  e^{-\Omega(mn^{2\eps})}
 	    \(\varLambda \varDelta)^{m/2}  e^{O(m)}
 	  \frac{\pi^{(n-m)/2}}
 	  { |Q|^{1/2} } 
 	  e^{O\left(\frac{n}{\varLambda\varDelta}\right)}  
 	  = e^{-\Omega(m n^{2\eps})} J_0.
 \end{align*}
 Summing over $m$ and  multiplying by $2$ for the symmetry of $(0,\ldots,0)$ and $(\pm\pi,\ldots\pm\pi)$, we find that 
 \[
 	\int_{ \calB \setminus (\calB' \cup \calB'')} 
 	|F_{S,\tvec}(\thetavec)|  d\thetavec
 	\leq 2 \sum_{m=1}^{n^{1-\eps}}  J(m)    =  e^{-\Omega(n^{2\eps})} J_0.
 \]
Using Lemma \ref{l:firstcut} and Lemma \ref{l:secondcut}, we conclude the following.
\begin{corollary}\label{Cor:outside}
	Under the assumptions of Theorem \ref{thm:enum} and for sufficiently small 
	$\eps$, 
	\[
	\int_{ \calB } 
 	|F_{S,\tvec}(\thetavec)|  d\thetavec =  e^{-\Omega(n^{2\eps})} 
 	\int_{ \calB_0}  F_{S,\tvec}(\thetavec)  d\thetavec.
 	\]
\end{corollary}
%%%%%%%%%%%%%%%%%%%%%%%%%%%%%%%%%%%%%%%%%%%%%%%%%%%%%%%%%

\nicebreak
\subsection{Prescribed edges in  random factors}\label{s:prescribed}

Finally, we verify the assumption $\range{\betavec} \leq c$ for some fixed $c>0$ in Lemma~\ref{matrix_non-bi}.
We do so by establishing a deep connection between $S_{\tvec} $
(a uniform random element of the set of $\tvec$-factors of~$S$)  
and  the corresponding  $\beta$-model:  for each set of vertex pairs,
the probabilities in each model for them to all be edges are
asymptotically the same.

The following lemma will be useful for  investigating system \eqref{system}.
\begin{lemma}\label{l:Kowa}
	Let $\rvec : \Reals^n \to \Reals^n$, $\delta>0$,
	and   $U  =  \{\xvec\in \Reals^n \st \|\xvec - \xvec^{(0)} \| \leq \delta \norm{\rvec(\xvec^{(0)})} \}$
	and $\xvec^{(0)}\in \Reals^n$,  where 
	$\|\cdot\|$ is  any vector norm in $\Reals^n$.
	Assume that
	\[
	   \text{$\rvec$ is analytic in $U$} \qquad \text{and} \qquad \sup_{\xvec \in U} \|J^{-1}(\xvec)\| < \delta, 
	 \]
	 where $J$ denotes the Jacobian matrix   of $\rvec$ and  $\|\cdot\|$ stands for the induced matrix norm.
	Then there exists $\xvec^*\in U$ such that 
	$\rvec(\xvec^*) = \boldsymbol{0}$. 
\end{lemma}
 \begin{proof}
    Let $\yvec^{(0)} = \rvec(\xvec^{(0)})$ and note that $\xvec^{(0)}\in U$.
    If $\yvec^{(0)} =0$ there is nothing to prove so we may assume otherwise. 
 	Using the Cauchy-Kovalevskaya theorem,  define the curve $\xvec(t)$ by $\xvec(0) = \xvec^{(0)}$ and 
 	$\dfrac{d \xvec(t)}{dt} = -J^{-1}(\xvec(t))\yvec^{(0)}$.   
 	Note that $\xvec(t)$ remains in $U$ for $0\leq t \leq 1$, because
 	\[
 		\xvec(t) -\xvec(0)  = - \int_{0}^t  J^{-1}(\xvec(\tau)) \yvec^{(0)} d \tau.
 	\]
 	and  $\|\xvec(t) -\xvec(0)  \|  \leq t  \sup_{\xvec \in U} \|J^{-1}(\xvec) \yvec^{(0)} \|  
 	<  \delta \|\yvec^0\|$.
 	Observe  that $\dfrac{d\, \rvec(\xvec(t))}{d\, t} = -\yvec^{(0)}$. Therefore,
 	 	$\rvec(\xvec(t)) = (1-t)\yvec^{(0)} $.  Taking  $\xvec^* = \xvec(1)$  we complete the proof.
 \end{proof}

\begin{corollary}\label{cor:beta}
  Let $S$ satisfy assumption (A1) of  Lemma \ref{l:dense} and $\varDelta = \Omega(n^{1/2})$.   For 
  $\tvec \in \Reals^n$, let $\lambda = \dfrac{t_1+\cdots + t_n}{2 |E(S)|}$.  For $\betavec  \in \Reals^n$, define $\rvec(\betavec) = (r_1,\ldots,r_n)$   by
   \[	
    	r_j = r_j(\betavec) = -t_j + \sum_{k: jk\in S}\dfrac{e^{\beta_j+\beta_k}} {1 +e^{\beta_j+\beta_k}} \qquad \text{for all } j.
   \]
   Suppose, for some $\betavec^{(0)}$, we have  $\range{\betavec^{(0)}} \leq c$ and
    $\|\rvec(\betavec^{(0)})\|_\infty \ll \lambda(1-\lambda) \varDelta$. Then there exists a solution $\betavec^*$ of system 
   \eqref{system} such that
   \[
   	\|\betavec^* - \betavec^{(0)}\|_p =  O\left(\dfrac{\|\rvec(\betavec^{(0)})\|_p}{ \lambda(1-\lambda) \varDelta}\right), \qquad \text{ for any } p\in\{1,2,\infty\}.
      \]
\end{corollary}
\begin{proof}
	Observe that  
	\[
	  \frac{\partial}{\partial \beta_j} \left( \frac{e^{\beta_j+ \beta_j}}{1+e^{\beta_j+ \beta_j}}\right) 
	   = \frac{e^{\beta_j+ \beta_j} }{1+e^{\beta_j+ \beta_j} } \left(1 - \frac{e^{\beta_j+ \beta_j}}{1+e^{\beta_j+ \beta_j}}\right)
= \lambda_{jk}(1-\lambda_{jk}).
	 \]
	 Therefore, the Jacobian matrix $J(\betavec)$ of $\rvec(\betavec) $ coincides with $2A(\betavec)$, where 
	 $A(\betavec)$ is the matrix defined in \eqref{def_A} for $\beta$. Using the bounds of Lemma \ref{matrix_non-bi}(b), 
	 for any $\betavec \in \Reals^n$ that $\|\betavec - \betavec^{(0)}\|_\infty \leq c$, we have 
	\[
	 	\|J^{-1}(\betavec)\|_2 \leq \|J^{-1}(\betavec)\|_1 = \|J^{-1}(\betavec)\|_\infty= O\Bigl(\dfrac{1}{\varLambda(\betavec)\varDelta}\Bigr),
	\]
	where $\varLambda(\betavec) = \lambda(\betavec)(1- \lambda(\betavec))$ and 
	$\lambda(\betavec)$ is defined according \eqref{def:lambda}.  Note  that if 
	$\|\rvec(\betavec)\|_\infty \ll \lambda(1-\lambda) \varDelta$, then we get that 
	$\lambda(\betavec)= \Theta(\lambda)$ and $1- \lambda(\betavec) = \Theta(1-\lambda)$.
Applying Lemma \ref{l:Kowa} with $\delta=C/\(\lambda(1-\lambda)\Delta\)$ where $C>0$ is sufficiently large, we complete the proof.
\end{proof}
	
\begin{theorem}\label{thm:prescribed}
	Suppose  a graph $S$  and a degree sequence $\tvec$ satisfy the assumptions of Theorem \ref{thm:enum}.
	Let $H^{+}$ and $H^{-}$ be disjoint subgraphs of $S$ such that $\|\hvec\|_2 \ll (\varLambda \varDelta)^{1/2}$,
	where $\hvec$ is the degree sequence of $H^{+}\cup H^{-}$. Then, for any $\eps>0$,
	\begin{align*}
		\Pr( H^{+} \subseteq S_{\tvec} \text{ and }  H^{-} \not\subseteq S_{\tvec}) = 
		\left(1 + O\left(n^{-1/2+\eps} + \dfrac{\|\hvec\|_2^2}{ \varLambda \varDelta} \right)\right) 
		\prod_{jk \in H^+} \lambda_{jk}   \prod_{jk \in H^-} (1-\lambda_{jk}).
	\end{align*}
\end{theorem}

\begin{Remark}
The estimate in Theorem \ref{thm:prescribed}  for the case $S=K_n$ was proved by  Isaev and McKay in \cite[Theorem 5.2]{IM} under the additional  constraint that the maximum degree of  $H^+\cup H^-$ is  $O(n^{1/6})$.  
  When $S = K-n$ and $\tvec$ is near-regular, a more precise  formula  for  $\Pr( H^{+} \subseteq S_{\tvec} \text{ and }  H^{-} \not\subseteq S_{\tvec}) $ can be derived from \cite{McKay2011}[Theorem 1], provided $H^+\cup H^-$ has at most $n^{1+\eps}$ edges and maximum degree at most $n^{1/2+\eps}$.  
\end{Remark}

\begin{proof}[Proof of Theorem \ref{thm:prescribed}.]
	Let $S' = S-(H^{+} \cup H^{-})$ and $\tvec' \in \Naturals^n$ be such that $\tvec - \tvec'$ is the degree sequence of $H^+$.
	Then, by definition, 
\[	
	\Pr(H^{+} \subseteq S_{\tvec} \text{ and }  H^{-} \not\subseteq S_{\tvec})  =\frac{N(S', \tvec')}{N(S, \tvec)}.
\]
Since $\hvec$ is an integer vector, we have that
\begin{equation}\label{h_bounds}
	\|\hvec\|_1 \leq \|\hvec\|_2^2 \ll \varLambda\varDelta, \qquad
	\|\hvec\|_\infty  \leq \|\hvec\|_2 \ll (\varLambda\varDelta)^{1/2}. 
\end{equation}
Using $\betavec$ as  $\betavec^{(0)}$ in Corollary \ref{cor:beta}, we  find a 
solution $\betavec'$ of system   \eqref{system} for the graph  $S'$ and the vector $\tvec'$ such that 
\begin{equation}\label{diff_beta}
	\begin{aligned}
	 \|\betavec' - \betavec\|_\infty &\leq \|\betavec' - \betavec\|_2   = O\left( \frac{\|\hvec\|_2}{ \varLambda \varDelta}\right) = 
	 o\((\varLambda \varDelta)^{-1/2}\),  \\
	\\
		 \|\betavec' - \betavec\|_1 & = O\left( \frac{\|\hvec\|_1}{ \varLambda \varDelta}\right)    =
		  O\left( \frac{\|\hvec\|_2^2}{ \varLambda \varDelta}\right) = o(1).
		  \end{aligned}
\end{equation}
Observe  that $\range{\betavec'} = \range{\betavec} + o(1)$ and  $\|\hvec\|_\infty  \ll (\varLambda\varDelta)^{1/2} \ll  \dfrac{\varDelta^2}{n}$. Therefore, $S$ and $\tvec'$ 
also satisfy the assumptions of Theorem \ref{thm:enum}. 
Then we obtain
\[
	\Pr(jk \in S_{\tvec})= \bigl (1+ O(n^{-1/2+\eps})\bigr)\, 
	 \frac{|Q|^{1/2} \exp\left( \E u'(\X') - \tfrac12 \E v'^2(\X') \right)}
	 {|Q'|^{1/2}  \exp\left( \E u(\X') - \tfrac12 \E v^2(\X') \right)} \,R,
\]
where $Q$, $u$, $v$ and $Q'$, $u'$, $v'$  are matrices of \eqref{def_A} 
and polynomials of \eqref{def_polynomials} for $S$, $\tvec$ and $S'$, $\tvec'$, respectively, 
$\X$ and $\X'$ are the corresponding  normally distributed vectors and
\[
	R = \frac{\prod_{jk \in S'} (1+ e^{\beta_j'+\beta_k'})}{ \prod_{jk \in S} (1+ e^{\beta_j+\beta_k})} \; \prod_{j=1}^n \,e^{t_j \beta_j - t_j' \beta_j'}.
 \]
 Let $(\lambda_{jk})$ and $(\lambda_{jk}')$ be defined as in
   \eqref{def_lambda} for for $S$, $\tvec$ and $S'$, $\tvec'$. 
  From \eqref{diff_beta}, we get
    \begin{equation}\label{lambda_lambda}
         \lambda_{jk}'(1-\lambda_{jk}') =
         \(1+ O\(\beta_j+\beta_k - \beta_j' -\beta_k'\) \)\,\lambda_{jk} (1-\lambda_{jk})
          =O(\varLambda).
     \end{equation}
 Applying Taylor's theorem to $\log(1 + e^{x})$ and symmetry $\lambda_{jk}'=\lambda_{kj}'$, we obtain that
 \begin{align*}	 
 	\sum_{jk \in S'} \log \left( \frac{1+ e^{\beta_j+\beta_k}}{ 1+ e^{\beta_j'+\beta_k'}}\right) &=
 	\sum_{jk \in S'}
   \Bigl( \lambda_{jk}' (\beta_j+ \beta_k- \beta_j' - \beta_k')  \\
      &\qquad  \qquad  \qquad + O\left( \varLambda  ( |\beta_j- \beta_j'|^2 + |\beta_k- \beta_k'|^2) \right) \Bigr)\\
   &= \sum_{j=1}^n \sum_{k \st jk \in S'}   \Bigl( 
   \lambda_{jk}' (\beta_j - \beta_j') +       O\left( \varLambda  ( |\beta_j- \beta_j'|^2) \right) \Bigr)\\
          &= O\left(\varLambda \varDelta\|\betavec' - \betavec\|_2^2 \right)  + \sum_{j =1}^n 
          t_j' (\beta_j - \beta_j') 
 \end{align*}
 Then, using \eqref{diff_beta} again, we get that
 \begin{align*}
 	R &=  \(1 + O(\varLambda \varDelta\|\betavec' - \betavec\|_2^2 )\)
 	\frac{\prod_{j =1}^n e^{(t_j- t_j') \beta_j}}{\prod_{jk \in S-S'} (1+ e^{\beta_j+\beta_k})}
 	\\
 	&= 	 \left(1 + O\left(\dfrac{\|\hvec\|_2^2}{ \varLambda \varDelta}\right)\right)
 	\prod_{jk \in H^+} \lambda_{jk}   \prod_{jk \in H^-} (1-\lambda_{jk}). 
 \end{align*}
 Next, we will prove that
    \[
       \log \left( \dfrac{ |A|}{|A'|}\right) +  |\E u(\X)- \E u'(\X')| +  |\E v^2(\X) - \E v'^2(\X')| = O\left( n^{-1/2+\eps} + \dfrac{\|\hvec\|_2^2}{ \varLambda \varDelta}\right).
     \]
     
  Let $Q = (q_{jk})$ and $Q' = (q_{jk}')$.  Using  \eqref{lambda_lambda}, we find that
    \begin{equation}\label{A_diff}
          q_{jk}' - q_{jk}  = 
          \begin{cases} 	
          		O(\varLambda) \Bigl(  h_j  + \sum_{k: jk \in S'} |\beta_j +\beta_k - \beta_j'-\beta_k'| \Bigr)   , & \text{if } j=k;\\
          		O(\varLambda)|\beta_j +\beta_k - \beta_j'-\beta_k'| ,& \text{if } jk \in S'; \\ 
          		O(\varLambda), &\text{if } jk \in S - S';\\
          		0, &\text{otherwise.}
          \end{cases}
    \end{equation}
    If matrices $U,V$ are symmetric and positive definite, then $UV$ is similar
    to $U^{-1/2}UVU^{1/2}=(V^{1/2}U^{1/2})\trans (V^{1/2}U^{1/2})$, which is symmetric and positive
    definite. That is, the product of two symmetric positive definite matrices has
    positive real eigenvalues.
    In particular, $Q^{-1}Q'$ and $(Q')^{-1}Q$ have positive real eigenvalues, since
    $Q$ and $Q'$ are symmetric and positive definite (see Lemma \ref{matrix_non-bi}).
     Therefore, using $\log x\le x-1$, we can bound 
    \[	
    			\frac{|Q'|}{|Q|}  \leq e^{\tr(Q^{-1}Q') - n} = e^{\tr(Q^{-1}(Q' - Q))}, \qquad  
    			\frac{|Q|}{|Q'|}  \leq e^{\tr((Q')^{-1}Q) - n} = e^{\tr((Q')^{-1}(Q - Q'))}.
    \]
Using \eqref{diff_beta}, \eqref{A_diff} and the bounds of Lemma \ref{matrix_non-bi}(b), we get that 
\begin{align*}
	\tr(Q^{-1}(Q' - Q)) &= 
	O(\varDelta^{-1}) \sum_{j=1}^n\, \biggl(h_j + \sum_{k: jk \in S'} |\beta_j +\beta_k - \beta_j'-\beta_k'| \biggr) \\
	  &= O(\varDelta^{-1}) \left(\|\hvec\|_1 + \varDelta \|\betavec'-\betavec\|_1\right) = O\left(\dfrac{\|\hvec\|_2^2}{ \varLambda \varDelta}\right).
\end{align*}
    The same argument carries over for  $\tr((Q')^{-1}(Q - Q'))$ and thus $\log \dfrac{|Q|}{|Q'|} =
     O\left(\dfrac{\|\hvec\|_2^2}{ \varLambda \varDelta}\right)$.
 
 Next, repeating the arguments of  Lemma \ref{L:inside}  (see \eqref{E_u} and \eqref{Var_v})
and using \eqref{h_bounds}, \eqref{diff_beta}, \eqref{lambda_lambda}, we derive that  
    \begin{align*}
    	\E (u(\X') - u'(\X')) &= O\biggl(
    	    \varLambda \|\betavec'-\betavec\|_\infty \sum_{jk \in S} \dfrac{1}{ (\varLambda \varDelta)^2}+ 
    	 \varLambda \sum_{jk \in S-S'} \dfrac{1}{(\varLambda \varDelta)^2}\biggr)  
    	 \\
    	  &= O\left( n \dfrac{(\varLambda\varDelta)^{-1/2}}{\varLambda \varDelta}+ 
    	  \dfrac{\|\hvec\|_1}{\varLambda  \varDelta^2} \right)=     	   O(n^{-1/2+\eps})  
    \end{align*}
     and
     \[
     	  \E (v(\X') - v'(\X'))^2) = O\left( \dfrac{n}{\varLambda \varDelta} \|\betavec'-\betavec\|_\infty^2 + 
     	   \varLambda^2 \dfrac{\|\hvec\|_1^2}{ (\varLambda \varDelta)^3}\right) = O\left(\dfrac{n}{(\varLambda\varDelta)^2}\right).
     \]
  Observe also that (using arguments of  \eqref{Var_v})
  \[
      \E (v(\X') + v'(\X'))^2) \leq  2  \E  v^2(\X') + 2 \E  v'^2(\X') = O\left(\dfrac{n}{\varLambda \varDelta}\right).
  \]
   Applying the Cauchy-Schwartz inequality, we find that
   \begin{align*}
   	\E v^2(\X') - \E v'^2(\X') &= \E (v(\X') -v'(\X')) (v(\X')+ v'(\X')) \\ &=
   	 O\left(\sqrt{\dfrac{n}{(\varLambda\varDelta)^2 } \cdot \dfrac{n}{\varLambda\varDelta  }}\right)
   	 = O(n^{-1/2+\eps}).
   \end{align*}

 It remains for us to bound  $\E u(\X)- \E u(\X')$ and $\E v^2(\X) - \E v^2(\X')$.
    To do this we need to establish a few more bounds on the difference of the covariance matrices of $\X$ and~$\X'$.
    From  \eqref{h_bounds}, \eqref{diff_beta} and \eqref{A_diff}, 
    we get that 
    \[
     q_{jj}' - q_{jj} = O(\varLambda ) \left(\|\hvec\|_\infty + \varDelta \|\betavec-\betavec'\|_\infty \right) 
     =  O(\|\hvec\|_2) = o((\varLambda\varDelta)^{1/2}) 
     \]
     and, for $jk \in S'$,
     \[
     	q_{jk}' - q_{jk}  = O\left(\varLambda   \|\betavec-\betavec'\|_\infty \right) = O\left(\dfrac{\|\hvec\|_2}{\varDelta}\right) = 
     	o\bigl((\varLambda/\varDelta)^{1/2}\bigr).
     \]
Let  $Q^{-1}= (\sigma_{jk})$ and $(Q')^{-1}= (\sigma_{jk}')$. Observe that 
     \[
    	 Q^{-1}  - (Q')^{-1}=  Q^{-1} (Q'- Q)(Q')^{-1}.
    \]
  Then, using  \eqref{h_bounds} and  bounds of  Lemma \ref{matrix_non-bi}(b) for $Q$ and $Q'$, we obtain that
    \begin{align*}
  	\sigma_{jj}' - \sigma_{jj} &=  
  	O\biggl(\dfrac{|q_{jj}'-{q_{jj}|}}{ \varLambda^2\varDelta^2} +
  	\sum_{k=1}^n \dfrac{|q_{kk}'-{q_{kk}|}}{ \varLambda^2\varDelta^4} 
  	  	+ \sum_{k\ell \in S } \dfrac{|q_{k\ell}'-{q_{k\ell}|}}{\varLambda^2 \varDelta^3} \biggr)
  	  	\\ &=
       o\left( \dfrac{(\varLambda\varDelta)^{1/2}}{\varLambda^2\varDelta^{2}} + 
   \dfrac{n(\varLambda\varDelta)^{1/2}}{\varLambda^2\varDelta^4} 
   +  \dfrac{ n \varDelta (\varLambda/\varDelta)^{1/2}}{\varLambda^2\varDelta^3} 
   \right) + O\left(\dfrac{\|\hvec\|_1 \varLambda}{\varLambda^2 \varDelta^3} \right) \\ 
   &= O(n^{-3/2 + \eps/2}).
  \end{align*}
  Similarly, for $jk\in S'$ or $jk \notin S$, we have
  \begin{align*}
   \sigma_{jk}' - \sigma_{jk}&=
   O\biggl(\dfrac{|q_{jk}'-{q_{jk}|}}{ \varLambda^2\varDelta^2} +
   	\sum_{\ell=1}^n \,
   	\biggl( \dfrac{|q_{j\ell}'-q_{j\ell}|+|q_{k\ell}'-q_{k\ell}|}{\varLambda^2 \varDelta^3}
   	+
  	\dfrac{|q_{\ell\ell}'-q_{\ell\ell}|}{ \varLambda^2\varDelta^4}\biggr) 
  	+ \sum_{\ell m  \in S } \dfrac{|q_{\ell m}'-{q_{\ell m}|}}{\varLambda^2 \varDelta^4} \biggr)
  	\\ &= 
  	      o\left(\dfrac{(\varLambda/\varDelta)^{1/2}}{ \varLambda^2 \varDelta^2} 
  	       + \dfrac{  \varDelta (\varLambda/\varDelta)^{1/2}}{\varLambda^2\varDelta^3}  +
  	       \dfrac{ n (\varLambda\varDelta)^{1/2}}{\varLambda^2\varDelta^4}   \right)
  	       + O\left(\dfrac{\|\hvec\|_\infty \varLambda}{\varLambda^2 \varDelta^3}
  	       + \dfrac{\|\hvec\|_1 \varLambda}{\varLambda^2 \varDelta^4} \right)\\
  	&=O(n^{-5/2 + \eps}).
  \end{align*}
    For random vectors $\X$ and $\X'$, define 
    $(\sigma_{jk,\ell m})$ and $(\sigma_{jk,\ell m}')$ as in \eqref{sigma_def}.  From the above
    and    Lemma \ref{matrix_non-bi}(b), we obtain that
    \[
    	\sigma_{jk,\ell m} -  \sigma_{jk,\ell m}'  = 
    	 \begin{cases}
    	  O(n^{-3/2 + \eps}),& \text{if } \{j,k\} \cap \{\ell, m\} \neq \emptyset;\\
    	  O(n^{-5/2 + \eps}), & \text{if } \{j,k\} \cap \{\ell, m\} =  \{j\ell, jm, k\ell, km \} \cap (S - S')  = \emptyset.
    	  \end{cases} 
    \]
Now, using arguments of \eqref{E_u} and \eqref{Var_v}), we get that 
   \[
   	\E u(\X)- \E u(\X') =  O\biggl(\varLambda \sum_{jk\in S} 
   	|\sigma_{jk, jk}' - \sigma_{jk,jk}| \cdot|\sigma_{jk,\ell m}' + \sigma_{jk,\ell m}| \biggr) 
   	= O(n^{-1/2 + \eps}).
   \]	 
   Note that, if real $x,y,z,x',y',z'$ admit bounds $|x|,|x'| \leq a$, $|y|,|y'| \leq b$ and $|z|,|z'| \leq c$ for some positive 
   $a,b,c$, then
   \[
   	|xyz - x'y'z'|  \leq \left(\dfrac{|x-x'|}{a} + \dfrac{|y-y'|}{b} + \dfrac{|z-z'|}{c}\right) abc. 
   \]
 Thus, using \eqref{h_bounds} and \eqref{sigma_bounds} for  $(\sigma_{jk,\ell m})$ and $(\sigma_{jk,\ell m}')$, we find that 
   \begin{align*}
   	\E v^2 (\X)- \E v^2(\X') &=  O\biggl( \left(  \dfrac{n^{-3/2 + \eps}}{( \varLambda \varDelta)^{-1}} +  
   	   \dfrac{n^{-5/2 + \eps}}{(n \varLambda \varDelta)^{-1}} \right)
   	 \dfrac{n}{\varLambda \varDelta} + \varLambda^2 
   	\sum_{jk \in S}\, \sum_{\ell m  \in U_{jk}}  \dfrac{1}{\varLambda^3 \varDelta^4}\biggr)\\
   	&= O\left(n^{-1/2 + \eps} + \dfrac{ n\varDelta^2 \|\hvec\|_\infty}{\varLambda \varDelta^4} \right)
   	=O(n^{-1/2 + \eps}).
   \end{align*}
   where $U_{jk} = \left\{\ell m \in S \st
    \{j,k\} \cap \{\ell, m\} = \emptyset \text{ and  } \{j\ell, jm, k\ell, km\} \cap (S-S') \neq \emptyset\right\}$.  
    This completes the proof.
  \end{proof}

   Finally, we  are able to prove the result that was used in the coupling procedure. 
   Moreover, the assumption of Lemma~\ref{matrix_non-bi} is verified by~\eqref{bound_beta} below.
  \begin{proof}[Proof of  Lemma \ref{l:dense}]
		Let $\betavec^{(0)} = (\beta^{(0)}, \ldots, \beta^{(0)})$, where $\beta^{(0)}$ is defined by
		\[
			\frac{e^{2\beta^{(0)}}}{1+ e^{2\beta^{(0)}}} = \lambda =
			 \frac{t_1 + \cdots + t_n}{s_1 + \cdots + s_n} .
			\]
Applying Corollary \ref{cor:beta} and observing that, by the assumptions,
\[
	 \|\rvec(\betavec^{(0)})\|_\infty  = \|\tvec - \lambda \svec \|_\infty  \ll  \lambda(1-\lambda) \varDelta,
\]
we find a solution $\betavec$ of system \eqref{system} such that
\begin{equation}\label{bound_beta}
	\|\betavec-\betavec^{(0)}\|_\infty = O\left( \dfrac{ \|\tvec - \lambda \svec \|_\infty}{ \lambda(1-\lambda) \varDelta}\right).
\end{equation}
Applying Theorem \eqref{thm:prescribed} with $H^+ =\{jk\}$ and $H^{-} = \emptyset$, we find that
\begin{align*}
	\Pr(jk \in S_{\tvec}) = \bigl(1+ O(n^{-1/2 + \eps})\bigr) \lambda_{jk} 
\end{align*}
Using \eqref{bound_beta} and Taylor's theorem, we get
\[
	\lambda_{jk} = \lambda + O(\varLambda \|\betavec-\betavec^{(0)}\|_\infty )
	=   \left(1+ O\left( \dfrac{ \|\tvec - \lambda \svec \|_\infty}{ \lambda \varDelta}\right) \right) \lambda.
\] 
Combining the two bounds above, we complete the proof.
  \end{proof} 
%%%%%%%%%%%%%%%%%%%%%%%%%%%%%%%%%%%%%%%%%%%%%%%%%%

\nicebreak
\section{Switchings}\label{s:switchings}

In this section we prove  Lemma \ref{l:co-sparse}. For an edge  $jk \in S$ 
consider
the partition of the set of $\tvec$\nobreakdash-factors of $S$ into two disjoint sets $\mathcal{S}(\tvec,jk)$ and  $\mathcal{S}(\tvec, \overline{jk})$, where   elements of  $\mathcal{S}(\tvec,jk)$  contain $jk$ while 
 elements of $\mathcal{S}(\tvec,\overline{jk})$  do not. Since $S_{\tvec}$ is a uniform random $\tvec$-factor of $S$, we have
 \begin{equation}\label{switch1}
 	\Pr (jk \in S_{\tvec}) =  \frac{|\mathcal{S}(\tvec,jk)|}{ |\mathcal{S}(\tvec,jk)| + |\mathcal{S}(\tvec,\overline{jk})|}.
 \end{equation}
Thus, it is sufficient to estimate the ratio $|\mathcal{S}(\tvec,jk)| / |\mathcal{S}(\tvec,\overline{jk})|$.  We do it  using  the switching method which we briefly describe below.  

Given a $\tvec$-factor $T \in \mathcal{S}(\tvec,jk)$, consider the  set $\mathcal{F}(T) \subseteq \mathcal{S}(\tvec,\overline{jk})$ of  $\tvec$-factors of $S$ that can be obtained from $T$ by a certain switching operation. Similarly, for   $T' \in \mathcal{S}(\tvec,\overline{jk})$ we consider the  set $\mathcal{B}(T') \subseteq \mathcal{S}(\tvec,jk)$ of  $\tvec$-factors of $S$ that can be obtained from $T'$  by inverting this switching operation.  The main idea of  the switching method is  to define the switching operation in such a way that all sets $\mathcal{F}(T)$  
are of approximately the same size and also all  sets  $\mathcal{B}(T')$ are of approximately the same size.
Then, using the double counting argument, we can estimate 
\begin{equation}\label{switch2}
	 \frac{\min_{T' \in \mathcal{S}(\tvec,\overline{jk}) } |\mathcal{B}(T')|}{ 
	 \max_{T \in \mathcal{S}(\tvec,jk) } |\mathcal{F}(T)|}
	  \leq 
	  \frac{|\mathcal{S}(\tvec,jk)| }{  |\mathcal{S}(\tvec,\overline{jk})|} 
	  \leq \frac{\max_{T' \in \mathcal{S}(\tvec,\overline{jk}) } |\mathcal{B}(T')|}{ 
	 \min_{T \in \mathcal{S}(\tvec,jk) } |\mathcal{F}(T)|}.
\end{equation}

Next, we define our switching operation which is called \textit{$\ell$-switching}, where $\ell \geq 3$ is an integer.  To perform an $\ell$-switching on a graph $T\in \mathcal{S}(\tvec,jk)$, choose a sequence of vertices $u_1,v_1,u_2,v_2,\ldots,u_{\ell},v_{\ell}$ such that
\begin{itemize}
\item $u_1 = j$, $v_\ell = k$ and $u_i v_i$, $v_i u_{i+1}$, for $i=1,\ldots,\ell$  are $2\ell$  distinct edges in $S$   (for $i = \ell$, we put $u_{\ell+1} = u_{1}$ and repetitions of vertices are  allowed);
\item $u_iv_i$ are edges in $S-T$ for all  $i =1, \ldots, {\ell}$;
\item $v_iu_{i+1}$ are edges in $T$ for all $i =1, \ldots, {\ell-1}$.
\end{itemize} 
Then the $\ell$-switching replaces the edges  $\{v_iu_{i+1}\}_{i =1, \ldots, {\ell}}$  in $T$
by $\{u_iv_i\}_{i =1, \ldots, {\ell}}$. Observe that the resulting graph $T'$ has the same degree sequence and  $T' \in \mathcal{S}(\tvec,\overline{jk})$.
  The operation converting $T'$ to $T$ is called an \textit{inverse $\ell$-switching.}
See Figure~\ref{f:long} for an illustration.
\begin{figure}[h!]
\centering
\begin{tikzpicture}[scale=0.6]
\node (u1) at (0,4) {};
\node [above left=0 and -0.5 of u1] {$u_1=j$};
\node (vl)  at (2,4) {};
\node[above right =0 and -0.5 of vl]{$k=v_{\ell}$};
\node (v1) [label=left:{$v_1$}]   at (-1.5,2.5) {};
\node (ul)  [label=right:{$u_{\ell}$}]  at (3.5,2.5) {};
\node (u2) [label=left:{$u_2$}]  at (-1.5,0.5) {};
\node (v2) [label=right:{$v_{\ell-1}$}]  at (3.5,0.5) {};
\node at (1,-1)  {\Huge $\boldsymbol{\ldots}$};

\node at (1,1.6)  {\LARGE $T$};

\draw [-,thick] (u1) -- (vl);
\draw [-,dashed](u1) -- (v1);
\draw[-,thick] (v1) -- (u2);
\draw[-, thick] (ul)--(v2);
\draw[-,dashed](ul)--(vl);

\draw [fill] (u1) circle (0.2); \draw [fill] (v1) circle (0.2);  \draw [fill] (u2) circle (0.2);   \draw [fill] (v2) circle (0.2);  \draw [fill] (ul) circle (0.2); 
\draw [fill] (vl) circle (0.2);

 \node   at (5.8,1.5)  {$\Longleftrightarrow$};

\begin{scope}[shift={(9.5,0)}]

\node (u1) at (0,4) {};
\node [above left=0 and -0.5 of u1] {$u_1=j$};
\node (vl)  at (2,4) {};
\node[above right =0 and -0.5 of vl]{$k=v_{\ell}$};
\node (v1) [label=left:{$v_1$}]   at (-1.5,2.5) {};
\node (ul)  [label=right:{$u_{\ell}$}]  at (3.5,2.5) {};
\node (u2) [label=left:{$u_2$}]  at (-1.5,0.5) {};
\node (v2) [label=right:{$v_{\ell-1}$}]  at (3.5,0.5) {};
\node at (1,-1)  {\Huge $\boldsymbol{\ldots}$};

\node at (1,1.6)  {\LARGE $T'$};

\draw[-,dashed] (u1) -- (vl);
\draw [-,thick] (0,4) -- (-1.5,2.5);
\draw[-,dashed](v1) -- (u2);
\draw [-,dashed] (ul)--(v2);
\draw [-,thick](3.5,2.5)--(2,4);

\draw [fill] (u1) circle (0.2); \draw [fill] (v1) circle (0.2);  \draw [fill] (u2) circle (0.2);   \draw [fill] (v2) circle (0.2);  \draw [fill] (ul) circle (0.2); 
\draw [fill] (vl) circle (0.2);

\end{scope}
\end{tikzpicture}
\caption{$\ell$-switching.}
\label{f:long}
\end{figure}

In the following we let $|\mathcal{F}_\ell (T)|$ be the  number of $\ell$-switchings applicable to 
a  graph $T\in  \mathcal{S}(\tvec,jk)$. Similarly,  let $| \mathcal{B}_\ell(T')|$  be the number of inverse $\ell$-switchings applicable to a graph $T'\in \mathcal{S}(\tvec,\overline{jk})$.
Recall from~\eqref{Z_def} that 
    \[
  	 \langle\xvec,\yvec\rangle_S = \sum_{(jk)\st jk \in S} x_j y_k,
  \] 
  and that $A(G)$ is the adjacency matrix of a graph $G$.
  Let $\evec_i$ denote the standard unitary column vector with $1$ in the $i$-th component. For  nonnegative integers $a,b$ define 
  \begin{equation}\label{def_W}
  	w_{a,b}(S,T) = \max_{i,P} \|P \evec_i\|_\infty,
  \end{equation}
  where the maximum is taken over all $i\in [n]$ and matrices $P$ which are product of
  $a$ factors $A(S)$ and $b$ factors $A(T)$ (e.g. for $a=1$, $b=2$, the matrix $P$ can be one of 
  $A(S)A(T) A(T)$, $A(T)A(S) A(T)$, $A(T) A(T) A(S)$).   Note that the 
   components of $P \evec_i$ 
   correspond to the number of walks that start at $i$ and finish at a given vertex
       which use  $a$ edges from $S$ and $b$ edges from $T$  in a predetermined order (corresponding to $P$). Thus, $w_{a,b}(S,T)$
       is an upper bound  on the number of such walks.
  
  \begin{lemma}\label{l:FB}
   Assume $\ell = 2 h+1$ for  some positve integer $h$. 
   Let $A = A(S)$.
    \begin{itemize}
    \item[(a)]     If $T \in \mathcal{S}(\tvec,jk)$,  $B=A(T)$ and    $w_{a,b} = w_{a,b}(S, T)$, then 
    \begin{align*}
    	\langle(BA)^{h} \evec_j , 
	              (BA)^{h} \evec_k\rangle_S
	               &\geq 
	             |\mathcal{F}_\ell (T)|   
	             \\
	             &\geq 
	             	\langle(BA)^{h} \evec_j , 
	              (BA)^{h} \evec_k\rangle_S
  -	               \ell\, w_{\ell-1,\ell}\\ & {~~}-
  \ell^2      (\varDelta(\tvec))^{\ell-1}(\varDelta(S))^{\ell-2}
	                               \max_{\lfloor \ell /3\rfloor  \leq a \leq \ell-2} 
	                               \dfrac{w_{a,a}+w_{a,a-1}\varDelta(S)}{(\varDelta(\tvec)\varDelta(S))^a}.
	                                                                  \end{align*}
	                                                                  
	   \item[(b)]   If $T' \in \mathcal{S}(\tvec,\overline{jk})$,  $B'=A(T')$ and    $w_{a,b}' = w_{a,b}(S, T')$, then                   
	    \begin{align*}
    	\langle(B'A)^{h-1}B' \evec_j , 
	              (B'A)^h B'\evec_k\rangle_S
	               &\geq 
	             |\mathcal{B}_\ell (T')|   
	             \\
	             &\geq 
	              \langle(B'A)^{h-1}B' \evec_j , 
	              (B'A)^h B'\evec_k\rangle_S
  -	               \ell\, w_{\ell-2,\ell+1}'\\ &{~~}-
  \ell^2      (\varDelta(\tvec))^{\ell-1}(\varDelta(S))^{\ell-2}
	                                  \max_{\lfloor \ell /3\rfloor  \leq a \leq \ell-2} 
	                               \dfrac{w_{a,a}'+w_{a-1,a}'\varDelta(S)}{(\varDelta(\tvec)\varDelta(S))^a}.
	                                                                  \end{align*}
    \end{itemize}
\end{lemma}

\begin{proof}
Observe that the components of  	$(BA)^h \evec_j$ 
correspond to counts of walks of length $2h$ which alternate between edges of $S$ and $T$ starting from vertex $j$ and an edge from $S$.   We call such walks \textit{$ST$-alternating walks.} Clearly, this gives an upper bound for the number of walks that alternate between $S-T$ and $T$. Any $\ell$-switching is  determined by the sequence 
	$j = u_1, v_1, \ldots, u_{\ell}, v_{\ell} =k$ which consists of 
	edge $v_{h+1} u_{h+1}$ and 	two walks of length $2h$ which alternate between $T$ and $S-T$
		starting  from  vertices $j,k$.  Summing over all choices of $v_{h+1} u_{h+1} \in S$
	and estimating the choice for walks by  corresponding components of 	$(BA)^h \evec_j$
	and $(BA)^h \evec_k$,  we prove the upper bound for $ |\mathcal{F}_\ell (T)|$.

		The argument above counts $ST$-alternating walks 
			$W= u_1, v_1, \ldots, u_{\ell}, v_{\ell}$ 
			such that $u_1 = j$ and $v_{\ell}=k$ 			but some of  them 
	may be not valid $\ell$-switchings.  This could happen in the following cases:  
		\begin{itemize}\itemsep=0pt
			\item[(1)] one of the edges $u_i v_i$ which we choose from $S$ belongs also to  $T$;
			\item [(2)] collision of an edge from $S$; i.e.,  $\{u_i, v_i\} = \{u_{i'}, v_{i'}\}$ for some $i\neq i'$;
			  \item [(3)] collision of an edge from  $T$; i.e., 
			   $\{v_i, u_{i+1}\} = \{v_{i'}, u_{i'+1}\}$ for some $i\neq i'$.
		\end{itemize}
		Note that we do not need to consider collisions of the form $\{u_i,v_i\}=\{v_{i'},u_{i'+1}\}$ separately since it is already covered by  case (1).   
		Then, 
		\begin{equation}\label{e:pre-b}
			|\mathcal{F}_\ell (T)|  \geq  \langle(BA)^h \evec_j , 
	              (BA)^h \evec_k\rangle_S - N_1 - N_2 -N_3,
		\end{equation}
		where  $N_1, N_2, N_3$ denote the number of invalid choices for $W$
		corresponding to 
		cases (1), (2), (3), respectively.
		
							Recalling definition \eqref{def_W},  we get that, for any fixed $i$, the number of choices for  $W$ such that $u_i v_i \in T$ is at most $ w_{\ell-1,\ell}$, since $W$ consists of 	$\ell$ edges from $T$ and $\ell-1$ edges of $S$. 
						Letting $\ell$ be the number of choices for $i$, we get that 
							\[ N_1 \leq \ell\,w_{\ell-1,\ell}.
							\]
							
	Next, consider the collision of edges from $S$. For any fixed  $i< i'$, 
	we count the number of ways to choose three  
	$ST$-alternating  walks 		$W_1 = u_1, v_1, \ldots, v_{i-1}, u_{i}$,
	$W_2 = u_{i+1}, u_{i+1}, \ldots, v_{i'-1} u_{i'}$,  $W_3 = v_{i'}, u_{i'+1}, \ldots, u_{\ell}, v_{\ell}$
	and two edges $u_i v_i  \in S$,  $v_i u_{i+1} \in T$ 
	such that $u_1= j$, $v_{\ell}=k$ and $\{u_i, v_{i}\} = \{u_{i'},v_{i'}\}$. 
	 Note that  $W_1$, $W_2$, $W_3$  have even lengths and let $W^*$  be the longest (or one of the longest) among them.
	If $W^*$ consists of $2a$ edges then, clearly,  $6a \geq  2 \ell -4$ (the length of $W$ is $2 \ell-1$ but we need to remove $u_iv_i = u_{i'} v_{i'}$ and $v_i u_{i+1}$) and so 
	$a \geq \lfloor \ell /3\rfloor$.	We also have $a < \ell -1$ since  $W^*$ is at most 
	$W$ without two edges. 
	The number of ways to specify identities of all vertices of   $u_1, v_1, \ldots, u_{\ell}, v_{\ell}$ except $u_1=j$, $v_\ell=k$ and $a-1$ internal vertices of $W^*$ is bounded above by $(\varDelta(T))^{\ell-a-1} (\varDelta(S))^{\ell-a-2}$. Indeed, once a vertex is specified and we know that the next edge should be in $S$ (or $T$) then the number of choices for the next vertex is  at most $\varDelta(S)$ (or $\varDelta(T)$). 
	Overall we have $\ell-a-1$ edges of $T$ and $\ell-a-1$ edges of $S$ in $W-W^*$ because of the repeated edge $u_i v_i= u_{i'} v_{i'}$ but 
	one of the edges in $S$ is not needed for specification of vertices 
	(we either have a cycle or $W-W^*$ contains both $j$ and $k$).
	Given its endpoints,
	the  number of ways to choose $W^*$ is bounded above by $w_{a,a}$.
	Allowing $\ell^2$ for
	 the choice of $i, i'$ and for specifying between  $u_i = v_{i'}$ or $u_i = u_{i'}$, we get that 
	 \[ 
	   N_2 \leq \ell^2\max_{\lfloor \ell /3\rfloor \leq a \leq \ell-2} w_{a,a} \cdot  (\varDelta(T))^{\ell-a-1} (\varDelta(S))^{\ell-a-2}. 
	   \]
		
	To bound  $N_3$, we estimate 
	the number of choices for $W$ such that $\{v_i,u_{i+1}\} = \{v_{i'}, u_{i'+1}\}.$ 
	 In this case, 	for any fixed  $i< i'$, we specify $W$ by choosing three
	$ST$-alternating  walks 		$W_1 = u_1, v_1, \ldots, u_{i}, v_{i}$,
	$W_2 = u_{i+1}, v_{i+1}, \ldots, u_{i'} v_{i'}$,  $W_3 = u_{i'+1}, \ldots, v_{\ell}$
	and one edge  $v_i u_{i+1} = v_{i'}u_{i'+1} \in T$. 
	 Note that  $W_1$, $W_2$, $W_3$  have odd lengths with a first and last edge in $S$. Let $W^*$  be the longest (or one of the longest) among them.
	If $W^*$ consists of $2a-1$ edges then, clearly,  $6a-3 \geq  2 \ell -3$ (the length of $W$ is $2 \ell-1$ but we need to remove $v_i u_{i+1} = v_{i'}u_{i'+1}$) and  so $a \geq \lfloor \ell /3\rfloor$.  Also $a< \ell-1$ since each of $W_1$, $W_2$, $W_3$	has at least one edge.	Arguing similarly to the previous  paragraph, we find that 
	\[
		N_3 \leq \ell^2  \max_{\lfloor \ell /3\rfloor \leq a \leq \ell-2}  w_{a,a-1} \cdot 
		   (\varDelta(T) \varDelta(S))^{\ell-a-1}. 
	\] 
	Part (a) now follows from \eqref{e:pre-b}.

	Part (b) is proven in a completely similar way to  part (a). 
	 The number
	             walks  
	             of length $2\ell-1$ 
	               from $j$ to $k$             which alternate between edges of $S$ and $T'$   starting from a edge in $T'$ equals  
	$\langle(B'A)^{h-1}B' \evec_j , 
	              (B'A)^h B'\evec_k\rangle_S$ which is an upper bound for the number of  inverse $\ell$-switchings. For the lower bound, we need again to consider three cases when the constructed walk is not a valid inverse $\ell$-switching: 
	              an edge from $S$ also belongs to $S$;  collision of an edge from $T$;
	              collision of an edge from $T$.  Let $N_1'$,  $N_2'$, $N_3'$ corresponds to the counts for theses cases. Then we have 
	             $
	              	N_1' \leq \ell w_{\ell-2, \ell+1}'
	           $
	               because such walks consist of $\ell-2$ edges from $S$ and $\ell+1$ edges from $T$.
	               For edge collisions,  we consider the same splits into three walks and edges as  in part (a) but with swapped roles of $S$ and $T$. 
	               This leads to the following bounds:
	               \begin{align*}
	               	 N_2' &\leq   \ell^2 \max_{\lfloor \ell /3\rfloor \leq a \leq \ell-2} w_{a,a}'
	               	\cdot  (\varDelta(T))^{\ell-a-1} (\varDelta(S))^{\ell-a-2},\\
	               	N_3' &\leq \ell^2 \max_{\lfloor \ell /3\rfloor \leq a \leq \ell-2} w_{a-1,a}' \cdot             
	               	(\varDelta(T) \varDelta(S))^{\ell-a-1}.
	               \end{align*}       
	              Part (b) follows.
\end{proof}

As a demonstration of the method, we start from the case of dense $S$ and then we proceed to 
Lemma \ref{l:co-sparse} in a sparse setting. 

\nicebreak
\subsection{Dense $S$}
When the degrees of $S$ are linear,  we essentially need only assumption (A3) of  Lemma \ref{l:co-sparse} while assumptions (A1) and (A2) can be significantly simplified, see the lemma below.
\begin{lemma}\label{l:denseS}
    Let $\eps\in (0,1)$ be a constant
	and $S$ be a graph on $n$ vertices 
	such that 
	\[
		\varDelta(S) -\range{S} \geq \eps n,
	\qquad \text{and} \qquad 
		\left|\log \frac{\langle\xvec ,\yvec \rangle_S }{\|\xvec\|_1 \|\yvec\|_1 \varDelta(S)/n}\right| \leq \gamma,
	\]
	 for all $\xvec,\yvec \in [0,1]^n$ with $\|\xvec\|_1, \|\yvec\|_1 \geq \eps^6 n$. 
	Let  $\tvec$ be a degree sequence such that
	there exist a $\tvec$-factor of $S$ and
	\[
		\varDelta(\tvec) = o(n), \qquad  
		\frac{\varDelta(\tvec) - \range{\tvec}}{\varDelta(\tvec)} \geq \eps.
	\]
	Then, for any $jk \in S$, 
	\[
		\Pr(jk \in S_{\tvec}) = \exp\left(O
\left( \gamma+ \dfrac{\range{S}}{\varDelta(S) } + \dfrac{\varDelta(\tvec)}{n} + \dfrac{\range{\tvec}}{\varDelta(\tvec)} 
		\right)\right) \frac{\varDelta(\tvec)}{\varDelta(S)}.
	\]
 \end{lemma}

\begin{proof}      
           We fix $\ell =7$. Using  Lemma \ref{l:FB}, we  will estimate 
          $\mathcal{F}_\ell(T)$ for $T\in \mathcal{S}(\tvec,jk)$ and 
          $\mathcal{B}_\ell(T')$ for $T'\in \mathcal{S}(\tvec,\overline{jk})$.
           Let $A=A(S)$, $B=A(T)$, $B'=A(T')$.     
Let $s_{\min}$ and $t_{\min}$ denote $\varDelta(S)-\range{S}$ and $\varDelta(\tvec)-\range{\tvec}$ respectively.
 Using the assumptions of the lemma and $\|A \evec_j\|_\infty\le 1$, and noting that
 $\|B\|_{\infty}\le \varDelta(\tvec)$ and $\|A\|_{\infty}\le \varDelta(S)$, we find, for $h=1,2,3$,
          \[
          	   \frac{\|(BA)^h \evec_j\|_1}
          	   {\|(BA)^h \evec_j\|_\infty} \geq \frac{(t_{\min} s_{\min})^h}{
          	   \|B\|_\infty^h \|A\|_{\infty}^{h-1} \|A \evec_j\|_\infty} \geq \eps^{2h-1} s_{\min} \geq  \eps^{5}n.
          \]
          Similarly, using $\|A\vvec\|_{\infty} \leq \|\vvec\|_1$, we get
          \[
          	 \frac{\|(B'A)^h B' \evec_j\|_1}
          	   {\|(BA)^h  B'\evec_j\|_\infty}  \geq 
          	   \frac{(t_{\min} s_{\min})^h  t_{\min}}{
          	   \|B'\|_\infty^h \|A\|_{\infty}^{h-1} \|AB' \evec_j\|_\infty}
          	   \geq \eps^{2h-1} \frac{s_{\min} t_{\min}}{\|B' \evec_j\|_1} 
          	   \geq \eps^{2h} s_{\min} \geq \eps^{6} n.
          \]
          Note also that 
          \begin{align*}
          	\|(BA)^h \evec_j\|_1 &=  (\varDelta(\tvec) \varDelta(S))^h  
          	\exp\left( O\left(
          	\dfrac{\range{S}}{s_{\min}} + \dfrac{\range{\tvec}}{t_{\min}}
          	\right)\right),\\
          		\|(B'A)^h B' \evec_j\|_1 &=  (\varDelta(\tvec))^{h+1} (\varDelta(S))^h  
          	\exp\left( O\left(
          	\dfrac{\range{S}}{s_{\min}} + \dfrac{\range{\tvec}}{t_{\min}}
          	\right)\right).
          \end{align*}
  Using similar bounds for $\evec_k$ and the assumption of the lemma on $\langle \xvec, \yvec \rangle_S$, we find that 
      \begin{align*}
      		\langle(BA)^{3} \evec_j , 
	              (BA)^{3} \evec_k\rangle_S &= 
	              \exp\left( O\left( \gamma+
          	\dfrac{\range{S}}{s_{\min}} + \dfrac{\range{\tvec}}{t_{\min}}
          	  \right)\right) \varDelta(\tvec)^{6} \varDelta(S)^{7}/n,
          	 \\
          	 	\langle(B'A)^{2}B' \evec_j , 
	              (B'A)^3 B'\evec_k\rangle_S 
	              &= 
	              \exp\left( O\left( \gamma+
          	\dfrac{\range{S}}{s_{\min}} + \dfrac{\range{\tvec}}{t_{\min}}
          	 \right)\right) \varDelta(\tvec)^{7} \varDelta(S)^{6}/n.
      \end{align*}     
      Next, we need to bound the quantities $w_{a,b}$ and $w_{a,b}'$ that appear in the 
      lower bounds of Lemma~\ref{l:FB}. 
      Consider any product $P$ of $a\geq 1$ factors $A$ and $b$ factors $B$. 
      Representing $P= P_1 A P_2$, we get that 
      \[
      	    \|P_1 A P_2 \evec_i\|_\infty \leq \|P_1\|_\infty \|P_2 \evec_i\|_1 \leq \varDelta(S)^{a-1} \varDelta(\tvec)^{b} \leq    \varDelta(S)^a \varDelta(\tvec)^b/\eps n.
      \]
      Thus, we estimate  $w_{6,7} = O(\varDelta(S)^6 \varDelta(\tvec)^7/n)$, 
      $w_{5,8}' = O(\varDelta(S)^5 \varDelta(\tvec)^8/n) $ and
      \[
      	 \max_{2 \leq a \leq 5} \frac{w_{a,a} + w_{a,a-1} \varDelta(S)}
      	 { (\varDelta(S)\varDelta(\tvec))^a} = O(1/\varDelta(\tvec)), \qquad 
      	  \max_{2 \leq a \leq 5} \frac{w_{a,a}' + w_{a-1,a}' \varDelta(S)}
      	 { (\varDelta(S)\varDelta(\tvec))^a} = O(1/n).
      \] 
      Thus, applying Lemma \ref{l:FB}, we conclude that 
      \begin{align*}
      		|\mathcal{F}_\ell(T)| &= 
      		\exp\left( O\left( \gamma  + 
          	\dfrac{\range{S}}{s_{\min}} + \dfrac{\range{\tvec}}{t_{\min}}
          	 + \dfrac{\varDelta(\tvec)}{n}\right)\right) \varDelta(\tvec)^6 \varDelta(S)^7/n,\\
          	 	|\mathcal{B}_\ell(T')| &=
          	 	\exp\left( O\left(\gamma  + 
          	\dfrac{\range{S}}{s_{\min}} + \dfrac{\range{\tvec}}{t_{\min}}
          	 + \dfrac{\varDelta(\tvec)}{n}\right)\right) \varDelta(\tvec)^7 \varDelta(S)^6/n.
      \end{align*}
      Combining \eqref{switch1} and \eqref{switch2} completes the proof.
 \end{proof}

\nicebreak
\subsection{Preliminaries for sparse $S$}
%%%%%%%%%%%%%%%%%%%%%%%%%%%%%%

For a sparse  $S$,  estimating $|\mathcal{F}_\ell|$ and $|\mathcal{B}_\ell|$ accurately is a non-trivial task and relies heavily on the pseudorandom properties of $S$. Here we prove bounds 
for the quantities $w_{a,b}(S,T)$ which appear in Lemma \ref{l:FB}. 
First we consider the case when both graphs are regular.
Let $J$ denote the $n\times n$ matrix where every entry is 1.

\begin{lemma}\label{lem:walk}
Let $T$ be a $t$-regular graph and  $S$ be a  $s$-regular graph on the same vertex set $[n]$.
Assume that $\left\|A({S}) - \frac{{s}}{n} J\right\|_2 \leq {s} n^{-\alpha}$ for some $\alpha>0$.
Then, for any integers $a\ge 1/\alpha$, $b\geq 0$, we have
\[
w_{a b}({S},{T})\le \frac{2s^at^b}{n}.
\]
\end{lemma}

\begin{proof}
Let $A=A(S)$ and $B = B(T)$.
Consider any matrix $P$ which is a product of $a$ factors $A$ and $b$ factors $B$.
Let  $\tilde{P}$ denote the matrix obtained by replacing all factors $A$ in $P$ by   $A -  \frac{{s}}{n} J$. 
Write 	$\evec_i = \boldsymbol{1}/n + \vvec$, where $\boldsymbol{1}$ is the vector with all components equal $1$.  Note that 
$A \boldsymbol{1} = {s}  \boldsymbol{1}$,  
$B \boldsymbol{1} = {t}  \boldsymbol{1}$
and  $\vvec \perp  \boldsymbol{1}$.  
Since  operators $A$ and $A -  \frac{{s}}{n} J$ act identically on the space orthogonal to $\boldsymbol{1}$,  we find  that
$
	P\vvec = \tilde{P}\vvec.
$
Using  $\norm{\vvec}_2 \leq \norm{\evec_i }_2 =1$ and $\|B\|_2 \leq \|B\|_\infty = t$, we obtain that
\begin{align*}
	\| P \evec_{i}\|_\infty 
		&\leq   \frac{\norm{P\boldsymbol 1}_\infty}{n} + 
	 \| \tilde{P} \vvec\|_\infty \leq \frac{\norm{P\boldsymbol{1}}_\infty}{n} + 
	 \| \tilde{P} \vvec\|_2 \\
	 & \leq \frac{s^a t^b}{n} +  \| (A- \tfrac{{s}}{n}J)\|_2^a \cdot \|B\|_2^b
	 \leq  \frac{s^a t^b}{n} + s^a t^b n^{-a \alpha} \leq \frac{2s^a t^b}{n}. 
\end{align*}
Taking the maximum over all $i$ and $P$ completes the proof.
 \end{proof}

 We will need a bound similar to Lemma \ref{lem:walk} for non-regular $S$ and $T$ as well. For this purpose, we construct regular supergraphs $\tilde{S}\supseteq S$ and $\tilde{T} \supseteq T$ and estimate 
 \begin{equation}\label{eq:wSS}
 	w_{a,b}(S,T) \leq w_{a,b}(\tilde{S},\tilde{T}) .
 \end{equation}
 The next lemma shows that if $G$ is a graph with small $\range{G}$  then there exists a regular supergraph $\tilde{G}\supseteq G$ which is not much bigger than $G$.
\begin{lemma}\label{l:supergraph}
Let $G$ be a graph on $n$ vertices such that $\varDelta(G)+3\range{G} < n/4$.
If $d$ is an even number that
\[ 
 \varDelta(G) + \range{G}  \leq d \leq \varDelta(G)+2\range{G}
 \] 
	then 	there exists 	a $d$-regular supergraph $\tilde{G}$  of $G$. 
\end{lemma}
\begin{proof}
If $\range{G}=0$ there is nothing to prove as we can take $\tilde{G} = G$. Thus, we may assume otherwise. Define sequence $\rvec = (r_1,\ldots,r_n)$ by $r_i=d-d_G(i)$, where $d_G(i)$ is the degree of vertex $i$ in $G$. 
It is sufficient to  find an $\rvec$-factor of $K_n-G$ because  the union of this $\rvec$-factor and $G$ gives our desired $\tilde{G}$. 

By the assumptions, for all vertices $i$, we have
\begin{equation}\label{rran}
	\range{G}\leq d - \varDelta(G)  \leq  r_i \leq d - \varDelta(G) + \range{G} \leq 3 \range{G}. 
\end{equation}

Then, for any $U \subseteq [n]$,   we have
\[	
	\sum_{i \in U} r_i \leq 3 \range{G}\cdot |U|  \leq 
	 |U|\cdot(|U|-1) 
	 + \sum_{i \notin U} \min\{|U|, r_i\}.
\]
To see the above inequality holds, note that if $|U| \geq 3 \range{G}+1$ then the first term of the RHS is at least $3 \range{G}\, |U| $.
If, on the other hand, $|U|\leq 3\range{G}< n/4$ then the second term of the RHS
is at least $\frac{3}{4}n\min\{|U|,\range{G}\}\ge 3 \range{G}\, |U|$.
Also, $\sum_{i} r_i$ 
is even since $d$ is even. By the Erd\H{o}s-Gallai theorem we conclude that 
${\rvec}$ is a graphical degree sequence.

Let $R$ be an $\rvec$-factor of $K_n$ such that $R \cap G$ has the smallest number of edges. We use a switching-type argument to show that that $R \subseteq K_n - G$. By contradiction, assume that there is an edge $u_1 v_1 \in R \cap G$.  
Consider  edges $u_2 v_2 \in R$ 
such that $u_1 u_2 \in K_n -  (R\cap G)$.
  The number of choices for such $u_2 v_2$ is at least
\[	
	\(n-2- (\varDelta(G)-1) - (\varDelta(R)-1)\)  \range{G}
	= (n- \varDelta(G) - \varDelta(R))  \range{G}.
\] 
The first factor in the LHS is corresponds to choices $u_2$ that  $u_1 u_2 \notin R\cup G$ and  the second factor  in the LHS is  a lower bound for the number of ways to choose $v_2$ given $u_2$ (by~\eqref{rran}).  
Note that among all the choices for $u_2v_2$ above, at most $(\varDelta(G) + \varDelta(R)-2)\varDelta(R)$ choices satisfy $v_1 v_2 \in R\cup G$ (estimating the number of ways to choose $v_2$ and then $u_1$).
Also, among all the choices for $u_2v_2$ above, at most $\varDelta(R)$ 
 choices satisfy $v_1 = v_2$.
 By the assumptions and \eqref{rran},
 we find that
 \begin{align*}
 	(n- \varDelta(G) - \varDelta(R))  \range{G} >  (\varDelta(G) + \varDelta(R)-1)\varDelta(R).
 \end{align*}
 Therefore, we can find such $u_2 v_2 \in R$ that  $u_1 u_2 \notin R\cup G $ and 
  $v_1 v_2 \notin R\cup G$ and all vertices $u_1,v_1,u_2,v_2$ are distinct.  
  Then we can replace edges $u_1 v_1$ and $u_2 v_2$ by $u_1 u_2$ and $v_1 v_2$ to get an 
  $\rvec$-factor which has fewer common edges with $G$ than $R$ does. This contradicts our choice of  $R$.  Therefore $R \cap G$ must be empty, which completes the proof.
\end{proof}

%%%%%%%%%%%%%%%%%%%%%%%%%%%%%%%%%
%%%%%%%%%%%%%%%%%%%%%%%%%%%%%%%%%

\nicebreak
\subsection{Proof of Lemma \ref{l:co-sparse}.}

       If $\varDelta(S') \geq n/16$  then the required probability bound follows from  Lemma \ref{l:denseS}. 
        Indeed, take $\eps = (16 \beta)^{-1}$ and observe that
        \begin{align*}
        	\varDelta(S) - \range{S} &\geq  \varDelta(S')/\beta \geq \eps n 
        \\
     	  \frac{\varDelta(\tvec) - \range{\tvec}}{\varDelta(\tvec)} &\geq 1/\beta \geq \eps. 
     \end{align*}
     All the assumptions of  Lemma \ref{l:denseS} are satisfied. 
     In the following, we assume that 
     $\varDelta(S') < n/16$ which implies $\varDelta(S') + 3\range{S'}<n/4$ 
    and  $\varDelta(\tvec) + 3\range{\tvec}<n/4$.

             Take $\ell$ to be the odd number  
             from $\{\lceil 3/\alpha\rceil+3,\lceil 3/\alpha\rceil+4\}$. Using  Lemma \ref{l:FB}, we will estimate 
          $\mathcal{F}_\ell(T)$ for $T\in \mathcal{S}(\tvec,jk)$ and 
          $\mathcal{B}_\ell(T')$ for $T'\in \mathcal{S}(\tvec,\overline{jk})$.
           Let $A=A(S)$, $B=A(T)$, $B'=A(T')$.        
Denote by $s_{\min}=  \varDelta(S) - \range{S}$ the smallest degree of $S$      and by $t_{\min} = \varDelta(\tvec) -\range{\tvec}$   the smallest component of $\tvec$.       
           Let $\tilde{S}$, $\tilde{T}$  be regular supergraphs of $S'$ and $T'$ given by Lemma \ref{l:supergraph}. We have 
           \[
           	      \frac{ \varDelta(\tilde{S})}
           	      {s_{\min}} \leq  
           	      \frac{\varDelta(S') + 2 \range{S'}}{s_{\min}}
           	      \leq 1  + \frac{3 (\varDelta(S') - s_{\min}) }{s_{\min}} \leq \left( \frac{\varDelta(S')}{s_{\min}}\right)^3.
           \]
     Similarly, $\varDelta(\tilde{T})/t_{\min} \leq (\varDelta(\tvec)/t_{\min})^3$.   
     Note also that 
     \begin{align*}
     	\|A(\tilde{S}) - \dfrac{\varDelta(\tilde{S})}{n} J\|_2
     	&\leq \|A(S') - p'J\|_2 +  \|A(S') - A(\tilde{S})\|_2  + |\varDelta(\tilde{S}) - n p'|\\
     		&\leq  \|A(S') - p'J\|_2  + 2 |  \varDelta(\tilde{S}) -   \varDelta(S') + \range{S'}|\\
     		&\leq   \|A(S') - p'J\|_2 + 6  \range{S'} \leq n^{-\alpha}\varDelta(S'),\quad \mbox{by (A2)}.
     \end{align*}
         Combining Lemma \ref{lem:walk} and estimate \eqref{eq:wSS}, we find that, for $a \geq 1/\alpha$, 
\begin{equation}\label{bounds_wab}
\begin{aligned}
   w_{a,b}(S,T)&\leq  w_{a,b}(\tilde{S},\tilde{T}) \leq \frac{2 \varDelta(\tilde{S})^a  \varDelta(\tilde{T})^b}
   {n}
   \\ 
  &\leq \frac{2 s_{\min}^a t_{\min}^b}{n} 
   \left( \frac{\varDelta(S')}{s_{\min}} \cdot \frac{\varDelta(\tvec)}{t_{\min}} \right)^{\max\{a,b\}} 
   \leq \frac{2  s_{\min}^a t_{\min}^b}{n} \beta^{3 \max\{a,b\} \alpha}.
\end{aligned}
\end{equation}
The same bound holds for $w_{a,b}(S,T')$.
 Therefore, for $1/\alpha \leq h \leq (\ell-1)/2$, we have 
 \begin{align*}
 	\frac{\|(BA)^h \evec_j\|_1}
 	{\|(BA)^h \evec_j\|_\infty} &\geq 
 	\frac{(s_{\min} t_{\min})^h}{w_{h,h}} \geq 
 	n \beta^{ -3 h \alpha}/2 \geq  n  \beta^{ -11}/2 ,\\
 	 \frac{\|(B'A)^h B' \evec_j\|_1}
          	   {\|(BA)^h  B'\evec_j\|_\infty}  &\geq 
 	\frac{s_{\min}^h  t_{\min}^{h+1}}{w_{h,h}} \geq 
 	n \beta^{ -3 (h+1) \alpha}/2 \geq  n  \beta^{ -14}/2.
 \end{align*}
  Note also that 
          \begin{align*}
          	\|(BA)^h \evec_j\|_1 &=  (\varDelta(\tvec) \varDelta(S))^h  
          	\exp\left( O\left(
          	\dfrac{\range{S}}{\alpha s_{\min}} + \dfrac{\range{\tvec}}{\alpha t_{\min}}
          	\right)\right),\\
          		\|(B'A)^h B' \evec_j\|_1 &=  (\varDelta(\tvec))^{h+1} (\varDelta(S))^h  
          	\exp\left( O\left(
          	\dfrac{\range{S}}{\alpha s_{\min}} + \dfrac{\range{\tvec}}{\alpha t_{\min}}
          	\right)\right).
          \end{align*}
   Using also similar bounds for $\evec_k$ and the assumption on $\langle \xvec, \yvec \rangle_S$, we find that
      \begin{align*}
      		\langle(BA)^{\frac{\ell-1}{2}} \evec_j , 
	              (BA)^{(\frac{\ell-1}{2}} \evec_k\rangle_S &= 
	              \exp\left( O\left( \gamma + 
          	\dfrac{\range{S}}{\alpha s_{\min}} + \dfrac{\range{\tvec}}{\alpha t_{\min}}
          	 \right)\right) \varDelta(\tvec)^{\ell-1} \varDelta(S)^{\ell}/n,
          	 \\
          	 	\langle(B'A)^{\frac{\ell-3}{2}} B' \evec_j , 
	              (B'A)^{\frac{\ell-1}{2}} B'\evec_k\rangle_S 
	              &= 
	              \exp\left( O\left( \gamma + 
          	\dfrac{\range{S}}{\alpha  s_{\min}} + \dfrac{\range{\tvec}}{\alpha t_{\min}}
          \right)\right) \varDelta(\tvec)^{\ell} \varDelta(S)^{\ell-1}/n.
      \end{align*}     
      Next, we apply \eqref{bounds_wab} to bound 
       quantities $w_{a,b}$ and $w_{a,b}'$ that appear in the 
      lower bounds of Lemma \ref{l:FB}. Thus, we find that 
      \begin{align*}
      		\ell w_{\ell-1,\ell} &= O\Bigl(\dfrac{\varDelta(S)^{\ell-1} \varDelta(\tvec)^{\ell}}{\alpha n}\Bigr),
      		\qquad  \ell^2 \max_{\lfloor \ell/3\rfloor \leq a \leq \ell-2} \frac{w_{a,a} + w_{a,a-1} \varDelta(S)}
      	 { (\varDelta(S)\varDelta(\tvec))^a} = O\Bigl(\dfrac{\varDelta(S)}{\alpha^2 n\varDelta(\tvec)}\Bigr),
      	 \\
      	 \ell w_{\ell-2,\ell+1} &=  O\Bigl(\dfrac{\varDelta(S)^{\ell-2} \varDelta(\tvec)^{\ell+1}}{\alpha n}\Bigr),
      	 \qquad 
      	  \ell^2 \max_{\lfloor \ell/3\rfloor \leq a \leq \ell-2} \frac{w_{a,a}' + w_{a-1,a}' \varDelta(S)}
      	  { (\varDelta(S)\varDelta(\tvec))^a}  = O\Bigl(\dfrac{1}{\alpha^2 n}\Bigr).
      \end{align*}
  Applying Lemma \ref{l:FB}, we conclude that 
      \begin{align*}
      		\frac{|\mathcal{B}_\ell(T)|}{|\mathcal{F}_\ell(T')| } = 
      		\exp\left( O\left(\gamma  +
          	\dfrac{\range{S}}{\alpha  s_{\min}} + \dfrac{\range{\tvec}}{\alpha t_{\min}}
          	 + \dfrac{\varDelta(\tvec)}{\alpha \varDelta(S)}
          	 + \dfrac{1}{\alpha^2 \varDelta(\tvec) \varDelta(S)}\right) \right) 
          	 \frac{\varDelta(\tvec)}{\varDelta(S)}.
      \end{align*}
            Combining  \eqref{switch1} and \eqref{switch2} completes the proof.

%%%%%%%%%%%%%%%%%%%%%%%%%%%%%%%%%
%%%%%%%%%%%%%%%%%%%%%%%%%%%%%%%%%

\nicebreak 
 \section{Appendix}\label{sec:appendix}
 
 Here we prove or cite the technical lemmas that are used in the proofs.
 This section  is self-contained and does not rely on assumptions other than those stated.
 
 \subsection{Some properties of $\G(n,p)$}

In this section we establish asymptotic probability bounds as $n \to \infty$ for the random graph $G \sim \G(n,p)$ to satisfy certain properties needed in Section  \ref{S:co-sparse}.
 
 \begin{lemma}\label{A:spec}
 	Let $A_p$ be the adjacency matrix of $G\sim G(n,p)$  for some 
 	$1 \geq p \gg \log n/n$.  Assume 
 	$\eps = \eps(n)>0$  such that
 	$ \sqrt{\dfrac{\log n}{pn}}\ll \eps \ll 1$. Then 
 	\[
 		\Pr\left(\|A_p - pJ \|_2  \leq \eps p n)\right)
 		 =1 - e^{-\Omega(\eps ^2 p^2 n^2)},
 	\]
 	where $J$ denotes the $n\times n$ matrix with all entries equal $1$.
 \end{lemma}
 \begin{proof}
 	For $pn \ge (\log n)^2$, the assertion follows from \cite[Theorem 1.4]{Vu2007}
 	 and the concentration result \cite[Theorem 1.2]{Vu2007}. For smaller values of $p$,  we use the bound 
 	 for $\|A_p - \E A_p \|_2 $ of \cite[Corollary 3.3.]{BGBK} (which even has a better exponent in the probability estimate). Observing that  $pJ - \E A_p = p I$ has a negligible spectral norm completes the proof.
 \end{proof}

 Next, we prove that binomial random graphs are pseudorandom 
 in a strong sense.
 \begin{lemma}\label{A:z_con}
 	Let  $G \sim \G(n,p)$ for some $1\geq p \gg \log n/n$.  
 	Assume 
 	$\eps = \eps(n)>0$  such that
 	$ \sqrt{\dfrac{\log n}{pn}}\ll \eps \ll 1$.
 	Then, with probability $1- e^{-\Omega(\eps^2 p n^2)}$,  we have 
 	\[
 			\Big|\langle \xvec, \yvec \rangle_G  -  p \|\xvec\|_1  \|\yvec\|_1 \Big| \leq  \eps p \|\xvec\|_1  \|\yvec\|_1
 	\] 
  	uniformly for all  $\xvec, \yvec \in [0,1]^n$  with $\|\xvec\|_1,\|\yvec\|_1 = \Omega(n)$, where 
  	$\langle \xvec, \yvec \rangle_G$ is defined according to  \eqref{Z_def}.
 \end{lemma}
 \begin{proof}
 	For any
 	$\xvec, \yvec \in [0,1]^n$  with $\|\xvec\|_1,\|\yvec\|_1  = \Omega(n)$,  we have that
 	\[
 		\E \langle \xvec, \yvec \rangle_G= p \|\xvec\|_1 \|\yvec\|_1 -  p \sum_{j=1}^n x_j y_j = 
 		 \left(1+ O(n^{-1})\right)p \|\xvec\|_1 \|\yvec\|_1.
 	\]
 	Note that $\Var \langle \xvec, \yvec \rangle_G  \leq 2pn^2$.
 	 Using McDiarmid's inequality \cite[Theorem 2.7]{McDiarmid} (with $V=\Var \langle \xvec, \yvec \rangle_G$, $b=2$ and $t=\eps p \|\xvec\|_1  \|\yvec\|_1/2$), we  get
 	\begin{align*}
 		\Pr \left( \frac{|\langle \xvec, \yvec \rangle_G- \E \langle \xvec, \yvec \rangle_G|}{p \|\xvec\|_1  \|\yvec\|_1} \geq \eps /2  \right)
 		 &\leq  2\exp \left( - \frac{ (\eps p \|\xvec\|_1  \|\yvec\|_1/2)^2 }{ 2 \Var \langle \xvec, \yvec \rangle_G+ (2/3)  \eps p \|\xvec\|_1  \|\yvec\|_1  }\right)
 		\\ &= 		e^{-\Omega(\eps^2 p n^2)}.
 	\end{align*}	
 	To make the probability estimate hold for all such $\xvec$, $\yvec$, we approximate them 
 	with $\xvec',\yvec' \in \{ j/n \st j=1,\ldots,n\}^n$ such that 
 	$\|\xvec - \xvec'\|_\infty \leq n^{-1}$ and $\|\yvec - \yvec'\|_\infty \leq n^{-1}$. 
 	Denoting $\boldsymbol{1} = (1,\ldots,1)\trans $, we find that   
 	 \[ 
 	 \langle \xvec, \yvec \rangle_G 
 	  =   \langle \xvec', \yvec' \rangle_G  + 
 	  O(n^{-1})  \langle\xvec' , \boldsymbol{1}\rangle_G  + 
 	  O(n^{-1}) \langle \boldsymbol{1}, \yvec' \rangle_G   
 	  + 
 	  O(n^{-2}) \langle \boldsymbol{1}, \boldsymbol{1}\rangle_G   =   \langle \xvec', \yvec' \rangle_G +O(n),
 	 \]
	 as
	 $ \langle\xvec' , \boldsymbol{1}\rangle_G , \langle \boldsymbol{1}, \yvec' \rangle_G\le \langle \boldsymbol{1}, \boldsymbol{1}\rangle_G  \le n^2$.
 	Observe that $\eps p \|\xvec\|_1  \|\yvec\|_1  \gg n$ so all the error terms are within the required range.
 	Allowing $n^{2n}$ for choice of $\xvec',\yvec'$, using the union bound
 	and recalling that $n \log n \ll \eps^2 p n^2$, we complete the proof.
 \end{proof}	
    
%%%%%%%%%%%%%%%%%%%%%%%%%%%%%%%%%
%%%%%%%%%%%%%%%%%%%%%%%%%%%%%%%%%

\subsection{When common neighbours are not rare}
Here, we explore the properties of graphs  which any two vertices 
have sufficiently many common neighbours.

	\begin{lemma}\label{A:neighbours}
		Let $G$ be a graph on $n$ vertices and $\gamma>0$ be fixed.  Assume that 
		any two vertices have at least $\dfrac{\gamma \varDelta^2}{ n}$ common neighbours in $G$, where
	 $\varDelta = \varDelta(G)$.	Then the following hold.
		\begin{itemize}	
			\item[(a)] The minimal degree of $G$ is at least $\gamma \varDelta$.
			\item[(b)]  For any $\xvec \in \Reals^n$, we have 
			$\sum_{jk\in G} (x_j + x_k)^2 \geq \dfrac{\gamma^4}{256} \norm{\xvec}_2^2\,\varDelta.
			$
		\end{itemize}
	\end{lemma}
   \begin{proof}
   	For a vertex $j$ let's count its common neighbours with other vertices. Note that any vertex is counted at most $\varDelta-1$ times 
   	(since it is already connected to $j$). Therefore, the degree of $j$ is at least $\dfrac{(n-1)\gamma \varDelta^2}{  (\varDelta-1)n} \geq \gamma \varDelta$ which proves (a). For the rest of the argument, note that $\gamma\leq 1$.

Let $Q_G$ denote the matrix defined  by $\xvec\trans  Q_G \xvec = \sum_{jk\in G} (x_j + x_k)^2 $. 
The matrix $Q_G$ is known as  \textit{signless Laplacian matrix}.  From \cite[Theorem 3.2]{DR1991}, we find 
	that all eigenvalues of  $Q_G$ are bounded below by $\dfrac{\psi^2(S)}{4\varDelta}$, 
	where
\[
	\psi(G) = \min \left\{ \dfrac{\eps_b(G[U])+|\partial_G(U)|}{|U|} \,:\, \emptyset \neq U \subset V(G) \right\},
\]
where $G[U]$ denotes the induced subgraph and $\eps_b(G[U])$  is the minimal number of edges required to delete from the graph $G[U]$ to make it  bipartite.  Thus, to prove (b), it is sufficient to show $\psi(G) \geq \dfrac{\gamma^2}{8}\varDelta$. 

First, consider the case  $|U| \leq n(1 - \gamma/4)$. Observe that, for any common neighbour $\ell$ of two vertices  $j\in U$ and $k \notin U$,   either $j \ell $ or $j k$ contributes to $\partial_G U$. 
   	By the assumptions, the number of choices of $j,k$ and $\ell$ is at least $|U|(n-|U|)\dfrac{\gamma \varDelta^2}{ n}$. We need to divide by $2\varDelta$ to adjust over-counting. Thus, we get  
   	\[
   		\dfrac{|\partial_G U|}{|U|} \geq  \dfrac{(n-|U|) \gamma \varDelta^2}{ 2 n\varDelta } \geq 
   		   \dfrac{\gamma^2 }{8} \varDelta.
   	\]
Now, assume  $|U|> n (1 - \gamma/4)$.   Consider any partition $(W_1, W_2)$ of $U$ into two disjoints sets. We may assume 
$|W_1| \geq |W_2|$. If $|W|_2 \leq  \gamma n/4$  then,   bounding degrees of vertices in $W_1$ below by $\gamma \varDelta$ 
and degrees of vertices of $W_2$ above by $\varDelta$, we get that
\begin{align*}
	|\partial_G U| + |E(G[W_1])| &\geq \gamma \varDelta |W_1| - \varDelta |W_2|  \geq 
	\gamma  (1 - \gamma /2 - 1/4) \varDelta n   \\ &\geq 
	 \gamma  \dfrac{3/4 - \gamma/2}  {1- \gamma/4}  \varDelta  |U| \geq  \dfrac{ \gamma}{3} \varDelta |U| >
	 \dfrac{\gamma^2}{8} \varDelta |U|.
\end{align*}
If $|W|_2 >  \gamma n/4$, observe that, for any common neighbour $\ell$ of two vertices  $j\in W_1$ and $k \notin W_2$,    at least one of $\{j \ell,  k\ell\}$ contributes  to $E(G[W_1]) \cup E(G[W_2]) \cup \partial_{G}(U)$. 
By the assumptions, the number of choices of $j,k$ and $\ell$ is at least 
$
|W_1| |W_2|\dfrac{\gamma \varDelta^2}{ n} 
$
Dividing by $2\varDelta$ to adjust for over-counting, we get
\begin{align*}
	|E(G[W_1])| + |E(G[W_2])|  + |\partial_G(U)|&\geq 
	\frac{ |W_1|\, |W_2|\dfrac{\gamma \varDelta^2 }{ n} }
	{2 \varDelta} 
	\geq    (1-\gamma/2) \dfrac{\gamma^2}{4}  n\varDelta     
	 \\ 
	 &\geq  \frac{  (1-\gamma/2) \gamma^2 }{4 - \gamma}  \varDelta |U|
	 > \dfrac{\gamma^2 }{8}  \varDelta |U|.
\end{align*}
Combining above,  we get in any case that
$
	 \dfrac{\eps_b(G[U])+|\partial_G(U)|}{|U|} \geq   \dfrac{\gamma^2 }{8}  \varDelta.
$
Part (b) follows.
   \end{proof}
   
%%%%%%%%%%%%%%%%%%%%%%%%%%%%%%%%%
%%%%%%%%%%%%%%%%%%%%%%%%%%%%%%%%%

\subsection{Integration theorem}\label{S:mother}
Here, we quote the results from \cite{IM}that were used in Section \ref{s:enumeration}. 
For a domain $\varOmega \subseteq \Reals^n$ and a twice continuously differentiable function $q:\varOmega\to \Complexes$, define

\[
H(q,\varOmega)=(h_{jk}),\ \text{where}\  h_{jk}= 
\sup_{\xvec \in\varOmega }\left| \dfrac{\partial^2 q}{\partial x_j \partial x_k} (\xvec)\right|.
\]
 \begin{theorem}[Theorem 4.4 of \cite{IM}]\label{thm:integral}

  Let $c_1,c_2,c_3,\eps,\rho_1,\rho_2,\phi_1,\phi_2$ be
  nonnegative real constants with $c_1,\eps>0$.
 Let $Q$ be an $n\times n$ positive-definite symmetric real matrix
 and let $T$ be a real matrix such that $T\trans  QT=I$.
 Let $\varOmega$ be a measurable set such that
 $U_n(\rho_1)\subseteq T^{-1}(\varOmega)\subseteq U_n(\rho_2)$,
   and let 
   $f: \Reals^n\to\Complexes$ and
   $g: \Reals^n\to\Reals$ be twice continuously differentiable
   and let $h:\varOmega\to\Complexes$ be integrable.
 We make the following assumptions.
   \begin{itemize}\itemsep=0pt
     \item[(a)] $c_1(\log n)^{1/2+\eps}\le\rho_1\le\rho_2$.
     \item[(b)] For $\xvec\in T(U_n(\rho_1))$,
        $2\rho_1\,\norm{T}_1\,\left| \dfrac{\partial f} {\partial x_j}(\xvec) \right|
         \le \phi_1 n^{-1/3}\le\dfrac23$ for $1\le j\le n$ and\\
         $4\rho_1^2\,\norm{T}_1\,\norm{T}_\infty\,
         \norm{H(f,T(U_n(\rho_1)))}_\infty
         \le \phi_1 n^{-1/3}$.
     \item[(c)] For $\xvec\in\varOmega$, $\Re f(\xvec) \le g(\xvec)$.
        For $\xvec\in T(U_n(\rho_2))$, either\\
       (i) $2\rho_2\,\norm{T}_1\,\left| \dfrac{\partial g} {\partial x_j}(\xvec) \right|\le 
        (2\phi_2)^{3/2} n^{-1/2}$ for $1\le j\le n$, or\\
       (ii) $2\rho_2\,\norm{T}_1\,\left| \dfrac{\partial g} {\partial x_j}(\xvec) \right|
         \le \phi_2 n^{-1/3}$ for $1\le j\le n$ and
         \[4\rho_2^2\,\norm{T}_1\,\norm{T}_\infty\,
          \norm{H(g,T(U_n(\rho_2)))}_\infty
         \le  \phi_2 n^{-1/3}.\]
     \item[(d)] $\abs{f(\xvec)},\abs{g(\xvec)} \le n^{c_3} 
                    e^{c_2\xvec\trans  Q\xvec/n}$ for $\xvec\in\Reals^n$.
            \end{itemize}
  Let $\X$ be a random variable with the normal density
    $\pi^{-n/2} \abs{Q}^{1/2} e^{-\xvec\trans Q\xvec}$.
     Then, provided $\V f(\X) = \E (f (\X)- \E f(\X))^2$ and $\Var g(\X)$ are finite
     and $h$ is bounded in~$\varOmega$,
     \[
        \int_\varOmega e^{-\xvec\trans Q\xvec + f(\xvec)+h(\xvec)}\,d\xvec
        = (1+K) \pi^{n/2}\abs{Q}^{-1/2} e^{\E f(\X)+\frac12 \E \left(f(\X) - \E f(\X)\right)^2},
     \]     
   where, for some constant $C$ depending only on $c_1,c_2,c_3,\eps$,
   \begin{align*}
      \abs{K} &\le C e^{\frac12\Var\Im f(\X)}\,\Bigl( e^{\phi_1^3+e^{-\rho_1^2/2}}-1
        \\
      &{\qquad}+ \(2e^{\phi_2^3+e^{-\rho_1^2/2}}-2
        + \sup_{\xvec\in\varOmega}\,\abs{e^{h(\xvec)}-1} \)\,
           e^{\E(g(\X)-\Re f(\X))+\frac12(\Var g(\X)-\Var\Re f(\X))}\Bigr).
   \end{align*}
   In particular, if $n\ge (1+2c_2)^2$ and
   $\rho_1^2 \ge 15 + 4c_2 + (3+8c_3)\log n$, 
    we can take~$C=1$.
\end{theorem}

In order to apply Theorem~\ref{thm:integral}, we need to verify that $T$ exists and satisfies all required conditions.   The following lemma is a special case of~\cite[Lemma 4.9]{IM} (for trivial $\ker Q$ and $\gamma = {\mumin}/\dmax$). Recall that $\|\cdot\|_{\max}$ stands  for the maximum of the absolute values of the elements of a given matrix.
 \begin{lemma}\label{matrixthm1}
  Let $Q$ be an $n\times n$ real symmetric
  matrix with positive minimum eigenvalue $\mumin$.
 Let $D$ be a diagonal matrix  such that $\maxnorm{Q-D} \le r\dmin/n$
 for some~$r$.  Assume the diagonal entries of $D$ are in $[\dmin,\dmax]$ for some  $\dmax \geq \dmin>0$, then
\begin{itemize}[topsep=1ex]
   \item[(a)]  
     $\displaystyle\maxnorm{Q^{-1}-D^{-1}}\le 
      \frac{r(r  \dmax+\mumin)}{\mumin\dmin n}$.
\end{itemize}
  Furthermore, there exists a real matrix $T$ such that
     $T\trans  Q T=I$ and
\begin{itemize}[topsep=1ex,itemsep=0.5ex]
     \item[(b)]
        $\displaystyle\norm{T}_1,\norm{T}_\infty
        \le \frac{r \dmax^{1/2}+\mumin^{1/2}}{\mu_{\min}^{1/2} \dmin^{1/2}}$.
     \item[(c)]  $\displaystyle \norm{T^{-1}}_1, \norm{T^{-1}}_\infty
        \le \frac{(r+1)(r \dmax+\mumin^{1/2} \dmax^{1/2})}{\mumin^{1/2}}$.
\end{itemize}
\end{lemma}

%%%%%%%%%%%%%%%%%%%%%%%%%%%%%%%%%
%%%%%%%%%%%%%%%%%%%%%%%%%%%%%%%%%
%%%%%%%%%%%%%%%%%%%%%%%%%%%%%%%%%

\subsection{Weighted graphs and norm bounds}
In the case when the matrix has a specific graph-related  structure, 
the bounds of Lemma \ref{matrixthm1} can be improved. 
For a graph $G$ on $n$ vertices and  weights $W=(w_{jk})$, define the symmetric matrix 
$Q_W$ by
\begin{equation}\label{def_AW}
	\xvec\trans  Q_W \xvec = \sum_{jk \in G} w_{jk} (x_j + x_k)^2.
\end{equation}
Observe that if $w_{jj}=0$  for all $j$ then $D=Q_W - W$ is the diagonal matrix with the same diagonal elements as in $Q_W$.
  \begin{lemma}\label{A:weighted}
  	Let $G$ be a graph on $n$ vertices. Assume that  $\varDelta = \varDelta(G) = \Omega(n^{1/2})$
  	and  the number of common neighbours of any two vertices in  $S$
	  		  	is  $\Theta\left(\dfrac{\varDelta^2}{n}\right)$.	 Take  $n\times n$  matrix $W=(w_{jk})$ with positive real entries such
	  		  	 that $w_{jk} = \Theta(1)$ if $jk \in G$ and $w_{jk}=0$ otherwise. Then the following hold.
	  	\begin{itemize}
	  	\item[(a)] The diagonal elements of $Q_W$ are $\Theta (\varDelta)$. 
	  	\item[(b)] If  $Q_W^{-1} = (\sigma_{jk})$ then 
	  $
	  		\sigma_{jk} = 
	  		\begin{cases}
	  			 \Theta\left(\dfrac{1 }{\varDelta}\right), & \text{if } j=k;\\[1ex]
	  			 O\left(\dfrac{1 }{\varDelta^2 }\right), & \text{if } jk \in G;\\[1ex]
	  			  O\left(\dfrac{1 }{\varDelta n}\right), & \text{otherwise}.
	  		\end{cases}
	  	$
	  	\item[(c)] There exists a real matrix $T$ such that
     $T\trans  Q_W T=I$ and
     \[
         \|T\|_1, \|T\|_\infty =   O(\varDelta^{-1/2}), \qquad 
  	\|T^{-1}\|_1, \|T^{-1}\|_\infty =   O(\varDelta^{1/2}).
     \]
     \item[(d)] Let $G'$ be the graph obtained by deleting vertex $1$ from $G$
     and $W'$ be formed by deleting one row and one column from  $W.$  Define $Q_W'$  to be the matrix of
      \eqref{def_AW} for $G'$ and $W'$. Then $ |Q_W| = O(\varDelta)|Q_W'|$. 
	  	\end{itemize}
  \end{lemma} 
\begin{proof}
	In Lemma \ref{A:neighbours}(a) we prove that all degrees of $G$ are $\Theta(\varDelta)$.
Thus, the diagonal  elements of  $Q_W$ are $\Theta(\varDelta)$. From Lemma \ref{A:neighbours}(b), we find that for any non-trivial $\xvec \in \Reals^n$ 
\begin{equation}\label{eigen_A}
	\frac{\xvec\trans  Q_W \xvec}{\|\xvec\|_2^2} =\Theta(1) \frac{\sum_{jk\in S} (x_j  +x_k)^2} {\|\xvec\|_2^2} = \Omega(\varDelta). 
\end{equation}
 Therefore, the eigenvalues of $Q_W$ are  $\Theta(\varDelta)$.
  Let 
 \[
 	\tilde{Q} = (I - \tfrac12  W  D^{-1}) Q_W  (I - \tfrac12 D^{-1} W) = 
 	D -  \tfrac 34 W D^{-1}W + \tfrac14 W D^{-1} W D^{-1} W.
 \]
 Using the upper bound on the number of common neighbours in $G$, we find that 
 the off-diagonal elements of  $W D^{-1}W$  are $O\left(\dfrac{\varDelta}{n}\right)$,
  while its diagonal elements  are $O\left(1 \right)$. 
  Then all elements of   $W D^{-1} W D^{-1} W$  are
   $O\left( \dfrac{\varDelta}{n} +  \dfrac{1}{\varDelta} \right) = 
   O\left( \dfrac{\varDelta}{n} \right)$. 
   Then we get that 
    \[
     \|\tilde{Q} - \tilde{D}\|_{\max} = O\left(\dfrac{\varDelta}{n}\right) \qquad \text{and} \qquad \|\tilde{D} - D\|_{\max} = O(1),
    \]
    where $\tilde{D}$ is  the diagonal matrix with the same diagonal as $\tilde{Q}$.
     Next, observe, 
   \begin{equation}\label{eq:infty}
      \begin{aligned}
   	   \|D -  \tfrac12 W\|_\infty, \|D -  \tfrac12 W\|_1 &= O(\varDelta), \\
   	    \|(D -  \tfrac12 W)^{-1}\|_\infty, \|(D -  \tfrac12 W)^{-1}\|_1 &= O\left(\dfrac{1 }{\varDelta}\right)
   	    \end{aligned}
   \end{equation}
  Using part (a) and recalling \eqref{eigen_A}, 
  we find that  all eigenvalues of $\tilde{Q}$  are $\Theta(\varDelta)$. 
  Thus, we can apply  Lemma \ref{matrixthm1}(a) to matrix $\tilde{Q}$ to obtain that
  \[
  		\|\tilde{Q}^{-1} - \tilde{D}^{-1}\|_{\max} = O\left(\dfrac{1 }{\varDelta n}\right).
  \]
  Observe that $\|\tilde{D}^{-1} - D^{-1}\|_{\max} = O\left(\dfrac{1 }{\varDelta^2}\right)$ and
 \[
 	Q_W^{-1} =  D^{-1} (D - \tfrac12    W)  \tilde{Q}^{-1}  (D - \tfrac12  W  ) D^{-1}. 
 \]	
Since  $\|XY\|_{\max} \leq \|X\|_\infty \|Y\|_{\max}$ and  $\|XY\|_{\max} \leq \|X\|_{\max} \|Y\|_1$, 
we get from \eqref{eq:infty}
 \[
 	 \| (I - \tfrac12   D^{-1} W)  (\tilde{Q}^{-1} - \tilde{D}^{-1})  (I - \tfrac12  W  D^{-1})\|_{\max} = O\left(\dfrac{1 }{\varDelta n}\right).
 \] 
 Arguing as before to bound entries of $D^{-1}W \tilde{D}^{-1} W D^{-1}$, the part (b) follows.
 
  From  Lemma \ref{matrixthm1}(b,c), we find   a real matrix  $\tilde{T}$ such  that $\tilde{T}\trans  \tilde{Q} \tilde T = I$ and 
  \[
  	\|\tilde{T}\|_1, \|\tilde{T}\|_\infty =   O(\varDelta^{-1/2}), \qquad 
  	\|\tilde{T}^{-1}\|_1, \|\tilde{T}^{-1}\|_\infty =   O(\varDelta^{1/2}).
  \]
 Taking  $T =D^{-1}(W - \tfrac12   W) \tilde{T}$ and using \eqref{eq:infty},  we prove (c).

 For (d), we write the matrix $Q_W=(q_{jk})$ as follows.
	 \[
      Q_W = \begin{pmatrix}
               q_{11} & \qvec\trans  \\
               \qvec & Q_W' + \text{diag}(\qvec) 
             \end{pmatrix},
 \]
where $\qvec=(q_{12},\ldots,q_{1n})$ and
$\text{diag}(\qvec)$ is a diagonal matrix with the elements
of $\qvec$ down the diagonal.
Now perform the first step of Gaussian elimination by 
subtracting multiples of the first row from the other rows.
The result is
\[
   \begin{pmatrix}
               q_{11} & \qvec\trans  \\
       \boldsymbol{0} & Q_W' + \text{diag}(\qvec) - q_{11}^{-1}\avec\avec\trans  
   \end{pmatrix},
\]
Consequently,
$\abs {Q_W} = q_{11} \abs{Q_W'+ \text{diag}(\qvec) - q_{11}^{-1}\qvec\qvec\trans }
= q_{11} \abs{Q_W'}\, \abs{I+B}$ where
\[ B=(Q_W')^{-1}(\text{diag}(\qvec) - q_{11}^{-1}\qvec\qvec\trans ).\]
Observe that 
$G'$ and $W'$ satisfy all assumptions of Lemma \eqref{A:weighted}. Then all eigenvalues of $B$ are real since $Q_W'$ is positive definite symmetric by  \eqref{eigen_A} and   $(\text{diag}(\qvec) - q_{11}^{-1}\qvec\qvec\trans )$ is symmetric.
Consequently,
\[\abs{I+B} = \exp(\tr B + O(\tr B^2))\]
Let $B = (b_{jk})$. From $(b)$, we find that $b_{jk} = O\left(\dfrac{1}{\varDelta}\right)$. Observing also that   
$b_{jk} = 0$ for all  $1k \notin S$, we get that 
\[
	\tr B = \sum_{j:\, 1j\in S} b_{jj} = O(1), \qquad \tr B^2 = \sum_{jk :\, 1j \in S, \,1k\in S} b_{jk} b_{kj} = O(1). 
\]
By (a), we have  $q_{11} = \Theta(\varDelta)$, therefore $\abs{Q_W} = O(\varDelta) \abs{Q_W'}$.
\end{proof}
%%%%%%%%%%%%%%%%%%%%%%%%%%%%%%%%%
%%%%%%%%%%%%%%%%%%%%%%%%%%%%%%%%%

\nicebreak
  
 %%%%%%%%%%%%%%%%%%%%%%%%%%%%%%%%%%%%%%%%%%%%%%%%%%%%%%%%%%%%%%%
 
\end{document}